\documentclass[10pt,reqno]{amsart}
\usepackage{graphicx}
\usepackage{epstopdf}
\usepackage{pdfpages}
\usepackage{epsfig}
\usepackage{color}

\newtheorem{theorem}{Theorem}[section]
\newtheorem{lemma}[theorem]{Lemma}
\newtheorem{proposition}[theorem]{Proposition}

\theoremstyle{definition}

\theoremstyle{remark}

\newcommand{\fe}{\mathrm{e}}

\newcommand{\eps}{\varepsilon}
\newcommand{\bR}{{\mathbb R}}
\newcommand{\bT}{{\mathbb T}}
\newcommand{\bN}{{\mathbb N}}
\newcommand{\bZ}{{\mathbb Z}}
\newcommand{\bU}{{\bf U}}

\newcommand{\bv}{{\bf v}}

\newcommand{\bx}{\mathbf{x}}
\newcommand{\by}{{\mathbf y}}
\newcommand{\bh}{{\mathbf h}}
\newcommand{\br}{{\mathbf r}}

\newcommand{\bE}{\mathbf{E}}
\newcommand{\be}{\mathbf{e}}
\numberwithin{equation}{section}

\begin{document}

\title[Vlasov equations with non-homogeneous magnetic field]{Uniformly accurate methods for Vlasov equations with non-homogeneous strong magnetic field}

%    Remove any unused author tags.

%    author one information
\author[Ph. Chartier]{Philippe Chartier}
\address{\hspace*{-12pt}Ph.~Chartier: Univ Rennes, CNRS, IRMAR-UMR 6625, F-35000 Rennes, France}
\email{Philippe.Chartier@inria.fr}

%    author two information
\author[N. Crouseilles]{Nicolas Crouseilles}
\address{\hspace*{-12pt}N.~Crouseilles: Univ Rennes, CNRS, IRMAR-UMR 6625, F-35000 Rennes, France}
\email{nicolas.crouseilles@inria.fr}

\author[M. Lemou]{Mohammed Lemou}
\address{\hspace*{-12pt}M.~Lemou: Univ Rennes, CNRS, IRMAR-UMR 6625, F-35000 Rennes, France}
\email{mohammed.lemou@univ-rennes1.fr}

\author[F. M\'{e}hats]{Florian M\'{e}hats}
\address{\hspace*{-12pt}F.~M\'{e}hats: Univ Rennes, CNRS, IRMAR-UMR 6625, F-35000 Rennes, France}
\email{florian.mehats@univ-rennes1.fr}

\author[X. Zhao]{Xiaofei Zhao}
\address{\hspace*{-12pt}X.~Zhao: Univ Rennes, CNRS, IRMAR-UMR 6625, F-35000 Rennes, France}
\email{zhxfnus@gmail.com}

\subjclass[2010]{Primary }


\date{}

\dedicatory{}

\begin{abstract}
In this paper, we consider the numerical solution of highly-oscillatory Vlasov and Vlasov-Poisson equations with non-homogeneous magnetic field. Designed in the spirit of recent {\em uniformly accurate} methods, our schemes remain insensitive to the stiffness of the problem, in terms of both accuracy and computational cost. The specific difficulty (and the resulting novelty of our approach) stems from the presence of a non-periodic oscillation, which necessitates a careful ad-hoc reformulation of the equations. Our results are illustrated numerically on several examples. 
\end{abstract}

\maketitle

\section{Introduction}
In this article, we are concerned with the numerical solution of the $2$-dimensional Vlasov equation with non-homogeneous magnetic field \cite{bostan,bostan_finot,degond,filbet,lee}. More specifically, if $b: x \in \bR^2 \mapsto b(x) \in \bR$ and $f_0: (\bx, \bv) \in \bR^2 \times \bR^2 \mapsto f_0(\bx, \bv) \in \bR$ are given functions, and if $0<\eps\leq1$ denotes a dimensionless parameter and $]0,T]$ a non-empty interval of time, we shall consider the Cauchy problem for the distribution function $f^\eps: (t,\bx, \bv) \in [0,T] \times \bR^2 \times \bR^2 \mapsto f^\eps(t,\bx, \bv) \in \bR$ given by
\begin{subequations}\label{eq:1}
\begin{align}
&\partial_t f^\eps(t,\bx, \bv) +\bv\cdot\nabla_\bx f^\eps(t,\bx, \bv)+\left(\bE^\eps(t,\bx)+\frac{b(\bx)}{\eps}J\bv\right)\cdot\nabla_\bv f^\eps(t,\bx, \bv)=0,\label{eq:vp1}\\
&f^\eps(0,\bx,\bv)=f_0(\bx,\bv),
\end{align}
\end{subequations}
where
\begin{enumerate}
\item[(i)] the unidirectional magnetic field $B: x \in \bR^2 \mapsto B(\bx)= (0, 0, b(\bx)) \in \bR^3$ induces a Lorentz force $\bv\times B(\bx)$ which in the two-dimensional context simply becomes $b(\bx)J\bv$ with
\begin{equation}
\label{Jmatrix}
J=
\begin{pmatrix}
0 & 1 \\
-1&0
\end{pmatrix};
\end{equation}
\item[(ii)] the electric-field function $\bE^\eps: (t,\bx) \in \bR^+\times\bR^2 \mapsto \bE^\eps(t,\bx) \in \bR^2$ is either external and explicitly given or self-consistent. In the latter case, $\bE^\eps$ solves the Poisson equation
\begin{equation}
\label{poisson}
\nabla_\bx\cdot \bE^\eps(t,\bx)=\int_{\bR^2}f^\eps(t,\bx,\bv)d\bv-n_i(\bx),
\end{equation}
where $n_i$ denotes the ion density of the background.
\end{enumerate}
\medskip

Solving equation \eqref{eq:1} with standard methods is notoriously difficult for vanishing values of the parameter $\eps$, as the Lorentz term then creates high-oscillations  in the solution: this indeed imposes to use tiny time-steps (usually of the order of $\eps$) and leads to formidable computational costs. Hence, it is now admitted that specific techniques are required that can cope with this particular regime of small values of $\eps$ and which, as logic dictates, preserve the asymptotics of $f^\eps$ in the limit where $\eps$ goes to zero. Numerical methods obeying to this paradigm  (i.e. consistent with the limit equation for $f^0$)
and that are consistent with \eqref{eq:1}
when $\eps=O(1)$ have been called {\em asymptotic preserving} methods and may be found in various publications \cite{Filbet_rod,Filbet_rod2,Frenod,Sonnendrucker2,Shi}.

Nevertheless, if the value of $\eps$ is not known prior to the simulation, it is often observed that the error behaviour of asymptotic-preserving methods is largely deteriorated for certain (not so small) values of  $\eps$. As a consequence, it appears highly desirable to design numerical methods for  \eqref{eq:1} which are {\em uniformly accurate (UA)}
with respect to the parameter $\eps\in ]0, 1]$. That is to say, $p^{th}$-order methods which, when used with time step $\Delta t$,  deliver approximate solutions $f^\eps_{\Delta t}$  such that
$$
\|f^\eps_{\Delta t} - f^\eps\| \leq C (\Delta t)^p
$$
in an appropriate function-norm, where the constant $C$ as well as the computational cost are independent of $\eps$.

It is precisely the aim of this paper to introduce UA schemes for equation \eqref{eq:1}, which,  as we shall illustrate numerically, are indeed able to capture the various scales occurring in the system while keeping numerical parameters (in particular the time step)
independent of the degree of stiffness $\eps$. Although alternative options are possible \cite{NES,SNM,NUA,MRM,SHO}, the strategy we develop to reach this goal is very much inspired by the recent papers \cite{UAKG, UAVP4d}:
its main underlying idea consists in  separating explicitly
the two time scales naturally present in \eqref{eq:1}, namely the slow time $t$ and the fast time $t/\eps$. This is done at the level of the characteristic equations (resulting from the use of the Particle-In-Cell method, see e.g. \cite{Birdsall,SonnendruckerBook,Hesthaven}), which are, for each macro-particles,
stiff ordinary differential equations of the form
\begin{align} \label{eq:chareq}
\left\{
\begin{array}{cclc}
\dot{\bx} &=& \bv, & \bx(0) = \bx_0  \\
\dot{\bv} &=& \frac{b(\bx)}{\eps} J \bv + \bE^\eps(t,\bx), & \bv(0) = \bv_0
\end{array}
\right.
\end{align}
where the term $\frac{b(\bx)}{\eps} J \bv$ is the source of high-oscillations and at the origin of numerical difficulties.  As compared to our previous works \cite{UAKG,NUA},  the main obstacle  we are confronted with (and accordingly the main novelty of the proposed solution) is the fact that the aforementioned oscillations are {\bf not per se  time-periodic} (see also \cite{AAT}). However, we will show that the trajectory in the physical space $\bx(t)$ remains confined within an $\eps$-neighbourhood of the initial condition $\bx_0$, allowing to regard $e^{\frac{b(\bx_0)}{\eps} t J}$ as the principal  oscillation occurring  in the solution.  {\em Filtering} it out and {\em rescaling} the time according to $s=b(\bx_0) t$, we obtain
\begin{eqnarray} \label{eq:chareqf}
\left\{
\begin{array}{cclc}
\dot{\tilde{\bx}} &=& \frac{1}{b(\bx_0)} \, \fe^{\frac{s}{\eps} J} \,  \tilde{\by}(s), & \tilde{\bx}(0) = \bx_0  \\
\dot{\tilde{\by}} &=&\frac{1}{\eps} \, \left(\frac{b(\tilde{\bx})}{b(\bx_0)}-1\right) \, J\tilde{\by}+
\frac{1}{b(\bx_0)} \, \fe^{-\frac{s}{\eps} J} \, \bE^\eps(\frac{s}{b(\bx_0)},\tilde{\bx}), & \tilde{\by}(0) = \bv_0
\end{array}
\right.,
\end{eqnarray}
a system in a form which is now amenable to the {\em embedding} of $\tilde{\bx}(s)$ and $\tilde{\by}(s)$ into the functions $X(s,\tau)$ and $Y(s,\tau)$ periodic with respect to $\tau$ and
such that $X(s,s/\eps)=\tilde{\bx}(s)$ and $Y(s,s/\eps) = \tilde{\by}(s)$. The resulting transport equations
\begin{eqnarray} \label{eq:chareqds}
\left\{
\begin{array}{cclc}
\partial_s X + \frac{1}{\eps} \partial_\tau X &=& \frac{1}{b(\bx_0)} \, \fe^{\tau J} \,  Y, & X(0,0) = \bx_0  \\
\partial_s Y + \frac{1}{\eps} \partial_\tau Y &=&\frac{1}{\eps} \, \left(\frac{b(X)}{b(\bx_0)}-1\right) \, J Y+
\frac{1}{b(\bx_0)} \, \fe^{-\tau J} \, \bE^\eps(\frac{s}{b(\bx_0)}, X), & Y(0,0) = \bv_0
\end{array}
\right.
\end{eqnarray}
then need to be complemented with initial conditions $X(0,\tau)$ and $Y(0,\tau)$. Their choice is the fundamental ingredient of the two-scale strategy proposed in \cite{UAKG,APVP2d} and it requires to be handled here with additional care owing to the presence of the $1/\eps$-term in the right-hand side of  \eqref{eq:chareqds}. Under this form, the problem shares
 similarities with the model analyzed in \cite{UAKG}. However, \eqref{eq:chareqds} contains
 two main additional difficulties due to the presence of the term $(b(X)/b(\bx_0) -1)JY$: first, this nonlinear term prevents from
 a direct application of Gronwall lemma; second, this term is not smooth with respect to the unknown.
 Finally, we will see that the numerical solution also enjoys at a discrete confinement property, i.e.
 it remains confined within an $\eps$-neighbourhood of the initial data.

The organisation of the paper follows closely the steps exposed above. The two-scale formulation of the characteristics is introduced in Section \ref{sect:ts}, which includes in particular Subsection \ref{sect:sf} devoted to the scaling and filtering operations  and Subsection \ref{sect:ic} devoted to the detailed derivation of the initial conditions of \eqref{eq:chareqds}. The section concludes with rigorous estimates of the derivatives of $X$ and $Y$ (see Subsection \ref{sect:re}) and a numerical confirmation of the expected smoothness of the solution is brought in Subsection \ref{sect:ni}. Section \ref{sect:nm} is concerned with the effective derivation of numerical schemes of first and second orders for solving equations \eqref{eq:chareqds} and the proof of their convergence (Theorem \ref{thm1}) which constitutes the main result of this paper. Within Subsection \ref{vpcase} the adaptation of our numerical strategy to the situation of a coupling of Vlasov equation with Poisson equation is considered: although no rigorous statement is established at this stage, we provide empirical evidence of the efficiency of our method in this case. Section \ref{sect:ne} describes several examples and the corresponding numerical experiments, confirming the interest of the technique.

\section{Two-scale formulation of the characteristics equations} \label{sect:ts}
Using the Particle-In-Cell (PIC) discretisation,
\begin{equation}\label{dirac}
f^\eps(t,\bx,\bv)\approx\sum_{k=1}^{N_p}\omega_k\delta(\bx-\bx_k(t))\delta(\bv-\bv_k(t)),\quad t\geq0,\ \bx,\bv\in\bR^2,
\end{equation}
we get the characteristic equation for $1\leq k\leq N_p$,
\begin{subequations}
\label{charact}
\begin{align}
   & \dot{\bx}_k(t)=\bv_k(t), \\
  &\dot{\bv}_k(t)=\bE^\eps(t,\bx_k(t))+\frac{b(\bx_k(t))}{\eps}\bv_k^\bot(t),\quad t>0, \\
   & \bx_k(0)=\bx_{k,0},\quad \bv_k(0)=\bv_{k,0}.
\end{align}
\end{subequations}
We see (\ref{charact}) is a solution dependent highly oscillatory problem. As a basic requirement throughout the paper, we consider the magnetic field function $b(\bx)$ is uniformly above zero, i.e. for some constant $c_0>0$,
\begin{equation}\label{positivity}
b(\bx)\geq c_0>0,\quad \forall \bx\in\bR^2.
\end{equation}

\subsection{Scaling of time and filtering} \label{sect:sf}
For each $k$, introduce the scaled time
\begin{equation}\label{scale time}
s_k={b_k} t,\quad b_k=b(\bx_k(0)),
\end{equation}
and define
$$\tilde{\bx}_k(s):=\bx_k(t),\quad \tilde{\bv}_k(s):=\bv_k(t),\quad s\geq0,$$
where we omit the subscript $k$ in $s_k$ for simplicity of notations.  Note that under the assumption (\ref{positivity}), $s_k=s_k(t)$ is a monotone increasing function, which is interpreted as a time for particle $k$.  Through this transformation, we make each particle  living in its own time.
Then we can rewrite the characteristics equation (\ref{charact}) as
\begin{subequations}\label{charact2}
\begin{align}
   & \dot{\tilde{\bx}}_k(s)=\frac{\tilde{\bv}_k(s)}{b_k}, \\
  &\dot{\tilde{\bv}}_k(s)=\frac{b(\tilde{\bx}_k(s))}{\eps b_k}J\tilde{\bv}_k(s)+\frac{\bE^\eps(s/b_k,\tilde{\bx}_k(s))}{b_k},\quad s>0, \\
   & \tilde{\bx}_k(0)=\bx_{k,0},\quad \tilde{\bv}_k(0)=\bv_{k,0}.
\end{align}
\end{subequations}
Next, we isolate the main oscillation term,
$$\dot{\tilde{\bv}}_k(s)=\frac{1}{\eps }J\tilde{\bv}_k(s)+\left(\frac{b(\tilde{\bx}_k)}{b_k}-1\right)\frac{J\tilde{\bv}_k(s)}{\eps}+\frac{\bE^\eps(s/b_k,\tilde{\bx}_k(s))}{b_k},\quad s>0,$$
and then filter it out by introducing
\begin{equation}\label{filter}
\tilde{\by}_k(s):=\fe^{-Js/\eps}\tilde{\bv}_k(s),\qquad s\geq0.
\end{equation}
Then (\ref{charact2}) becomes
\begin{subequations}\label{charact3}
\begin{align}
&\dot{\tilde{\bx}}_k(s)=\fe^{Js/\eps}\frac{\tilde{\by}_k(s)}{b_k}, \label{charact3a}\\
&\dot{\tilde{\by}}_k(s)=\left(\frac{b(\tilde{\bx}_k(s))}{b_k}-1\right)\frac{J\tilde{\by}_k(s)}{\eps}+
\fe^{-Js/\eps}\frac{\bE^\eps(s/b_k,\tilde{\bx}_k(s))}{b_k},\quad s>0, \label{charact3b}\\
   & \tilde{\bx}_k(0)=\bx_{k,0},\quad \tilde{\by}_k(0)=\bv_{k,0}.
\end{align}
\end{subequations}

We shall analyse and solve (\ref{charact3}) up to any fixed time $$S_k= T b_k>0.$$ For technical reasons, hereafter we shall assume that  the given electric field $\bE^\eps(t,\bx)$ and the magnetic field $b(\bx)$ are globally Lipschitz functions, i.e.
\begin{align}
&|\bE^\eps(t,\bx_1)-\bE^\eps(t,\bx_2)|\leq C_\bE|\bx_1-\bx_2|,\quad \forall \bx_1,\bx_2\in\bR^2,\ 0\leq t\leq T,\nonumber\\
&|b(\bx_1)-b(\bx_2)|\leq C_b|\bx_1-\bx_2|,\quad \forall \bx_1,\bx_2\in\bR^2,\label{Lips}
\end{align}
for two constants $C_\bE,C_b>0$ independent of $\eps$. Here and after,
the norm $|\cdot|$ of a vector always refers to the standard euclidian norm in $\mathbb{R}^2$, whereas it refers to the absolute value
when it is applied to a scalar quantity.

\begin{lemma}\label{lemma1}
Under assumption (\ref{Lips})  and $\bE^\eps(t,\bx)\in C([0,T],\bR^2),\,b(\bx)\in {\mathcal C}^1(\bR^2)$, for the solution of (\ref{charact3}), we have that for each $1\leq k\leq N_p$,
\begin{equation}
|\tilde{\bx}_k(s)-\bx_{k,0}|\leq C_1\eps, \quad |\tilde{\by}_k(s)|\leq C_2,\quad 0\leq s\leq S_k,
\end{equation}
for some $C_1,C_2>0$. Hence,
\begin{equation}\label{lemma result}
\frac{1}{\eps}\left|\frac{b(\tilde{\bx}_k(s))}{b_k}-1\right|\leq C_3,\quad 0\leq s\leq S_k,
\end{equation}
for some $C_3>0$. The constants $C_1, C_2, C_3$ depend on $k$ but are independent of $\varepsilon$.
\end{lemma}

\begin{proof}
Based on the assumption for $E^\eps(t,\bx)$ and $b(\bx)$, we have the global well-posedness of (\ref{charact3}) and $\tilde{\bx}_k(s),\tilde{\by}_k(s)\in C(\bR^+)$. Indeed, taking the inner product on both sides of (\ref{charact3a}) and (\ref{charact3b}) with $\tilde{\bx}_k(s)$ and $\tilde{\by}_k(s)$ and applying Cauchy-Schwarz inequality respectively gives
\begin{align*}
&\frac{d}{ds}|\tilde{\bx}_k(s)|^2\leq \frac{2}{b_k}|\tilde{\bx}_k(s)| \; |\tilde{\by}_k(s)|,\quad \frac{d}{ds}|\tilde{\by}_k(s)|^2\leq \frac{2}{b_k}|\bE^\eps(s/b_k,\tilde{\bx}_k(s))| \; |\tilde{\by}_k(s)|.
\end{align*}
Hence we find
\begin{align*}
\frac{d}{ds}|\tilde{\bx}_k(s)|\leq &\frac{2}{b_k}|\tilde{\by}_k(s)|,\\
 \frac{d}{ds}|\tilde{\by}_k(s)|\leq &\frac{2}{b_k}|\bE^\eps(s/b_k,\tilde{\bx}_k(s))|
\leq\frac{2}{b_k}\left(|\bE^\eps(s/b_k,\bx_{k,0})|+C_\bE|\tilde{\bx}_k(s)-\bx_{k,0}|\right)\\
\leq&\frac{2}{b_k}\|\bE^\eps(\cdot,\bx_{k,0})\|_{L^\infty(0,T)}+\frac{2C_\bE}{b_k}|\bx_{k,0}|
+\frac{2C_\bE}{b_k}|\tilde{\bx}_k(s)|.
\end{align*}
Then, adding the two last inequalities and by Gronwall's inequality,
one can get an a priori estimate for boundedness of the solution, i.e.
$$|\tilde{\bx}_k(s)|+|\tilde{\by}_k(s)|\leq C_2,\quad 0\leq s\leq S_k,$$
where
$C_2=(2/b_k) \left(\|\bE^\eps(\cdot,\bx_{k,0})\|_{L^\infty(0,T)}+C_\bE|\bx_{k,0}|\right)\exp(2T\max\{1,C_\bE\})$.

By applying the Duhamel's principle to (\ref{charact3a}) and then integrating by parts, we have
\begin{align*}
\tilde{\bx}_k(s)&=\tilde{\bx}_k(0)+\int_0^s\fe^{J\theta/\eps}\frac{\tilde{\by}_k(\theta)}{b_k}d\theta\\
&=\tilde{\bx}_k(0)-\frac{J\eps}{b_k}\left(\fe^{Js/\eps}\tilde{\by}_k(s)-\tilde{\by}_k(0)\right)+
\eps J\int_0^s\fe^{J\theta/\eps}\frac{\dot{\tilde{\by}}_k(\theta)}{b_k}d\theta.
\end{align*}
Hence from (\ref{charact3b}), we have
\begin{align*}
\tilde{\bx}_k(s)=&\tilde{\bx}_k(0)-\frac{J\eps}{b_k}\left(\fe^{Js/\eps}\tilde{\by}_k(s)-\tilde{\by}_k(0)\right)\\
&+
\eps J\int_0^s\frac{\fe^{J\theta/\eps}}{b_k}\left[\left(\frac{b(\tilde{\bx}_k(\theta))}{b_k}-1\right)\frac{J\tilde{\by}_k(\theta)}{\eps}+
\fe^{-J\theta/\eps}\frac{\bE^\eps(\theta/b_k,\tilde{\bx}_k(\theta))}{b_k}\right]d\theta.
\end{align*}
Then we can see for all $0\leq s\leq S_k$,
\begin{align*}
\frac{1}{\eps}\left|\tilde{\bx}_k(s)-\bx_{k,0}\right|\leq&
\frac{1}{b_k}\left(C_2+|\by_{k,0}|\right)+T\|\bE^\eps\|_\infty
+\frac{C_2C_b}{b_k^2}\int_0^s\frac{1}{\eps}\left|\tilde{\bx}_k(\theta)-\bx_{k,0}\right|
d\theta,
\end{align*}
where $$\|\bE^\eps\|_\infty:=\sup\{|\bE^\eps(t,\bx)|: 0\leq t\leq T,\, |\bx|\leq C_2\}.$$
By Gronwall's inequality, we get estimate
$$\frac{1}{\eps}\left|\tilde{\bx}_k(s)-\bx_{k,0}\right|\leq C_1,\quad \forall 0\leq s\leq S_k,$$
for a constant $C_1>0$ independent of $\eps$.
The last assertion (\ref{lemma result}) then follows from
$$b(\tilde{\bx}_k(s))-b_k=\int_0^s\nabla_\bx b\left(\theta \tilde{\bx}_k(s)+(1-\theta)\tilde{\bx}_k(0)\right)d\theta\cdot(\tilde{\bx}_k(s)-\tilde{\bx}_k(0)).$$
\end{proof}

Thanks to Lemma \ref{lemma1}, we observe that the right-hand-side of (\ref{charact3}) is bounded as $\eps\to0$. If we consider $\bE^\eps(t,\bx)$ is a given external field and contains no fast frequency in the $t$-variable, the formulation (\ref{charact3}) fits the requirement of the two-scale strategy.

\subsection{Two-scale formulation}
Let us first of all consider the case that $\bE^\eps(t,\bx)$ is a given external field without any fast frequencies in the $t$-variable. We shall address the Vlasov-Poisson case later.

Now we perform the two-scale formulation on (\ref{charact3}). Denote the fast variable $\tau=s/\eps$ and separate it out in (\ref{charact3}), then we have
\begin{subequations}\label{charact 2scale}
\begin{align}
&\partial_s X_k+\frac{1}{\eps}\partial_\tau X_k=\frac{\fe^{\tau J}}{b_k}Y_k, \label{charact 2scale a}\\
&\partial_s Y_k+\frac{1}{\eps}\partial_\tau Y_k=\frac{1}{\eps}\left(\frac{b(X_k)}{b_k}-1\right)JY_k+
\fe^{-J\tau}\frac{\bE^\eps(s/b_k,X_k)}{b_k},\quad s>0, \ \tau\in\bT,\label{charact 2scale b}
\end{align}
\end{subequations}
where $X_k=X_k(s,\tau),\,Y_k=Y_k(s,\tau)$ and $\bT=\bR/(2\pi)$ is a torus.
Choosing $X_k(0,0)=\bx_k(0),\,Y_k(0,0)=\bv_k(0)$, we recover the original
unknown by considering the two-scale unknown on the diagonal $\tau=s/\eps$
\begin{equation}\label{diagonal}
X_k(s,s/\eps)=\tilde{\bx}_k(s),\quad Y_k(s,s/\eps)=\tilde{\by}_k(s),\quad s\geq0.
\end{equation}
In the next section, numerical schemes will be proposed for the two-scale system \eqref{charact 2scale}, and our aim will be to prove that these schemes enjoy  uniform accuracy with respect to $\eps$. This property requires a preliminary analysis.
Indeed, one can observe that no initial condition for \eqref{charact 2scale} is evident since only the condition $X_k(0,0)=\bx_k(0),\,Y_k(0,0)=\bv_k(0)$ is required.
This degree of freedom will be used to derive initial conditions
$X_k(0,\tau),\,Y_k(0,\tau)$ such that the two-scale unknown
$X_k(s,\tau),\,Y_k(s,\tau)$ and its time derivative are uniformly bounded.
This will be the objective of the rest of this section.

First, we start with the following elementary lemma
\begin{lemma}\label{lemma2}
Let $\mathcal{H}$ be a Banach algebra space of real-valued functions and let $\mathcal{H}^2$ denote the cartesian product $\mathcal{H} \times \mathcal{H}$.
Consider the following system of ordinary differential equations in $\mathcal{H}^2$ (i.e. $X^\varepsilon$ and $Y^\varepsilon$ are considered as functions from $[0, S_k]$ to $\mathcal{H}^2$)
\begin{eqnarray*}
\frac{d X^\varepsilon}{ds} &=& \frac{1}{\varepsilon} e^{Js/\varepsilon} Y^\varepsilon, \nonumber\\
\frac{d Y^\varepsilon}{ds} &=& \alpha^\varepsilon(s) Y^\varepsilon +\beta^\varepsilon(s) X^\varepsilon + \gamma^\varepsilon(s), \nonumber\\
X^\varepsilon(0)=X_0, && Y^\varepsilon(0)=Y_0 \;\;\;  \mbox{two given initial data},
\end{eqnarray*}
with $\alpha^\varepsilon(s), \beta^\varepsilon(s)$ are $2\times 2$ matrices with coefficients in $\mathcal{H}$ and
$\gamma^\varepsilon :  [0, S_k] \to \mathcal{H}^2$.
Assume that there exists a constant $C>0$ independent of $\varepsilon$ such that
$\| \alpha^\varepsilon(s) \|_{\mathcal{H}}\leq C$, $\| \beta^\varepsilon(s) \|_{\mathcal{H}}\leq C$ and
$\| \gamma^\varepsilon(s) \|_{\mathcal{H}}\leq C, \; \forall s\in [0, S_k]$. Then there exists a constant $M>0$ independent
of $\varepsilon$ such that
$$
\|X^\varepsilon(s)\|_{\mathcal{H}^2} + \|Y^\varepsilon(s)\|_{\mathcal{H}^2} \leq M (1+ \|X_0\|_{\mathcal{H}^2} + \|Y_0\|_{\mathcal{H}^2}), \; \forall s\in [0, S_k].
$$
\end{lemma}

\begin{proof}
We integrate the equation on $X^\eps$ and perform an integration by parts to get
\begin{eqnarray*}
X^\eps(s) &=& X_0 + \frac{1}{\varepsilon} \int_0^s e^{J\sigma /\varepsilon} Y^\eps(\sigma) d\sigma \nonumber\\
&=& X_0 + \frac{1}{\varepsilon} \left\{ \left[ -\varepsilon J e^{J\sigma /\varepsilon} Y^\eps(\sigma)\right]_0^s
+ \varepsilon J  \int_0^s e^{J\sigma /\varepsilon} \frac{d Y^\eps(\sigma)}{ds } d\sigma \right\}\nonumber\\
&=& X_0 - J e^{Js /\varepsilon} Y^\eps(s) + J Y_0
+ \varepsilon J  \int_0^s e^{J\sigma /\varepsilon} \left[ \alpha^\eps(\sigma) Y^\eps(\sigma) +\beta^\eps(\sigma) X^\eps(\sigma) + \gamma^\eps(\sigma) \right] d\sigma,
\end{eqnarray*}
where we used the equation on $Y^\eps$.
We always use $C$ in  proofs to denote a positive constant independent of $\eps$,
and its value may change from one line to the next.
Considering the norm $\|\cdot \|_{\mathcal H}^2$ leads to
$$
\|X^\eps(s)\|_{\mathcal{H}^2} \leq \|X_0\|_{\mathcal{H}^2} +C \|Y_0\|_{\mathcal{H}^2} +\|Y^\eps(s)\|_{\mathcal{H}^2}+C +C\int_0^s \left[ \|Y^\eps(\sigma)\|_{\mathcal{H}^2}+\|X^\eps(\sigma)\|_{\mathcal{H}^2}   \right] d\sigma,
$$
using $\|e^{J \tau} u \|_{\mathcal{H}^2}\leq \|u\|_{\mathcal{H}^2}$ for all $u\in \mathcal{H}^2$, with $C$ independent of $\tau$.
Integrating now the equation on $Y^\eps$  gives directly
$$
\|Y^\eps(s)\|_{\mathcal{H}^2} \leq \|Y_0\|_{\mathcal{H}^2} +C +C\int_0^s \left[ \|Y^\eps(\sigma)\|_{\mathcal{H}^2}+\|X^\eps(\sigma)\|_{\mathcal{H}^2}   \right] d\sigma,
$$
which reported in the former inequality gives
$$
\|X^\eps(s)\|_{\mathcal{H}^2} \leq  C(\|X_0\|_{\mathcal{H}^2} +\|Y_0\|_{\mathcal{H}^2}) +C +C\int_0^s \left[ \|Y^\eps(\sigma)\|_{\mathcal{H}^2}+\|X^\eps(\sigma)\|_{\mathcal{H}^2}   \right] d\sigma.
$$
The two last inequalities clearly lead to
$$
\|X^\eps(s)\|_{\mathcal{H}^2} + \|Y^\eps(s)\|_{\mathcal{H}^2} \leq  C(\|X_0\|_{\mathcal{H}^2} +\|Y_0\|_{\mathcal{H}^2}) +C +C\int_0^s \left[ \|Y^\eps(\sigma)\|_{\mathcal{H}^2}+\|X^\eps(\sigma)\|_{\mathcal{H}^2}   \right] d\sigma.
$$
A standard Gronwall lemma enables to prove the result of the lemma.
\end{proof}

For convenience, we shall denote $L^\infty_s:=L^\infty_s([0, S_k])$
and $L^\infty_\tau:=L^\infty_\tau(\bT)$ the functional spaces
in $s$ and $\tau$ variables. Now, for a smooth periodic function $u(\tau)$ on $\bT$,
we introduce
$$
\|u\|_{W_\tau^{1,\infty}} =\max\left\{ \|u\|_{L_\tau^{\infty}}, \|\partial_\tau u\|_{L_\tau^{\infty}}\right\}.
$$
For a smooth vector field
$\bU(\tau)$ on $\bT$, we define its $W^{1, \infty}_\tau$-norm as
$$\|\bU\|_{W_\tau^{1,\infty}}=\max\left\{\|u_1\|_{W_\tau^{1,\infty}}, \|u_2\|_{W_\tau^{1,\infty}}\right\},\quad
\mbox{for}\quad \bU(\tau)=\binom{u_1(\tau)}{u_2(\tau)}.$$
%In the sequel, we implicitly use $C$ as a generic positive constant,  independent of $\eps$ but whose values may  from one  line to the other.

\begin{lemma}\label{lemma21}
Assume (\ref{Lips}), $E^\eps(t,\bx)\in {\mathcal C}^1([0,T]\times\bR^2)$ and $b(\bx)\in {\mathcal C}^1(\bR^2)$.
For the two-scale problem (\ref{charact 2scale}), if initially
\begin{equation}\label{lemma2 assump}
\frac{1}{\eps}\|X_k(0,\cdot)-\bx_{k,0}\|_{W_\tau^{1, \infty}}+\|Y_k(0,\cdot)\|_{W_\tau^{1, \infty}}\leq C_0,
\end{equation}
for some constant $C_0>0$  independent of $\eps$, then
\begin{equation}\label{lemma2_conf}
\frac{1}{\eps}\|X_k-\bx_{k,0}\|_{L_s^\infty(W_\tau^{1, \infty})}+
\|Y_k\|_{L_s^\infty( W_\tau^{1, \infty})}\leq C_1,
\end{equation}
for some $C_1>0$  independent of $\eps$. Hence,
\begin{equation}\label{lemma2 result}
\frac{1}{\eps}\left\|\frac{b(X_k)}{b_k}-1\right\|_{L_s^\infty( W_\tau^{1, \infty})}\leq C_2,
\end{equation}
for some $C_2>0$ independent of $\eps$.
\end{lemma}

\begin{proof}
Consider $\tau$ as a parameter in $\bT$ and define
\begin{equation}
\label{xktau}
{X}_{k,\tau}(s):=X_k\left(s,\tau+\frac{s}{\eps}\right),\quad {Y}_{k,\tau}(s):=Y_k\left(s,\tau+\frac{s}{\eps}\right),\quad s\geq0.
\end{equation}
The two scale problem (\ref{charact 2scale}) then reads
\begin{subequations}
\label{lm2eq1}
 \begin{align}
&\dot{{X}}_{k,\tau}=\frac{\fe^{J(\tau+s/\eps)}}{b_k}{Y}_{k,\tau}, \label{lm2eq1a}\\
&\dot{{Y}}_{k,\tau}=\frac{1}{\eps}\left(\frac{b({X}_{k,\tau})}{b_k}-1\right)J{Y}_{k,\tau}+
\fe^{-J(\tau+s/\eps)}\frac{\bE^\eps(s/b_k,{X}_{k,\tau})}{b_k},\quad s>0. \label{lm2eq1b}
\end{align}
\end{subequations}
Using the same strategy as in the proof of Lemma \ref{lemma1}, we have
$$
|X_k(s, \tau+s/\eps)-\bx_{k,0}|\leq C\eps \mbox{ and }  |Y_k(s, \tau+s/\eps)| \leq C, \;\; \forall \; 0\leq s\leq S_k, 0\leq \tau\leq 2\pi,
$$
so that
\begin{equation}
\label{xestimate_lemma23}
\left\|{X}_{k}\left(s,\cdot\right)-\bx_{k,0}\right\|_{L^\infty_\tau}\leq C\eps, \quad 0\leq s\leq S_k.
\end{equation}
%avec $|\cdot |$ norme euclidienne. $\| x\|_\infty \leq |x| \leq n^{1/2} \| x\|_\infty$ with $x\in \mathbb{R}^n$. }
%from which we deduce
%\begin{align}\label{priorL2}
%\|X_k\left(s,\cdot\right)\|_{L^\infty}+\|Y_k\left(s,\cdot\right)\|_{L^\infty}\leq C,\quad \forall \; 0\leq s\leq S_k.
%\end{align}
%Integrating (\ref{lm2eq1a}) with the integration by parts, we can find that for any $s\geq0$,
%\begin{align}
%&X_{k,\tau}(s)=X_{k,\tau}(0)-\frac{\eps J\fe^{\tau J}}{b_k}\left(\fe^{Js/\eps}Y_{k,\tau}(s)-Y_{k,\tau}(0)\right)\label{lm2eq1p5}\\
%&+
%\frac{\eps J\fe^{\tau J}}{b_k}\int_0^s\fe^{J\theta/\eps}\left(\frac{b(X_{k,\tau}(\theta))}{b_k}-1\right)\frac{JY_{k,\tau}(\theta)}{\eps}d\theta+
%\frac{\eps J}{b_k^2}\int_0^s\bE^\eps(\theta/b_k,X_{k,\tau}(\theta))d\theta.\nonumber
%\end{align}
%Under (\ref{lemma2 assump}) and (\ref{priorL2}), similarly as before we can see
%$$
%\left\|{X}_{k}\left(s,\cdot\right)-\bx_{k,0}\right\|_{L^\infty}\leq C\eps, \quad 0\leq s\leq S_k.
%$$

From now, we focus on the $\tau$-derivative to get $W^{1,\infty}_\tau$ estimate.
First, we rewrite \eqref{lm2eq1} by considering
the new unknown $\tilde{Y}_{k,\tau} = \frac{\fe^{\tau J}}{b_k}{Y}_{k,\tau}$
\begin{subequations}
\label{lm2eq2}
 \begin{align}
&\dot{{X}}_{k,\tau}= \fe^{J s/\eps} \tilde{Y}_{k,\tau}, \label{lm2eq2a}\\
&\dot{\tilde{Y}}_{k,\tau}=\frac{1}{\eps}\left(\frac{b({X}_{k,\tau})}{b_k} - 1 \right)J\tilde{Y}_{k,\tau}+
\fe^{-J s/\eps }\frac{\bE^\eps(s/b_k,{X}_{k,\tau})}{b_k^2},\quad s>0. \label{lm2eq2b}
\end{align}
\end{subequations}
Taking now the derivative with respect to $\tau$ of \eqref{lm2eq2} and denoting
\begin{equation}
\label{change2}
\hat{X}_{k,\tau}(s,\tau):= \frac{1}{\eps}\partial_\tau X_{k,\tau}(s, \tau),\quad \hat{Y}_{k,\tau}(s,\tau):=\partial_\tau \tilde{Y}_{k,\tau}(s,\tau)
\end{equation}
we get
\begin{subequations}
\begin{align}
\dot{ \hat{X}}_{k,\tau}=&\frac{\fe^{J s/\eps}}{\eps}  \hat{Y}_{k,\tau}, \label{lm2eq3a}\\
\dot{ \hat{Y}}_{k,\tau}=&\frac{1}{\eps}  \left( \frac{b(X_{k,\tau})}{b_k}-1 \right)  J\hat{Y}_{k,\tau}  + \nabla_\bx b(X_{k,\tau})
\cdot \hat{X}_{k,\tau} J \tilde{Y}_{k,\tau} +\eps\frac{\fe^{-J s/\eps }}{b_k^2} \nabla_\bx\bE^\eps(s/b_k,X_{k,\tau})  \hat{ X}_{k,\tau} \label{lm2eq3b}
\end{align}
\end{subequations}
Now, using Lemma \ref{lemma2} with $\mathcal{H} = L^{\infty}_\tau([0, 2\pi])$ and \eqref{xestimate_lemma23}, we conclude that
$$
\|\hat{X}_{k,\tau}(s, \cdot)\|_{L^\infty_\tau} +  \|\hat{Y}_{k,\tau}(s, \cdot)\|_{L^\infty_\tau}  \leq C, \;\; \forall 0 \, \leq s\leq S_k,
$$
since $\hat{X}_{k,\tau}(0, \cdot)=O(1)$ and $\hat{Y}_{k,\tau}(0, \cdot)=O(1)$ by assumption, which concludes the proof.
\end{proof}

\subsection{Suitable initial data for the two-scale formulation} \label{sect:ic}
In this subsection, we look for an initial data $X_k(0, \tau), Y_k(0, \tau)$ of
the two-scale formulation \eqref{charact 2scale}, which will ensure that
the time derivatives of the solutions $X_k(s, \tau), Y_k(s, \tau)$
are uniformly bounded. This will be done
using Chapman-Enskog expansion of the solution.

\paragraph{\bf{First order preparation}}
We perform the Chapman-Enskog expansion to get the full initial data $X_k(0,\tau)$, $Y_k(0,\tau)$ for (\ref{charact 2scale}). This will be done by formal arguments and a rigorous statement will be proved in the next subsection.

Denote
\begin{equation}\label{CP}
X_k(s,\tau)=\underline{X}_k(s)+\bh_k(s,\tau),\quad Y_k(s,\tau)=\underline{Y}_k(s)+\br_k(s,\tau),
\end{equation}
where
$$
\underline{X}_k(s)=\Pi X_k(s,\tau),\quad \underline{Y}_k(s)=\Pi Y_k(s,\tau),
$$
with the average operator $\Pi$ defined for some periodic function $u(\tau)$ on $\bT$ as $\Pi u=\frac{1}{2\pi}
\int_0^{2\pi}u(\theta)d\theta.$ Denoting $Lu(\tau)=\partial_\tau u(\tau)$, we have
\begin{equation}\label{eq:X}
\left\{\begin{split}
&\partial_s {\underline{X}}_k=\Pi\fe^{\tau J}\frac{Y_k}{b_k},\quad s>0,\ \tau\in\bT,\\
&\partial_s \bh_k+\frac{1}{\eps}L \bh_k=(I-\Pi)\fe^{\tau J}\frac{Y_k}{b_k},
 \end{split}\right.
\end{equation}
and
\begin{equation}\label{eq:Y}\left\{\begin{split}
&\partial_s {\underline{Y}}_k=\Pi\left[ \frac{1}{\eps}\left(\frac{b(X_k)}{b_k}-1\right)JY_k+
\fe^{-J\tau}\frac{\bE^\eps(s/b_k,X_k)}{b_k} \right],\quad s>0,\ \tau\in\bT,\\
&\partial_s \br_k+\frac{1}{\eps}L \br_k=(I-\Pi)\left[ \frac{1}{\eps}\left(\frac{b(X_k)}{b_k}-1\right)JY_k+
\fe^{-J\tau}\frac{\bE^\eps(s/b_k,X_k)}{b_k} \right].
 \end{split}\right.
\end{equation}
Taking the inverse of $L$, which for a zero average function $u(\tau)$ is computed as
$L^{-1}u(\tau)=(I-\Pi)\int_0^\tau u(\theta)d\theta,$
on the above equations for $\bh_k$ and $\br_k$ and denoting $A:=L^{-1}(I-\Pi)$, we get
\begin{subequations}\label{hr def}
\begin{align}
&\bh_k(s,\tau)=\eps A \fe^{\tau J}\frac{Y_k}{b_k}-\eps L^{-1}\partial_s \bh_k,\label{hr def a}\\
&\br_k(s,\tau)=A\left(\frac{b(X_k)}{b_k}-1\right)JY_k+\eps A\fe^{-J\tau}\frac{\bE^\eps(s/b_k,X_k)}{b_k}-\eps L^{-1}
\partial_s\br_k.\label{hr def b}
\end{align}
\end{subequations}
Assuming that $\partial_s \bh_k,\partial_s \br_k=O(1)$ as $\eps\to 0$, from (\ref{hr def a}) we have firstly
$$\bh_k(s,\tau)=O(\eps),\quad s\geq0,\ \tau\in\bT.$$
Thanks to Lemma \ref{lemma21}, we have
$b(X_k)/b_k-1=O(\eps)$, and consequently from (\ref{hr def b}) we get
$$
\br_k(s,\tau)=O(\eps),\quad s\geq0,\ \tau\in\bT.
$$
We now take the time derivative of (\ref{hr def})
\begin{subequations}\label{dthr def}
\begin{align}
&\partial_s\bh_k(s,\tau)= \frac{\eps}{b_k}A \fe^{\tau J}\left(\partial_s{\underline{Y}}_k+\partial_s\br_k\right)-\eps L^{-1}\partial_s^2 \bh_k,\label{dthr def a}\\
&\partial_s\br_k(s,\tau)=A\left(\frac{b(X_k)}{b_k}-1\right)J\left(\partial_s{\underline{Y}}_k+\partial_s\br_k\right)
+\frac{1}{b_k}A\nabla_\bx b(X_k)\cdot(\partial_s{\underline{X}}_k+\partial_s\bh_k)JY_k\nonumber\\
&\qquad+ \frac{\eps}{b_k}A\fe^{-J\tau}\left(\frac{\partial_t\bE^\eps(s/b_k,X_k)}{b_k}+\nabla_\bx\bE^\eps(s/b_k,X_k)
(\partial_s{\underline{X}}_k+\partial_s\bh_k)\right)-\eps L^{-1}
\partial_s^2\br_k.\label{dthr def b}
\end{align}
\end{subequations}
Assuming that $\partial^2_s \bh_k,\partial^2_s \br_k=O(1)$ and observing  that
$\partial_s{\underline{X}}_k=\Pi (\fe^{\tau J} \br_k)/b_k = O(\eps)$,
we get $\partial_s \bh_k,\partial_s \br_k=O(\eps)$.

We then obtain the first order asymptotic expansions from (\ref{hr def})
\begin{align*}
\bh_k(0,\tau)=& \frac{\eps}{b_k}A \fe^{\tau J}\underline{Y}_k(0)+O(\eps^2),\\
\br_k(0,\tau)=&A\left(\frac{b\left(\underline{X}_k(0)+\bh_k(0,\tau)\right)}{b_k}-1\right)J\underline{Y}_k(0)
\hspace{-1.4cm}&+\eps A\fe^{-J\tau}\frac{\bE^\eps(0,\underline{X}_k(0))}{b_k}+O(\eps^2),
\end{align*}
To determine $\bh_k$ and $\br_k$ at $s=0$, we then need to compute $\underline{X}_k(0)$ and $\underline{Y}_k(0)$.
These quantities will be determined from the initial conditions $\bx_{k,0}$ and $\bv_{k,0}$. Indeed,
we recall that
\begin{equation}
\label{ci_xy}
\bx_{k,0} = \underline{X}_k(0)+\bh_k(0,0),\qquad \bv_{k,0}=\underline{Y}_k(0)+\br_k(0,0).
\end{equation}
Since $\bh_k=O(\eps)$, we have
$$
\bh_k(0,\tau)= \bh_k^{1st}(\tau)+O(\eps^2), \mbox{ with }  \bh_k^{1st}( \tau):=\frac{\eps}{b_k}A \fe^{\tau J}\bv_{k,0} = -\frac{\eps}{b_k} J\be^{\tau J}\bv_{k,0},
$$
so that we get the following first order expansion for $X_k$
$$
X_k(0, \tau)= \underline{X}_k(0) + \bh^{1st}_k(\tau) +O(\eps^2),
$$
so that,
\begin{equation}
\label{ini 1st}
X_k(0, \tau)=X^{1st}_k( \tau)+O(\eps^2),
\end{equation}
where, using \eqref{ci_xy}  we define
\begin{eqnarray}
X^{1st}_k( \tau) &:=&\bx_{k,0} +\bh_k^{1st}(\tau)- \bh_k^{1st}(0)\nonumber\\
&=&  \bx_{k,0} -\frac{\eps}{b_k} J(\be^{\tau J}-I)\bv_{k,0}.
\label{Xfirst}
\end{eqnarray}
Since $\br_k=O(\varepsilon)$, we have
\begin{align*}
&\br_k(0,\tau)=A\left(\frac{b \left({X}_k^{1st}(\tau)\right)}{b_k}-1\right)J\bv_{k,0}+\eps A\fe^{-J\tau}\frac{\bE^\eps(0,\bx_{k,0})}{b_k}+O(\eps^2).
\end{align*}
We then derive a first order expansion for $Y_k$
$$
Y_k(0, \tau)=\underline{Y}_k(0)+\br_k^{1st}(\tau)+O(\eps^2),
$$
with
\begin{equation}
\label{rk1st}
\br_k^{1st}(\tau):=A\left(\frac{b(X_k^{1st}(\tau))}{b_k}-1\right)J\bv_{k,0}+\eps J\be^{-\tau J}\frac{\bE^\eps(0,\bx_{k,0})}{b_k}.
\end{equation}
We combine this identity with \eqref{ci_xy} to get the first order approximation of $Y_k$
\begin{equation}
\label{Yinit}
Y_k(0, \tau)=Y^{1st}_k(\tau)+O(\eps^2),
\end{equation}
where
\begin{eqnarray}
Y_k^{1st}(\tau) &:=& \bv_{k,0} + \br_k^{1st}(\tau)-\br_k^{1st}(0),\nonumber\\
\label{Yfirst}
&=& \bv_{k,0} + \int_0^\tau (I-\Pi) \left[ \frac{b(X_k^{1st}(\sigma))}{b_k}-1\right] d\sigma J\bv_{k,0}+\eps J (\be^{-\tau J}-I)\frac{\bE^\eps(0,\bx_{k,0})}{b_k}. \qquad %\nonumber
\end{eqnarray}

\paragraph{\bf{Second order preparation}}
We continue the preparation of initial data to the second order in $\eps$.
Inserting \eqref{dthr def a} in \eqref{hr def a} and using $\partial_s \br_k=O(\eps)$ and $Y_k(0, \tau)-Y_k^{1st}(\tau)=O(\eps^2)$, we get
\begin{eqnarray}
\bh_k(0,\tau)
&=& \eps A \fe^{\tau J}\frac{Y_k^{1st}(\tau)}{b_k}-\frac{\eps^2}{b_k} L^{-1}  A \fe^{\tau J} \partial_s \underline{Y}_k(0) + O(\varepsilon^3),
\label{h1}
\end{eqnarray}
where we assumed $\partial_s^3 \bh_k=O(1)$ which implies as above that
$\partial_s^2 \bh_k=O(\eps)$.
Now, from \eqref{eq:Y}, since $X_k(s, \tau)-\underline{X}_k(s)=O(\eps)$ and
$X_k(0, \tau)-X_k^{1st}( \tau)= O(\varepsilon^2)$, we can write
\begin{eqnarray*}
 \partial_s {\underline{Y}}_k(0)
&=& \Pi \left[ \frac{1}{\eps}\left(\frac{b(X_k^{1st})}{b_k}-1\right)J\underline{Y}_k(0)+\fe^{-J\tau}\frac{\bE^\eps(0,\underline{X_k}(0))}{b_k} \right] + O(\varepsilon)\nonumber\\
&=& \Pi \left[ \frac{1}{\eps}\left(\frac{b(X_k^{1st})}{b_k}-1\right)\right]J \underline{Y}_k(0)  + O(\varepsilon),
\end{eqnarray*}
using $\Pi \fe^{-\tau J} = 0$.  We then define  $\bh_k^{2nd}(\tau)$
($\bh_k=\bh_k^{2nd}+O(\varepsilon^3)$) by injecting the previous expansion in \eqref{h1} to get
\begin{equation}
\label{h2nd}
\bh_k^{2nd}(\tau) :=\frac{ \eps}{b_k} A (\fe^{\tau J}Y_k^{1st}) -\frac{\eps^2}{b_k} L^{-1} A \fe^{\tau J}  \Pi \left[ \frac{1}{\eps}\left(\frac{b(X_k^{1st})}{b_k}-1\right)\right]J \bv_{k,0},
\end{equation}
since $\underline{Y}_k(0)=\bv_{k,0} + O(\varepsilon)$ by \eqref{ci_xy}.
We then get the following second order expansion for $X_k$
$$
X_k(0, \tau)= \underline{X}_k(0) + \bh^{2nd}_k( \tau) +O(\eps^3),
$$
so that
\begin{equation}
\label{ini 2nd}
X_k(0, \tau)=X^{2nd}_k(\tau)+O(\eps^3),
\end{equation}
where, using  \eqref{ci_xy}   we define
\begin{equation}
\label{Xsecond}
X^{2nd}_k(\tau) := \bx_{k,0} + \bh_k^{2nd}(\tau) - \bh_k^{2nd}(0),
\end{equation}
where $\bh_k^{2nd}$ is given by \eqref{h2nd} and where $X_k^{1st}, Y_k^{1st}$ are given
by \eqref{Xfirst} and \eqref{Yfirst}.

Let us deal with the second order expansion of $Y_k$. From \eqref{hr def b},
we get
\begin{equation}
\label{rfull}
\br_k(0,\tau):=A \, B_k(X_k^{2nd}) J Y_k^{1st} + \frac{\eps}{b_k}
A\fe^{-J\tau}\bE^\eps(0,X_k^{1st})-\eps L^{-1}\partial_s\br_k +O(\eps^3),
\end{equation}
with $B_k(X)=\frac{b(X)}{b_k}-1$. Therefore, it  remains to find an expansion
of $\partial_s\br_k$ (given by \eqref{dthr def b}) up to order $1$ in $\eps$.
To that purpose, we use  \eqref{dthr def b}, \eqref{dthr def a} and the first equation of \eqref{eq:Y}. We find $\partial_s\br_k (0, \tau)= {\mathcal T}^{2nd}(\tau) + O(\eps^2)$ where  ${\mathcal T}^{2nd}$ is given by
\begin{align}
 \mathcal{T}^{2nd}(\tau) &= \frac{1}{\eps} A B_k(X_k^{1st}) J \Pi \left[B_k(X_k^{1st}) J \bv_{k,0} \right] +\frac{\eps}{b_k^2} A\fe^{-\tau J}  \partial_t E(0, \bx_{k,0})\nonumber\\
&+\frac{1}{b_k^2}A \nabla_\bx b (X_k^{1st}) \cdot \left[ \Pi (\fe^{\tau J} Y_k^{1st})
+A\fe^{\tau J} \Pi (B_k(X_k^{1st}) J \bv_{k,0} \right] J\bv_{k,0}. %\nonumber\\
% \mathcal{T}^{2nd}(0, \tau) &= \frac{1}{\eps} A B_k(X_k^{1st}) J \Pi \left[B_k(X_k^{1st}) J \underline{Y}_k \right] \nonumber\\
%&+\frac{1}{b_k^2}A \nabla b (X_k^{1st}) \cdot \left[ \Pi (\fe^{\tau J} Y_k^{1st})
%+A\fe^{\tau J} \Pi (B_k(X_k^{1st}) J \underline{Y}_k \right] JY_k^{1st}\nonumber\\
%&+\frac{\eps}{b_k^2} A\fe^{-\tau J} \left[ \phantom{\int} \!\!\!\!\!  \partial_t E(0, X_k^{1st}) +\right. \nonumber\\
%&\left.\nabla E(0, X_k^{1st}) \left( \Pi (\fe^{\tau J} Y_k^{1st}) + \eps A\fe^{\tau J} \Pi (B_k(X^{1st}_k)) J \underline{Y}_k   \right)    \phantom{\int} \!\!\!\!\! \right].
\label{tsecond}
\end{align}
We now insert this expression in \eqref{rfull} to get
\begin{equation}
\label{r2nd}
\br_k^{2nd}(\tau):=A B_k(X_k^{2nd})JY_k^{1st}+\frac{\eps}{b_k}
A\fe^{-J\tau}\bE^\eps(0,X_k^{1st})-\eps L^{-1}\mathcal{T}^{2nd}.
\end{equation}
We then get the following third order expansion for $Y_k$
$$
Y_k(0, \tau)= \underline{Y}_k(0) + \br^{2nd}_k(\tau) +O(\eps^3),
$$
so that
\begin{equation}
\label{ini 3rd}
Y_k(0, \tau)=Y^{2nd}_k( \tau)+O(\eps^3),
\end{equation}
 where, using  \eqref{ci_xy}
\begin{equation}
\label{Ysecond}
Y^{2nd}_k(\tau) := \bv_{k,0} + \br_k^{2nd}(\tau) - \br_k^{2nd}(0),
\end{equation}
where $\br_k^{2nd}$ is given by \eqref{r2nd} and where $X_k^{1st}, Y_k^{1st}, X_k^{2nd}, \mathcal{T}^{2nd}$
are given by \eqref{Xfirst}, \eqref{Yfirst}, \eqref{Xsecond},  and \eqref{tsecond}.

\subsection{Estimates of the time derivatives} \label{sect:re}
In this subsection, we prove that the time derivatives of
$X_k(s, \tau)/\eps, Y_k(s, \tau)$ are uniformly bounded when the initial data is
chosen following the Chapman-Enskog procedure presented previously.
Note that due to the $1/\eps$ factor for $X_k$, the expansion for $X_k$ has
to be performed one order further compared to the expansion of $Y_k$.
This is stated in the following proposition.
%\textcolor{red}{(Explain/comment why $X_k$ has to prepare more than $Y_k$)}
\begin{proposition}\label{prop:1st}
$(i)$ Assuming (\ref{Lips}) and  $E^\eps(t,\bx)\in {\mathcal C}^2([0,T]\times\bR^2)$ and $b(\bx)\in {\mathcal C}^2(\bR^2)$.
With the first order initial data $X_k(0,\tau)=X_k^{1st}(\tau)$ given by  \eqref{Xfirst}
and $Y_k(0,\tau)=\bv_{k,0}$
the solution of the two-scale system  (\ref{charact 2scale}) satisfies
$$\frac{1}{\eps}\|\partial_sX_k\|_{L_s^\infty( W^{1,\infty}_\tau)}+\|\partial_sY_k\|_{L_s^\infty( W^{1,\infty}_\tau)}\leq C_0,$$
for some constant $C_0>0$  independent of $\eps$.

$(ii)$ Assuming (\ref{Lips}) and $E^\eps(t,\bx)\in {\mathcal C}^3([0,T]\times\bR^2)$ and $b(\bx)\in {\mathcal C}^3(\bR^2)$. With the second order initial data $X_k(0,\tau)=X_k^{2nd}(\tau)$ given by \eqref{Xsecond} and $Y_k(0,\tau)=Y_k^{1st}(\tau)$ given by \eqref{Yfirst}, the solution of the two-scale system  (\ref{charact 2scale}) satisfies
$$
\frac{1}{\eps}\|\partial_s^\ell X_k\|_{L_s^\infty( W_\tau^{1,\infty})}+\|\partial_s^\ell Y_k\|_{L_s^\infty( W_\tau^{1,\infty})}\leq C_0,\quad \ell=1,2,
$$
 for some constant $C_0>0$ independent of $\eps$.
\end{proposition}

\begin{proof}
The time derivatives of the unknown, denoted in this proof as $\tilde{X}_k(s, \tau):=\partial_s X_k(s, \tau),$ $\tilde{Y}_k(s, \tau):=\partial_s Y_k(s, \tau)$ satisfy
\begin{align}
\partial_s \tilde{X}_k+\frac{1}{\eps}\partial_\tau \tilde{X}_k=&\frac{\fe^{\tau J}}{b_k}\tilde{Y}_k, \label{xtilde_s}\\
\partial_s \tilde{Y}_k+\frac{1}{\eps}\partial_\tau \tilde{Y}_k=&\frac{1}{\eps}\left(\frac{b(X_k)}{b_k}-1\right)J\tilde{Y}_k+\frac{1}{\eps b_k}\nabla_\bx b(X_k)\cdot\tilde{X}_k JY_k \nonumber\\
&+\frac{\fe^{-J\tau}}{b_k}\left(\frac{\partial_t\bE^\eps(s/b_k,X_k)}{b_k}+\nabla_\bx\bE^\eps(s/b_k,X_k)\tilde{X}_k\right).\label{ytilde_s}
\end{align}
With $X_k(0,\tau)=X_k^{1st}(\tau)$ (given by (\ref{ini 1st})-\eqref{Xfirst}) and $Y_k(0,\tau)=\bv_{k,0}$
and using equation (\ref{charact 2scale}), we find the following initial data for the previous system
\begin{align}
\partial_s X_k(0, \tau)&=\tilde{X}_k(0,\tau)=-\frac{1}{\eps}\partial_\tau X_k^{1st}( \tau)+\frac{\fe^{\tau J}}{b_k}\bv_{k,0}\nonumber\\
\label{dsX_0}
&=-\frac{\fe^{\tau J}}{b_k}\bv_{k,0}+\frac{\fe^{\tau J}}{b_k}\bv_{k,0}=0,\\
\partial_s Y_k(0, \tau)&=\tilde{Y}_k(0,\tau)=-\frac{1}{\eps}\partial_\tau \bv_{k,0}+\frac{1}{\eps}\left(\frac{b(X_k^{1st})}{b_k}-1\right)J\bv_{k,0}+\fe^{-J\tau}\frac{\bE^\eps(0,X_k^{1st})}{b_k}, \nonumber\\
\label{dsY_0}
&\hspace{-1.59cm}=-\left[ \frac{\nabla_\bx b(\bx_{k,0})}{b_k}\cdot J(\fe^{\tau J}-I)\bv_{k,0}+O(\eps) \right]J\bv_{k,0} +\fe^{-J\tau}\frac{\bE^\eps(0,X_k^{1st})}{b_k}=O(1).
\end{align}
Now we fix $\tau\in\bT$ as a parameter and together with \eqref{xktau}, we define
$$
\tilde{X}_{k,\tau}(s):=\frac{1}{\eps}\tilde{X}_k\left(s,\frac{s}{\eps}+\tau\right),\quad
\tilde{Y}_{k,\tau}(s):=\frac{e^{\tau J}}{b_k}\tilde{Y}_k\left(s,\frac{s}{\eps}+\tau\right),
$$
which solves
\begin{align*}
\dot{ \tilde{X}}_{k,\tau}=&\frac{\fe^{J s/\eps}}{\eps}\tilde{Y}_{k,\tau}, \\
\dot{\tilde{Y}}_{k,\tau}=&\frac{1}{\eps}\left(\frac{b(X_{k,\tau})}{b_k}-1\right)J\tilde{Y}_{k,\tau}+\frac{e^{J\tau}}{b_k^2}\nabla_\bx b(X_{k,\tau})\cdot\tilde{X}_{k,\tau} JY_{k,\tau}\\
&+\frac{\fe^{-J s/\eps}}{b_k^2}\left(\frac{\partial_t\bE^\eps(s/b_k,X_{k,\tau})}{b_k}+\eps\nabla_\bx\bE^\eps(s/b_k,X_{k,\tau})\tilde{X}_{k,\tau}\right).
\end{align*}
We can apply Lemma \ref{lemma2} with $\mathcal{H} = L_\tau^{\infty}([0, 2\pi])$
since all the assumptions of this  lemma are fullfilled. Indeed, $\tilde{X}_{k,\tau}(0)$ and $\tilde{Y}_{k,\tau}(0)$ are uniformly bounded in $\eps$, the functions ${X}_{k,\tau}$ and ${Y}_{k,\tau}$ enjoy some boundedness properties (thanks to Lemma  \ref{lemma21}) and the functions $b$ and $E$ are smooth. This enables to derive the following estimate
$$
\frac{1}{\eps}\|\partial_s X_k\|_{L^\infty_s ( L^\infty_\tau)}+\|\partial_sY_k\|_{L_s^\infty ( L^{\infty}_\tau)}\leq C_0.
$$

We now deal with $W_\tau^{1, \infty}$
estimates. This is done by differentiating the above system with respect to $\tau$. The so-obtained system
is still of the form of the Lemma \ref{lemma2} and we have to check that the initial data remains bounded.
Clearly the $\tau$-derivative of $\tilde{X}_k(0, \tau)$ is equal to zero ;
concerning the $\tau$-derivative of $\tilde{Y}_k(0, \tau)$, we have
$$
\partial_\tau \tilde{Y}_k(0,\tau)=\frac{1}{\eps b_k} \nabla_\bx b(X_k^{1st})\cdot \partial_\tau X_k^{1st} J\bv_{k,0}+ \partial_\tau \left(\fe^{-J\tau}\frac{\bE^\eps(0,X_k^{1st})}{b_k}\right).
$$
From the definition \eqref{ini 1st}-\eqref{Xfirst} of  $X^{1st}_k(0, \tau)$, we have $\partial_\tau X_k^{1st} =\frac{\eps}{b_k} e^{\tau J} \bv_{k,0}$
so that $\partial_\tau \tilde{Y}_k(0,\tau) = O(1)$. Then,  using Lemma \ref{lemma2} with $\mathcal{H} = L^{\infty}_\tau([0, 2\pi])$ ends the proof of $(i)$.

Let us now prove $(ii)$. We denote $\check{X}_k:=\partial_s^2 X_k,\; \check{Y}_k:=\partial_s^2 Y_k$
which are solutions of the following system
\begin{align*}
\partial_s \check{X}_k+\frac{1}{\eps}\partial_\tau \check{X}_k=&\frac{\fe^{\tau J}}{b_k}\check{Y}_k, \\
\partial_s \check{Y}_k+\frac{1}{\eps}\partial_\tau \check{Y}_k=&\frac{1}{\eps}\left(\frac{b(X_k)}{b_k}-1\right)J\check{Y}_k+\frac{2}{\eps b_k}\nabla_\bx b(X_k)\cdot\tilde{X}_k J\tilde{Y}_k\\
&+\frac{1}{\eps b_k}\nabla_\bx b(X_k)\cdot\check{X}_k JY_k+\frac{1}{\eps b_k}\tilde{X}_k^T\nabla_\bx^2 b(X_k)\tilde{X}_k JY_k\\
&+\frac{\fe^{-J\tau}}{b_k}\left(\frac{\partial_t^2\bE^\eps(s/b_k,X_k)}{b_k^2}+\frac{2}{b_k}\nabla_\bx\partial_t\bE^\eps(s/b_k,X_k)\tilde{X}_k\right)\\
&+\frac{\fe^{-J\tau}}{b_k}\left(\tilde{X}_k^T\nabla_\bx^2\bE^\eps(s/b_k,X_k)\tilde{X}_k+\nabla_\bx\bE^\eps(s/b_k,X_k)\check{X}_k\right).
\end{align*}
In order to apply Lemma \ref{lemma2}, we first check that the initial data for this system is uniformly bounded (in $\eps$). To do so, we use the system \eqref{xtilde_s}-\eqref{ytilde_s} at $s=0$ to obtain
\begin{equation}
\label{ds2X_0}
\check{X}_k(0,\tau)=\partial_s \tilde{X}_k(0, \tau) = -\frac{1}{\eps}\partial_\tau \tilde{X}_k(0,\tau)+\frac{\fe^{\tau J}}{b_k} \tilde{Y}_k(0,\tau).
\end{equation}
Firstly, we find
\begin{eqnarray*}
\partial_s X_k(0, \tau) &=& \tilde{X}_k(0,\tau)=-\frac{1}{\eps}\partial_\tau X_k^{2nd}(\tau)+\frac{\fe^{\tau J}}{b_k}Y_k^{1st}(\tau)\nonumber\\
&=&-\frac{1}{\eps}\partial_\tau X_k^{1st}( \tau)+\frac{\fe^{\tau J}}{b_k}v_{k,0} + O(\eps),
\end{eqnarray*}
which leads to $\tilde{X}_k(0,\tau)=O(\eps)$ thanks to \eqref{dsX_0},
so that we deduce $\frac{1}{\eps}\partial_\tau \tilde{X}_k(0,\tau) = O(1)$.
%Therefore, $\check{X}_k(0,\tau)$ is uniformly bounded.
Concerning the second term in \eqref{ds2X_0}, we look at $\tilde{Y}_k(0, \tau)$
\begin{align*}
&\partial_s Y_k(0, \tau) =\tilde{Y}_k(0,\tau)=-\frac{1}{\eps}\partial_\tau Y_k^{1st}+\frac{1}{\eps}B_k(X_k^{2nd})JY_k^{1st}
+\frac{\fe^{-J\tau}}{b_k}\bE^\eps(0,X_k^{2nd})\\
&=\frac{1}{\eps}B_k(X_k^{2nd})JY_k^{1st} + O(1)=\frac{1}{\eps b_k} \nabla_\bx b(\bx_{k,0}) \cdot  (X_k^{2nd} - \bx_{k,0})JY_k^{1st}+ O(1) = O(1),
\end{align*}
with $B_k(X)=\frac{b(X)}{b_k}-1$. This enables to prove that $\check{X}_k(0, \tau)$ given by \eqref{ds2X_0} is uniformly bounded.

We now focus on $\check{Y}_k(0, \tau)$ which is given by
\begin{align*}
\partial_s^2 Y_k(0, \tau) =  \check{Y}_k(0,\tau)
=&\frac{1}{\eps}B_k(X_k^{2nd})J\tilde{Y}_k(0,\tau)+\frac{1}{\eps b_k}\nabla_\bx b(X_k^{2nd})\cdot\tilde{X}_k(0,\tau) JY_k^{1st}\\
&
+\frac{\fe^{-J\tau}}{b_k}\left(\frac{\partial_t\bE^\eps(0,X_k^{2nd})}{b_k}+\nabla_\bx\bE^\eps(0,X_k^{2nd})\tilde{X}_k(0,\tau)\right)-\frac{1}{\eps}\partial_\tau \tilde{Y}_k(0,\tau)\\
=&-\frac{1}{\eps}\partial_\tau \tilde{Y}_k(0,\tau)+O(1).
\end{align*}
Let us now look at $\partial_\tau \tilde{Y}_k(0,\tau)$. First, we have
$$
\tilde{Y}_k(0, \tau) = \partial_s Y_k(0, \tau) = -\frac{1}{\eps}\partial_\tau Y_k^{1st} +
\frac{1}{\eps}B_k(X^{2nd}_k) JY_k^{1st} +\frac{\fe^{-\tau J}}{b_k}\bE^\eps(0, X_k^{2nd}).
$$
%with $B_k(X)=\frac{1}{b_k}(b(X)-1)$.
We want to prove that $\tilde{Y}_k(0, \tau)=O(\eps)$
to ensure $\check{Y}_k(0, \tau)=O(1)$. Using \eqref{Yfirst}, we have
$$
\partial_\tau Y_k^{1st}(0, \tau) =B_k(X_k^{1st}) J\bv_{k,0} +\eps \frac{\fe^{-\tau J}}{b_k}\bE^\eps(0, \bx_{k,0}).
$$
Injecting this last identity in the expression of $\tilde{Y}_k(0, \tau)$ leads to
\begin{eqnarray*}
\tilde{Y}_k(0, \tau) &=& -\frac{1}{\eps} \left[B_k(X_k^{1st}) J\bv_{k,0} +\eps \frac{\fe^{-\tau J}}{b_k}\bE^\eps(0, \bx_{k,0})\right] +
\frac{1}{\eps}B_k(X^{2nd}_k) JY_k^{1st} +\frac{\fe^{-\tau J}}{b_k}\bE^\eps(0, X_k^{2nd})\nonumber\\
&=& -\frac{1}{\eps} B_k(X_k^{1st}) J\bv_{k,0} +
\frac{1}{\eps}B_k(X^{1st}_k+ O(\eps^2)) J (\bv_{k,0} + O(\eps)) +O(\eps)\nonumber\\
&=& O(\eps).
\end{eqnarray*}
Hence we have $\check{Y}_k(0,\tau)=O(1)$.

Again, considering new unknown $\frac{1}{\eps}\check{X}_k(s, \frac{s}{\eps} + \tau)$ and
$\frac{e^{\tau J}}{b_k}\check{Y}_k(s, \frac{s}{\eps} + \tau)$ enables to recast the previous system
so that, using the previous estimates, we can use Lemma \ref{lemma2} with $\mathcal{H} = L^{\infty}_\tau([0, 2\pi])$
provided that the initial data $\frac{1}{\eps}\check{X}_k(0, \tau)=O(1)$. Then, we compute
$$
\frac{1}{\eps}\check{X}_k(0, \tau)=-\frac{1}{\eps^2}\partial_\tau \tilde{X}_k(0, \tau) + \frac{1}{\eps b_k}\fe^{\tau J}\tilde{Y}(0, \tau) =
-\frac{1}{\eps^2}\partial_\tau \tilde{X}^{2nd}_k + \frac{1}{\eps b_k}\fe^{\tau J}\tilde{Y}^{1st}.
$$
We focus on the first term
\begin{eqnarray*}
-\frac{1}{\eps^2}\partial_\tau \tilde{X}^{2nd}_k
&=& -\frac{1}{\eps^2 b_k}\partial_\tau (\fe^{\tau J} Y^{1st}) +\frac{1}{\eps^2 b_k}\fe^{\tau J} \Pi B_k(X^{1st}) J\bv_{k,0} +\frac{1}{\eps^2 b_k}\partial_\tau (\fe^{\tau J} Y^{1st})\nonumber\\
&=& -\frac{1}{\eps^2 b_k} \fe^{\tau J} \Pi B_k(X^{1st}) J\bv_{k,0}.
\end{eqnarray*}
Then, we compute the second term
\begin{eqnarray*}
\frac{1}{\eps b_k}\fe^{\tau J}\tilde{Y}^{1st} &=& \frac{1}{\eps b_k}\fe^{\tau J}\Big[ -\frac{1}{\eps}\partial_\tau Y^{1st} +\frac{1}{\eps} B(X^{2nd}) JY^{1st} +\frac{1}{b_k}\fe^{-\tau J} \bE(0, X^{2nd})   \Big]\nonumber\\
&=&-\frac{1}{\eps^2 b_k}\fe^{\tau J} \Big[ (I-\Pi)B(X^{1st}) J\bv_{k,0} +\frac{\eps}{b_k} \fe^{-\tau J}\bE(0, \bx_{k,0})\Big] \nonumber\\
&&+\frac{1}{\eps^2 b_k}\fe^{\tau J}B(X^{2nd})JY^{1st} +\frac{1}{\eps b_k^2}\bE(0, X^{2nd})\nonumber\\
&=& \frac{1}{\eps^2 b_k}\fe^{\tau J} \Big[ -(I-\Pi) B(X^{1st}) J\bv_{k,0} + B(X^{2nd}) JY^{1st}\Big]+ \frac{1}{\eps b_k^2}(\bE(0, X^{2nd}) - \bE(0, \bx_{k,0}))\nonumber\\
&=& \frac{1}{\eps^2 b_k}\fe^{\tau J} \Big[ -(I-\Pi) B(X^{1st}) J\bv_{k,0} + B(X^{2nd}) JY^{1st}\Big] + O(1).
\end{eqnarray*}
Gathering the two term leads to
$$
\frac{1}{\eps}\check{X}_k(0, \tau) = -\frac{1}{\eps^2 b_k}\Big[-B(X^{1st}) J\bv_{k,0} +B(X^{2nd}) J Y^{1st}    \Big] = \frac{B(X^{1st})}{\eps^2 b_k}J (\bv_{k,0} -Y^{1st}) + O(1) = O(1),
$$
since $X^{2nd} = X^{1st} + O(\eps^2)$.
Similar computations for $\partial_\tau \check{X}_k$ and $\partial_\tau \check{Y}_k$ leads to the required
estimate.

\end{proof}

\subsection{Numerical illustrations} \label{sect:ni}

To end this section, we illustrate the effect of the preparation of the initial data
on the behaviour of the time derivative of $X_k$ and $Y_k$. To do so,
we consider an example of a single particle (we then omit subscript $k$)
characteristics (\ref{charact}) with initial condition
$$
\bx_0=\binom{10}{1.3},\quad \bv_0=\binom{4.9}{-2.1},
$$
and
\begin{align}
&b(\bx)=1+\sin(x_1)\sin(x_2)/2,\quad E^\eps(t,\bx)=\binom{E_1(\bx)}{E_2(\bx)}(1+\sin(t)/2),\\
&
E_1(\bx)=\cos(x_1/2)\sin(x_2)/2,\quad E_2(\bx)=\sin(x_1/2)\cos(x_2).\nonumber
\end{align}
In Figures \ref{fig:norm1st}, we plot the time history of $\partial_s X,\, \partial_sY,\,\partial_s^2X$ and $\partial_s^2Y$
(by accurate numerical solver) under norm
$$
\|X(t,\cdot)\|_{L^2_\tau}:=\frac{1}{2}\left(\|X_1(t,\cdot)\|_{L^2_\tau}+\|X_2(t,\cdot)\|_{L^2_\tau}\right),\quad X=\binom{X_1}{X_2},
$$
for different initial data (first order initial data $(X^{1st}, \bv_{0})$ with $X^{1st}$ given by \eqref{Xfirst}
and  second order initial data $(X^{2nd}, Y^{1st})$ with $X^{2nd}$ given by \eqref{Xsecond} and $Y^{1st}$
given by \eqref{Yfirst}). The numerical results confirm the results of Proposition \ref{prop:1st}.

%\textcolor{red}{Define what "first order initial data" and "second order initial data" mean. Is it $(X^{2nd}, Y^{1st})$ ?
%Take $L^\infty_\tau$ norms instead of $L^2_\tau$.}

\begin{figure}[t!]
$$\begin{array}{cc}
\psfig{figure=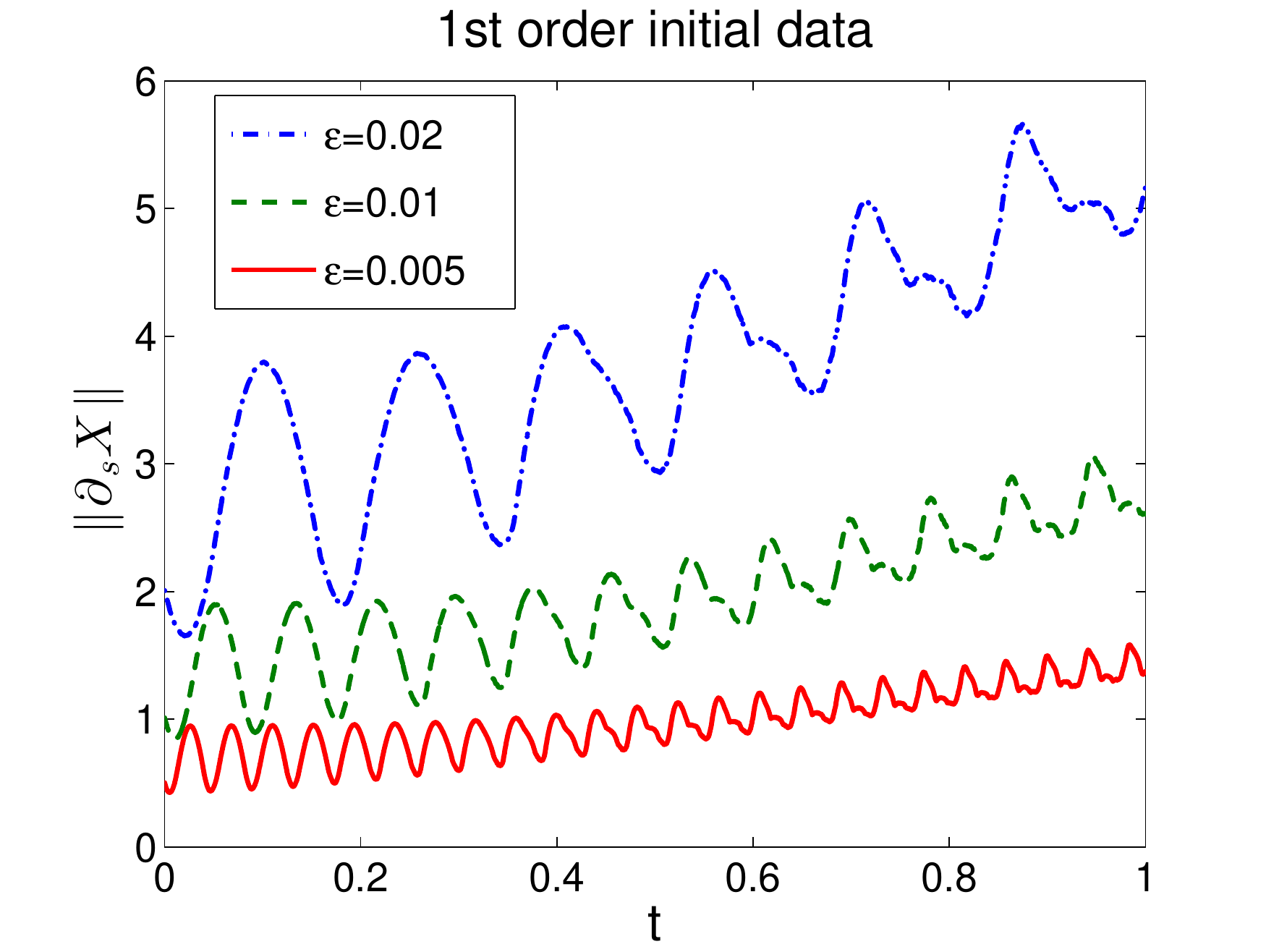,height=4cm,width=5cm}&\psfig{figure=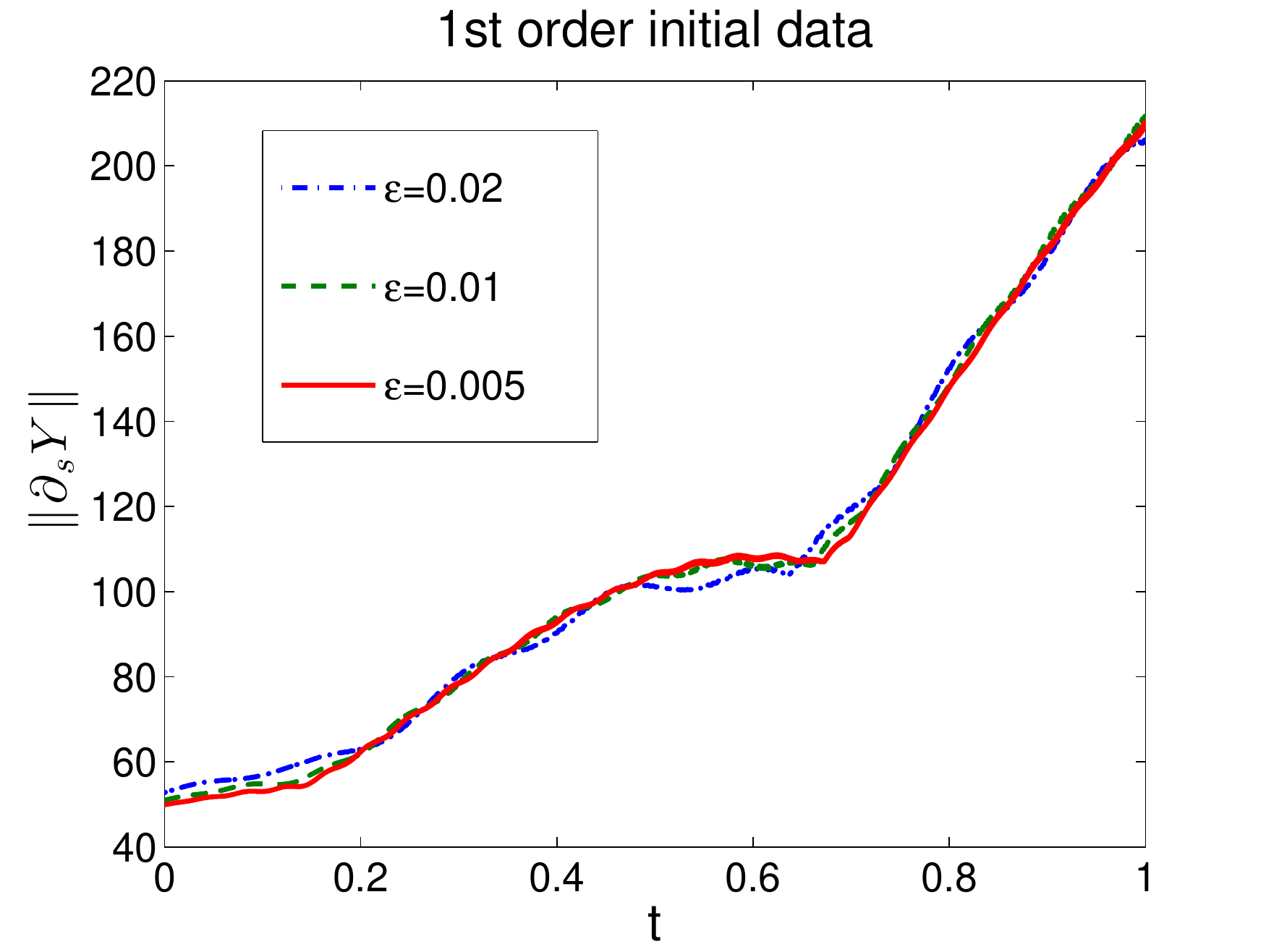,height=4cm,width=5cm}\\
\psfig{figure=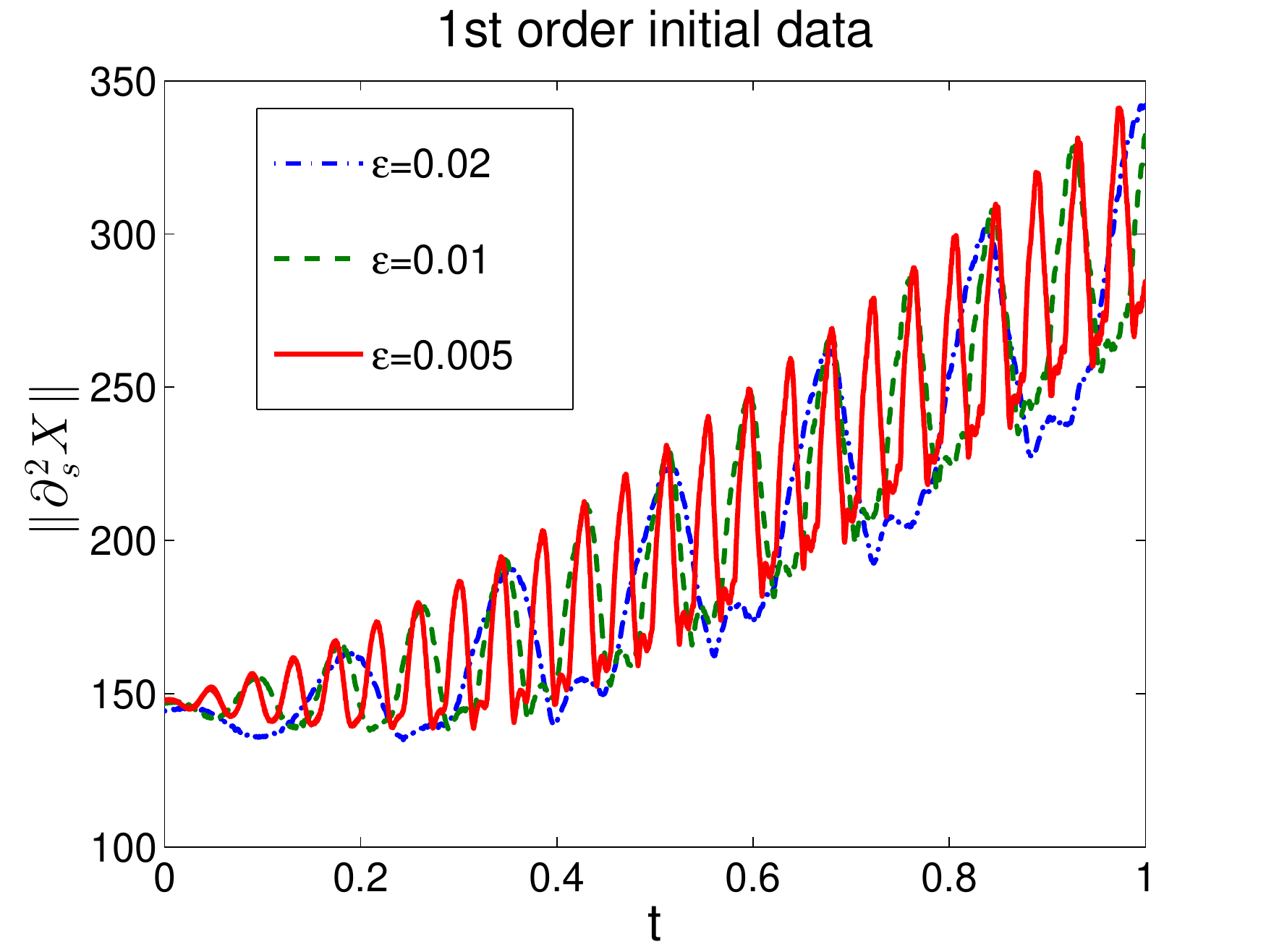,height=4cm,width=5cm}&\psfig{figure=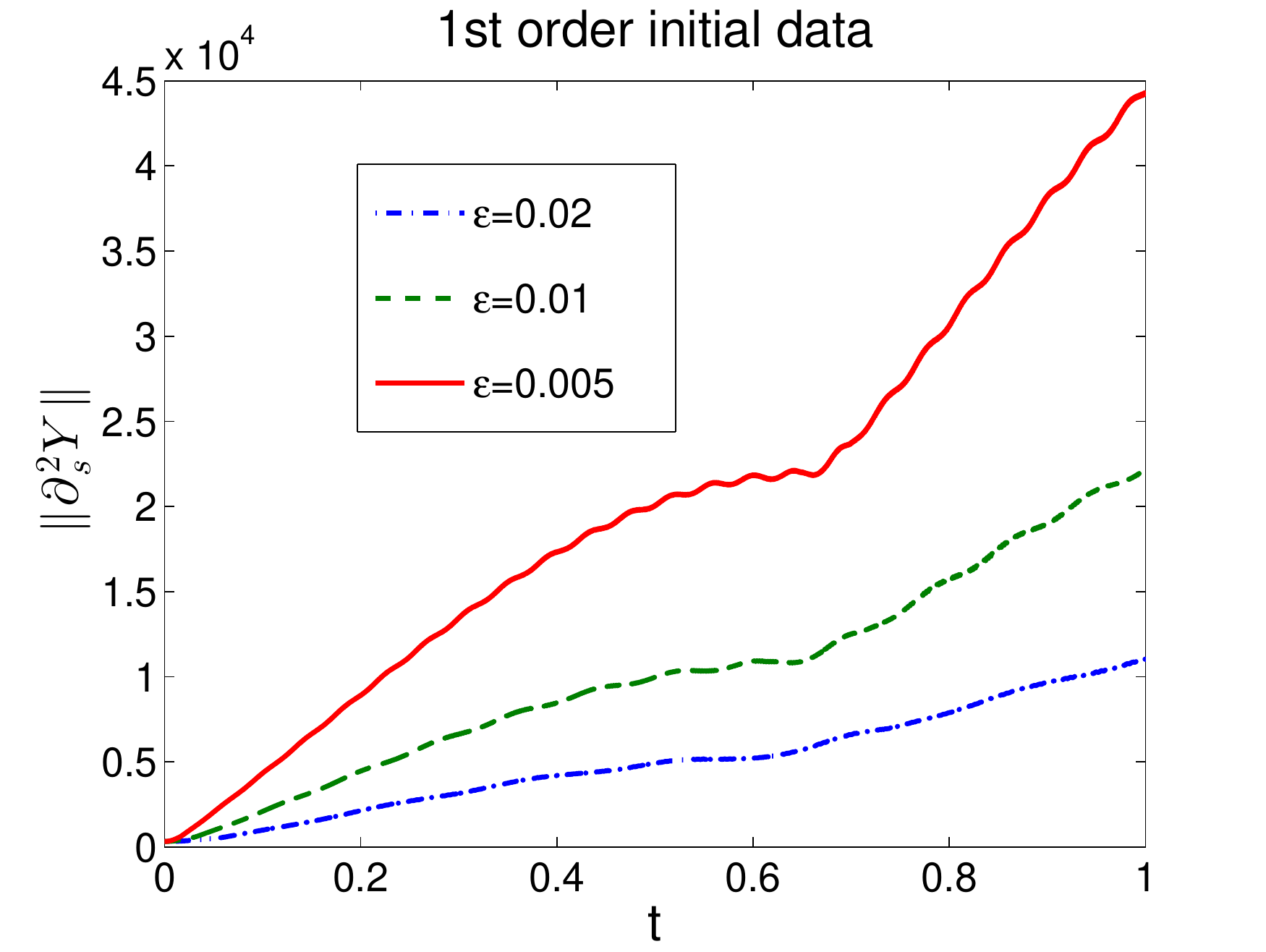,height=4cm,width=5cm}\\
\psfig{figure=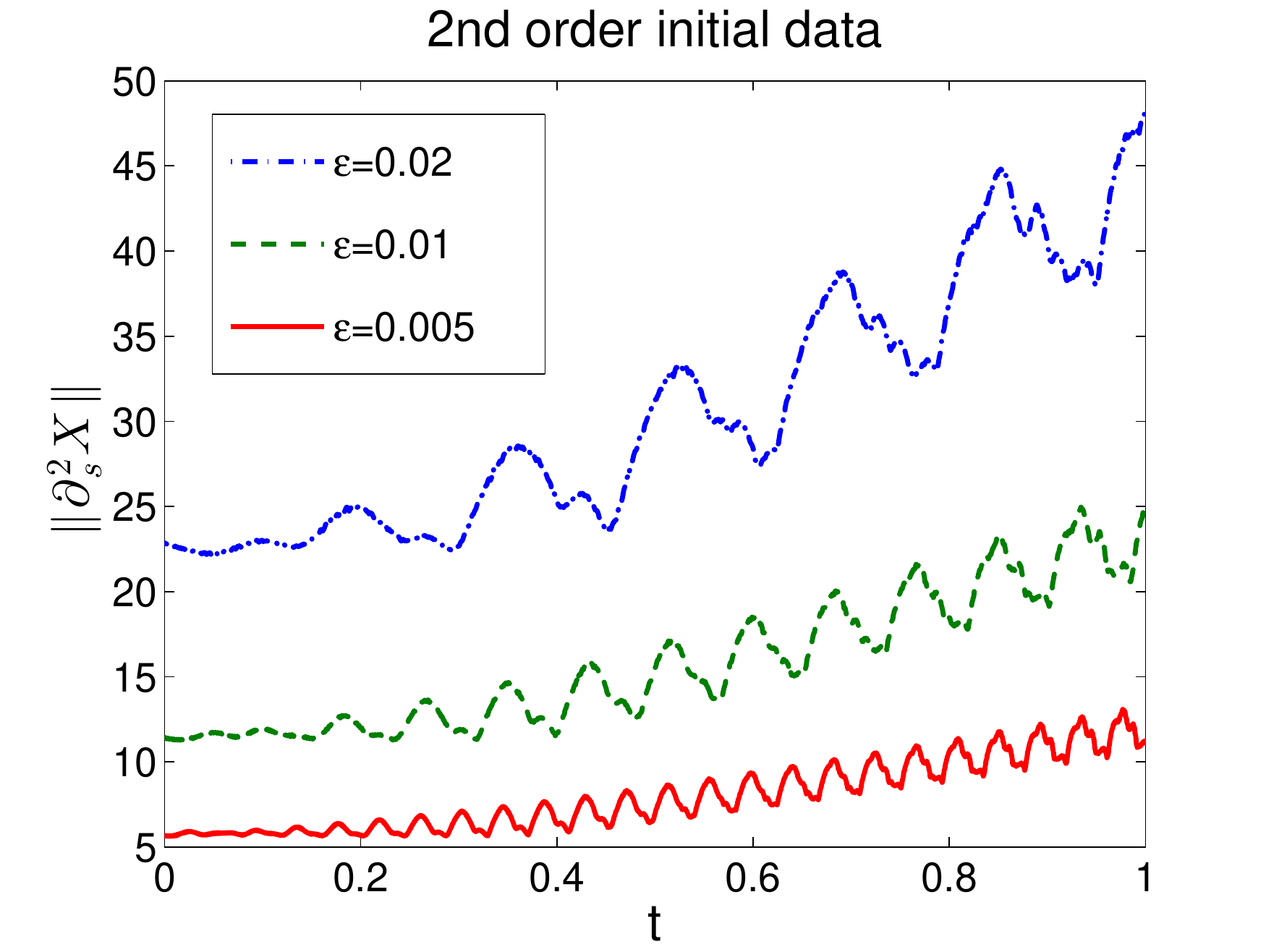,height=4cm,width=5cm}&\psfig{figure=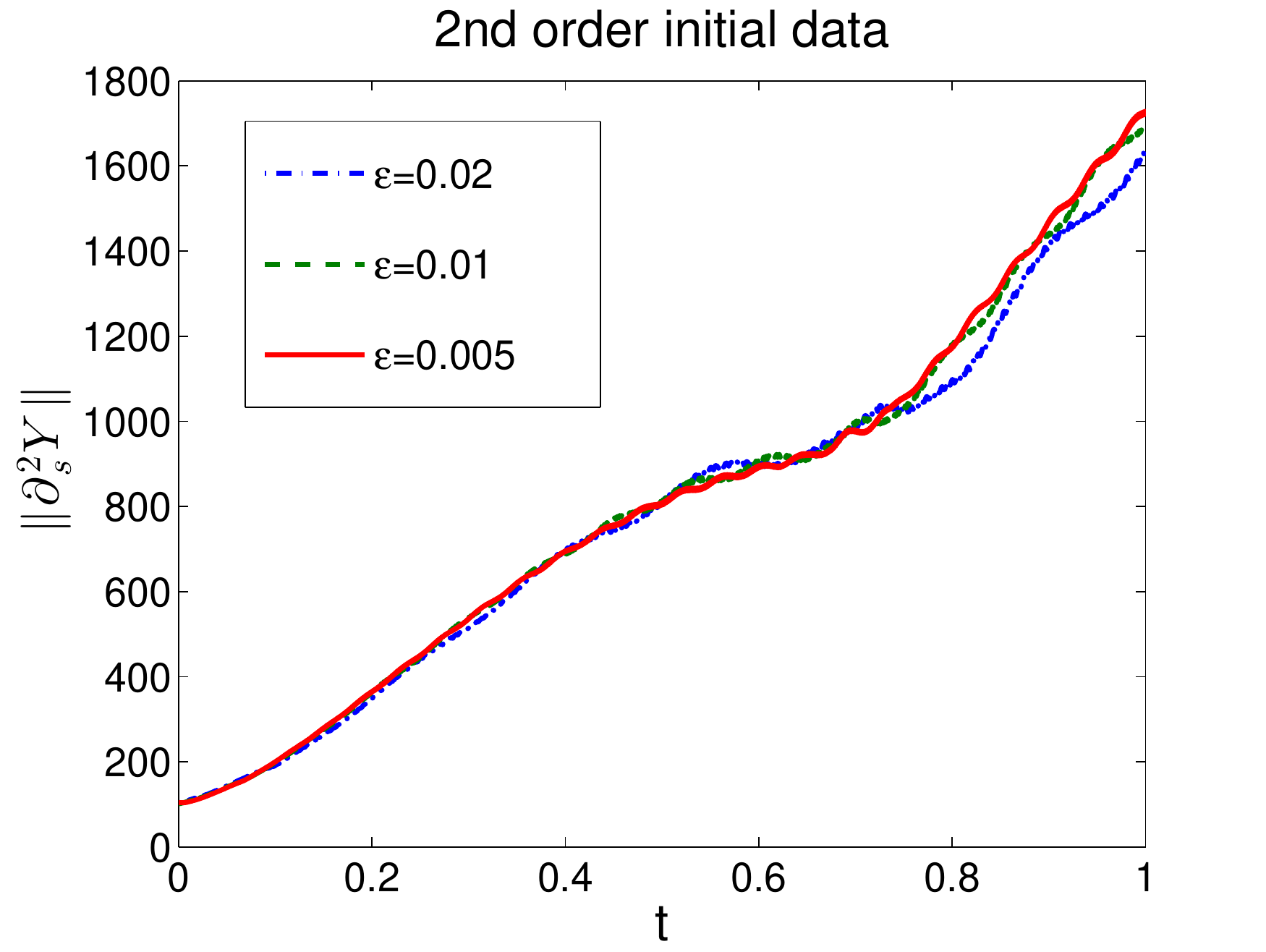,height=4cm,width=5cm}
\end{array}$$
\caption{Time history of the norm of $\partial_sX,\, \partial_sY$ (first row) and $\partial_s^2X,\, \partial_s^2Y$ (second and third rows) under first order initial data $(X^{1st}, \bv_{0})$ and respectively second order initial data
$(X^{2nd}, Y^{1st})$.}\label{fig:norm1st}
\end{figure}

\section{Numerical method} \label{sect:nm}
This section is devoted to the construction of numerical schemes
 for the two-scaled system (\ref{charact 2scale}).
We will perform the analysis of a first order numerical scheme:
we will prove that this numerical scheme enjoy the uniform accuracy property
with respect to $\eps$. In addition, we will prove that the scheme
is able to reproduce the confinement property \eqref{lemma2_conf}
at the discrete level. Then, we will propose a strategy to reach the second order accuracy and to handle the coupling with Poisson equation.

Let $\Delta t>0$ be the time step and denote $t_n=n\Delta t$ for $n\geq0$ as the descretisation of the $t$-variable. For each particle $1\leq k\leq N_p$, the discretisation of the scaled time $s$-variable is consequently as
$$\Delta s_k=b_k\Delta t,\quad s_{k}^n=n\Delta s_k,\quad n=0,1,\ldots.$$
We certainly omit this subscript $k$ for brevity, i.e. $\Delta s=\Delta s_k,\ s_{n}=s^n_{k}.$ Denote the numerical solution as
$$X_k^n(\tau)\approx X_k(s_n,\tau),\quad Y_k^n(\tau)\approx Y_k(s_n,\tau),\quad n=0,1,\ldots,$$
and choose $X_k^0(\tau)=X_k(0,\tau),\,Y_k^0(\tau)=Y_k(0,\tau)$.

\subsection{First order numerical scheme}
A first order implicit-explicit (IMEX1) finite difference scheme reads for $n\geq0,$
\begin{subequations}\label{fd1}
\begin{align}
&\frac{X_k^{n+1}-X_k^n}{\Delta s}+\frac{1}{\eps}\partial_\tau X_k^{n+1}=\frac{\fe^{\tau J}}{b_k}Y_k^{n+1},\\
&\frac{Y_k^{n+1}-Y_k^n}{\Delta s}+\frac{1}{\eps}\partial_\tau Y_k^{n+1}=
\frac{1}{\eps}\left(\frac{b(X_k^n)}{b_k}-1\right)JY^n_k+\fe^{-J\tau}\frac{\bE^\eps(t_{n},X_k^{n})}{b_k}.
\end{align}
\end{subequations}
In the Fourier space in $\tau$, the above scheme is easily diagonalized. By discretizing the $\tau$-direction as
$\tau_j=j\Delta \tau, j=0, \dots, N_\tau$ with $\Delta \tau=2\pi/N_\tau$, $N_\tau$ being some positive even integer,
one can use the Fourier transform in $\tau$ to get a fully discretized scheme. By doing so, let us remark that
the IMEX scheme \eqref{fd1} is explicit from a computational point of view and the error in $\tau$ is uniformly
(with respect to $\eps$) spectrally uniform.

By assuming that $\bE^\eps(t, \bx)$ and $b(\bx)$ are smooth given functions, we analyse the first order IMEX scheme \eqref{fd1}
for which we have the following uniform convergence results.
\begin{theorem}\label{thm1}
Under the assumptions in Proposition \ref{prop:1st} for the two-scaled system (\ref{charact 2scale}),
consider the first order IMEX1 scheme (\ref{fd1}) for solving (\ref{charact 2scale}) with the second order initial data $X_k^0(\tau)=X_k^{2nd}(\tau)$ given by \eqref{Xsecond} and $Y_k^0(\tau)=Y_k^{1st}(\tau)$ given by \eqref{Yfirst}. There exist constants $C_0,C_1,C_2>0$ independent of $\eps$, such that when $0<\Delta s\leq C_0$  we have
\begin{equation}
\label{thm1error}
\frac{1}{\eps}\|X_k^n-X_k(s_n,\cdot)\|_{L^\infty_\tau}+\|Y_k^n-Y_k(s_n,\cdot)\|_{L^\infty_\tau}\leq C_1\Delta s,\quad 0\leq n\leq\frac{S_k}{\Delta s},\end{equation}
and
\begin{subequations}\label{bound}
\begin{align}
&\|X_k^n-\bx_{k,0}\|_{L^\infty_\tau}\leq C_2 \eps,\\
& \|Y_k^n\|_{L^\infty_\tau}\leq \|Y_k\|_{L^\infty_s ( L^\infty_\tau)}+1,\quad 0\leq n\leq\frac{S_k}{\Delta s}.
\end{align}
\end{subequations}
\end{theorem}
The error estimate shows that the scheme with well-prepared initial data offers super-convergence in $\bx_k$.
As a consequence, this super-convergence is also true for space dependent macroscopic quantities
such as $\rho^\eps(t, \bx)$. The estimate \eqref{bound} indicates the confinement property at the discrete level.

We are going to prove this theorem by first introducing two lemmas concerning local truncation error and error propagation.
To simplify the notations, we will always use $C$ to denote a positive constant independent of $\Delta s, n$ or $\eps$
and it could change from line to line. We shall omit the subscript $k$ from now on.

Firstly, we define the local truncation error for the IMEX scheme \eqref{fd1} as
\begin{subequations}\label{local error1}
\begin{align}
\xi_X^{n}(\tau):=&\frac{X(s_{n+1},\tau)-X(s_{n},\tau)}{\Delta s}+\frac{1}{\eps}\partial_\tau X(s_{n+1},\tau)-\frac{\fe^{\tau J}Y(s_{n+1},\tau)}{b},\label{local error1X}\\
\xi_Y^{n}(\tau):=&\frac{Y(s_{n+1},\tau)-Y(s_{n},\tau)}{\Delta s}+\frac{1}{\eps}\partial_\tau Y(s_{n+1},\tau)-\fe^{-\tau J}\frac{\bE^\eps(t_{n},X(s_{n},\tau))}{b}\nonumber\\
&-\frac{1}{\eps}\left(\frac{b(X(s_n,\tau))}{b}-1\right)JY(s_n,\tau),\quad \tau\in\bT,\quad 0\leq n\leq S_k/\Delta s. \label{local error1Y}
\end{align}
\end{subequations}
We have
\begin{lemma}\label{lm:1}
Under the assumptions of Theorem \ref{thm1}, we have
\begin{equation}
\label{error_trunc}
\frac{1}{\eps}\|\xi_X^{n}\|_{W^{1,\infty}_\tau} + \|\xi_Y^{n}\|_{W^{1,\infty}_\tau} \leq C\Delta s,\quad 0\leq n\leq S_k/\Delta s.
\end{equation}
\end{lemma}

\begin{proof}
By Taylor's expansion, we have
$$
X(s_n,\tau)=X(s_{n+1},\tau)-\Delta s\partial_sX(s_{n+1},\tau)+\Delta s^2\int_0^1\theta\partial_s^2
X(s_n+\Delta s\, \theta,\tau)d\theta.
$$
Inserting it into (\ref{local error1X}) and then making use of the equation (\ref{charact 2scale a}), we get
$$\xi_X^n(\tau)=-\Delta s \int_0^1\theta\partial_s^2
X(s_n+\Delta s\, \theta,\tau)d\theta.$$
Hence by Proposition \ref{prop:1st}, we get $\|\xi_X^n\|_{W^{1,\infty}_\tau}\leq\Delta s\|\partial_s^2
X\|_{L^\infty_s( W^{1,\infty}_\tau)}\leq C\eps\Delta s$.

Similarly for the $Y$ equation, we have Taylor's expansions,
\begin{align*}
&Y(s_n,\tau)=Y(s_{n+1},\tau)-\Delta s\partial_sY(s_{n+1},\tau)+\Delta s^2\int_0^1\theta\partial_s^2
Y(s_n+\Delta s\, \theta,\tau)d\theta,\\
&\bE^\eps(t_n,X(s_n,\tau))=\bE^\eps(t_{n+1},X(s_{n+1},\tau))-\int_0^1\Big[\Delta t\partial_t\bE^\eps(t_{n}+\Delta t\theta,X(s_{n},\tau))\\
&\qquad\qquad\qquad\qquad+\Delta s\nabla_\bx\bE^\eps(t_{n+1},X(s_{n}+\Delta s\theta,\tau))\partial_sX(s_{n}+\Delta s\theta,\tau)\Big]d\theta,\\
&b(X(s_n,\tau))=b(X(s_{n+1},\tau))-\Delta s\int_0^1
\nabla_\bx b(X(s_{n}+\Delta s\theta,\tau))\cdot\partial_sX(s_{n}+\Delta s\theta,\tau)d\theta.
\end{align*}
Plugging them into (\ref{local error1Y}), we have
\begin{align*}
\xi_Y^n(\tau)=&\partial_sY(s_{n+1},\tau)+\frac{1}{\eps}\partial_\tau Y(s_{n+1},\tau)-\fe^{-J\tau}\frac{\bE^\eps(t_{n+1},X(s_{n+1},\tau))}{b} \\
&-\frac{1}{\eps}\left(\frac{b(X(s_{n+1},\tau))}{b}-1\right)JY(s_n,\tau)-\Delta s \int_0^1\theta\partial_s^2
Y(s_n+\Delta s\theta,\tau)d\theta\\
&+\frac{\fe^{-J\tau}}{b}\int_0^1\Big[\Delta t\partial_t\bE^\eps(t_{n}+\Delta t\theta,X(s_{n},\tau))\\
&+\Delta s\nabla_\bx\bE^\eps(t_{n+1},X(s_{n}+\Delta s\theta,\tau))\partial_sX(s_{n}+\Delta s\theta,\tau)
\Big]d\theta,\\
&+\frac{\Delta s}{b\eps}\int_0^1
\nabla_\bx b(X(s_{n}+\Delta s\theta,\tau))\cdot\partial_sX(s_{n}+\Delta s\theta,\tau)d\theta JY(s_n,\tau).
\end{align*}
which with the Taylor's expansion
\begin{align*}
&Y(s_n,\tau)=Y(s_{n+1},\tau)-\Delta s\int_0^1\partial_s
Y(s_n+\Delta s\theta,\tau)d\theta,
\end{align*}
and the equation (\ref{charact 2scale b}) becomes
\begin{align*}
\xi_Y^n(\tau)=&
\frac{\Delta s}{\eps}\left(\frac{b(X(s_{n+1},\tau))}{b_k}-1\right)J\int_0^1\partial_s
Y(s_n+\Delta s\theta,\tau)d\theta\\
&-\Delta s \int_0^1\theta\partial_s^2
Y(s_n+\Delta s\theta,\tau)d\theta+\frac{\fe^{-J\tau}}{b}\int_0^1\Big[\Delta t\partial_t\bE^\eps(t_{n}+\Delta t\theta,X(s_{n},\tau))\\
&+\Delta s\nabla_\bx\bE^\eps(t_{n+1},X(s_{n}+\Delta s\theta,\tau))\partial_sX(s_{n}+\Delta s\theta,\tau)
\Big]d\theta,\\
&+\frac{\Delta s}{b\eps}\int_0^1
\nabla_\bx b(X(s_{n}+\Delta s\theta,\tau))\cdot\partial_sX(s_{n}+\Delta s\theta,\tau)d\theta JY(s_n,\tau).
\end{align*}
Hence again by Proposition \ref{prop:1st} and Lemma \ref{lemma21}, we get
$\|\xi_Y^n\|_{W^{1,\infty}_\tau}\leq C\Delta s$.
\end{proof}
We now focus on the propagation error. First, we denote the error function as
$$
\be_X^n(\tau):=X(s_n,\tau)-X^n(\tau),\quad \be_Y^n(\tau):=Y(s_n,\tau)-Y^n(\tau),\quad \tau\in\bT,\ 0\leq n\leq \frac{S_k}{\Delta s}.
$$
It is obvious by the choice of the initial data that
$$\be_X^0(\tau)=\be_Y^0(\tau)=0.$$
Taking the difference between the local error (\ref{local error1}) and scheme (\ref{fd1}), we get the error equations
\begin{subequations}\label{error eq}
\begin{align}
 &\frac{\be_X^{n+1}-\be_X^n}{\Delta s}+\frac{1}{\eps}\partial_\tau \be_X^{n+1}=\frac{\fe^{\tau J}}{b}\be_Y^{n+1}+\xi_X^n,\\
&\frac{\be_Y^{n+1}-\be_Y^n}{\Delta s}+\frac{1}{\eps}\partial_\tau \be_Y^{n+1}=F^n+\xi_Y^n,\quad n\geq0,
\end{align}
\end{subequations}
where
\begin{align}
F^n(\tau)=&\frac{1}{\eps}\left[\left(\frac{b(X(s_n,\tau))}{b_k}-1\right)JY(s_n,\tau)-
\left(\frac{b(X^n(\tau))}{b_k}-1\right)JY^n(\tau)\right]\nonumber\\
\label{F_err}
&+\frac{\fe^{-J\tau}}{b_k}\left[\bE^\eps(t_{n},X(s_{n},\tau))-\bE^\eps(t_{n},X^{n}(\tau))\right].
\end{align}
%In the following notations,
%we denote  $V_1$ and $V_2$ as the first and respectively the second component of a vector $V$, e.g.
%$$\be_X^m=\binom{e_{X,1}^m}{e_{X,2}^m},\quad \be_Y^m=\binom{e_{Y,1}^m}{e_{Y,2}^m},\quad
%F^m=\binom{F_{1}^m}{F_{2}^m},\quad \xi_X^m=\binom{\xi_{X,1}^m}{\xi_{X,2}^m},\quad \xi_Y^m=\binom{\xi_{Y,1}^m}{\xi_{Y,2}^m}.$$ Moreover,
In the following, for any $\tau$-function $\varphi(\tau)$,
its Fourier coefficients are defined by
\begin{equation}
\label{fourier_tau}
\widehat{(\varphi)}_l = \frac{1}{2\pi} \int_0^{2\pi} e^{-i l \tau} \varphi(\tau) d\tau.
\end{equation}
\begin{lemma}\label{lm:2}
For the IMEX1 scheme (\ref{fd1}), we have the following formula for the error function on the $l\in\bZ$ Fourier coefficients  in $\tau$ and for $n\geq1$,
\begin{subequations}\label{formula}
\begin{align}
  \widehat{(\be_{X}^{n})}_l
  =&\frac{\Delta s^2}{2b}\sum_{j=1}^{n} \alpha_{j,l}(I+iJ)\left[\widehat{(F^{n-j})}_{l+1}+\widehat{(\xi_{Y}^{n-j})}_{l+1}\right] +
  \Delta s\sum_{j=0}^{n-1}p_l^{j+1}\widehat{(\xi_{X}^{n-1-j})}_{l} \nonumber\\
+&  \frac{\Delta s^2}{2b}\sum_{j=1}^{n} \beta_{j,l}(I-iJ)\left[\widehat{(F^{n-j})}_{l-1}+\widehat{(\xi_{Y}^{n-j})}_{l-1}\right]  \label{formulaa}\\
   \widehat{(\be_{Y}^{n})}_l
  =&\Delta s\sum_{j=1}^{n}p_l^{j}\left[\widehat{(F^{n-j})}_{l} +\widehat{(\xi_{Y}^{n-j})}_{l} \right],\label{formulac}
  \end{align}
  \end{subequations}
where
$$
\alpha_{j,l}= \frac{p_l^{j+1}p_{l+1}-p_l p_{l+1}^{j+1}}{p_l-p_{l+1}}, \;\;\;\; \beta_{j,l}=\frac{p_l^{j+1}p_{l-1}-p_l p_{l-1}^{j+1}}{p_l-p_{l-1}}  \mbox{ and } \;\;  p_l:=\frac{1}{1+il\Delta s/\eps},\quad l\in\bZ.
$$
\end{lemma}

\begin{proof}
Taking the Fourier transform \eqref{fourier_tau}
of (\ref{error eq}), we have
\begin{align*}
&\frac{\widehat{(\be_X^{n+1})}_l-\widehat{(\be_X^{n})}_l}{\Delta s}+\frac{il}{\eps}\widehat{(\be_X^{n+1})}_l=\frac{\widehat{(\fe^{\tau J}\be_Y^{n+1})}_l}{b}+\widehat{(\xi_X^n)}_l,\\
&\frac{\widehat{(\be_Y^{n+1})}_l-\widehat{(\be_Y^{n})}_l}{\Delta s}+\frac{il}{\eps}\widehat{(\be_Y^{n+1})}_l=\widehat{(F^n)}_l+\widehat{(\xi_Y^n)}_l,\quad l\in\bZ,\quad n\geq0.
\end{align*}
Noting that
\begin{equation*}
\widehat{(\fe^{\tau J}\be_Y^{n})}_l=\frac{1}{2}\left[ (I+iJ) \widehat{(\be_Y^{n})}_{l+1} + (I-iJ) \widehat{(\be_Y^{n})}_{l-1}\right],
\end{equation*}
we have
\begin{subequations}\label{thm1eq1}
\begin{align}
\widehat{(\be_{X}^{n+1})}_l\!=&p_l\widehat{(\be^{n}_{X})}_l \!
  +\!\frac{p_l\Delta s}{2b}\Big[\!(I\!+\!iJ)\widehat{(\be_{Y}^{n+1})}_{l\!+1}\!+\!(I\!-\!iJ)\widehat{(\be_{Y}^{n+1})}_{l\!-1}\!\Big]\!+\!p_l\Delta s\widehat{(\xi_{X}^{n})}_{l},\label{thm1eq1X1}\\
\widehat{(\be_{Y}^{n+1})}_l\!=&p_l\widehat{(\be_Y^{n})}_l
  +{p_l\Delta s}\left[\widehat{(F^n)}_l+\widehat{(\xi_Y^n)}_l\right],\quad l\in\bZ,\ n\geq0.\label{thm1eq1Y}
\end{align}
\end{subequations}
Let $n=0$ in the above relations, and then inserting (\ref{thm1eq1Y}) into (\ref{thm1eq1X1}), noting that
$\be_X^0(\tau)=\be_Y^0(\tau)=0,$ we get
\begin{align*}
  \widehat{(\be_{X}^{1})}_l=& \frac{p_l \Delta s^2}{2b} \!\left[ p_{l\!+1}(I+iJ)(\widehat{(F^{0})}_{l\!+1}\!+\widehat{(\xi_Y^{0})}_{l+1})
  + p_{l\!-1}(\widehat{(F^{0})}_{l\!-1}\!+\widehat{(\xi_Y^{0})}_{l-1})\right]\!+p_l\Delta s\widehat{(\xi_{X}^{0})}_{l},\\
  \widehat{(\be_{Y}^{1})}_l=&  p_l\Delta s \left[\widehat{(F^0)}_l+\widehat{(\xi_Y^0)}_l\right].
\end{align*}
Therefore, the formula (\ref{formula}) with $n=1$ is true. Let us assume (\ref{formula}) is true for $1\leq m\leq n$ and we check the case $m=n+1$.

Plugging (\ref{formulac}) into (\ref{thm1eq1Y}), we get
 \begin{align}
 \widehat{(\be_{Y}^{n+1})}_l=&\Delta s\sum_{j=1}^n p_l^{j+1}\left[\widehat{(F^{n-j})}_{l} + \widehat{(\xi_{Y}^{n-j})}_{l}\right]
 +{p_l\Delta s}\left[\widehat{(F^n)}_l+\widehat{(\xi_Y^n)}_l\right]\nonumber\\
 =&\Delta s\sum_{j=1}^{n+1}p_l^{j}\left[\widehat{(F^{n+1-j})}_{l} + \widehat{(\xi_{Y}^{n+1-j})}_{l}\right].\label{thm1eqnew}
 \end{align}
Hence  (\ref{formulac}) is checked. Next, plugging (\ref{thm1eqnew}) and (\ref{formulaa}) into (\ref{thm1eq1X1}) leads to
 \begin{align*}
 \widehat{(\be_{X}^{n+1})}_l=&\frac{\Delta s^2}{2b}\sum_{j=1}^{n} p_l\alpha_{j,l}(I+iJ)\left[\widehat{(F^{n-j})}_{l+1}+\widehat{(\xi_{Y}^{n-j})}_{l+1}\right] +
  \Delta s\sum_{j=0}^{n-1}p_l^{j+2}\widehat{(\xi_{X}^{n-1-j})}_{l} \nonumber\\
+&  \frac{\Delta s^2}{2b}\sum_{j=1}^{n} p_l\beta_{j,l}(I-iJ)\left[\widehat{(F^{n-j})}_{l-1}+\widehat{(\xi_{Y}^{n-j})}_{l-1}\right] +p_l\Delta s  \widehat{(\xi_{Y}^{n})}_{l}\\
+&\frac{p_l \Delta s^2}{2b}(I+iJ)\sum_{j=1}^{n+1}p_{l+1}^{j}\left[\widehat{(F^{n+1-j})}_{l+1} + \widehat{(\xi_{Y}^{n+1-j})}_{l+1}\right]\\
+&\frac{p_l \Delta s^2}{2b}(I-iJ)\sum_{j=1}^{n+1}p_{l-1}^{j}\left[\widehat{(F^{n+1-j})}_{l-1} + \widehat{(\xi_{Y}^{n+1-j})}_{l-1}\right].
 \end{align*}
First, we use the relation
$$
\sum_{j=0}^{n-1}p_l^{j+2}\widehat{(\xi_{X}^{n-1-j})}_{l} + p_l  \widehat{(\xi_{Y}^{n})}_{l} = \sum_{j=0}^{n} p_l^{j+1}\widehat{(\xi_{X}^{n-j})}_{l},
$$
which enables to recover the second term of \eqref{formulaa} for $n+1$ .
Second, we remark than $p_l\alpha_{0, l}=0$ so that the first term in the previous expression
of $\widehat{(\be_{X}^{n+1})}_l$ can be reformulated as
$$
\sum_{j=1}^{n} p_l\alpha_{j,l}\!\left[\!\widehat{(F^{n-j})}_{l\!+1}\!\!+\widehat{(\xi_{Y}^{n-j})}_{l\!+1}\!\right] \!\!
=\! \sum_{j=1}^{n+1} \frac{p_l^{j\!+1} p_{l\!+1}\! - p_l^2 p_{l+1}^j}{p_l-p_{l+1}}\! \left[\!\widehat{(F^{n+1\!-j})}_{l\!+1}\!\!+\widehat{(\xi_{Y}^{n+1\!-j})}_{l\!+1}\!\right]\!.
$$
Then, combining the latter with the term of the third line together with the following relation
\begin{eqnarray*}
\frac{p_l^{j\!+1} p_{l\!+1}\! - p_l^2 p_{l+1}^j}{p_l-p_{l+1}}+p_lp_{l+1}^j = \frac{p_l^{j\!+1} p_{l\!+1}\! - p_l p_{l+1}^{j+1}}{p_l-p_{l+1}}
=\alpha_{j,l},
\end{eqnarray*}
leads to the first term (with $n+1$) of \eqref{formulaa}. Similar arguments enable to recover the last term of \eqref{formulaa}.
Hence (\ref{formulaa}) holds for $n+1$ and the induction proof is done so that the formula (\ref{formula}) holds for all $n\geq1$.
\end{proof}
Now with Lemmas \ref{lm:1} and \ref{lm:2}, we give the proof of Theorem \ref{thm1}.

\textit{Proof of Theorem \ref{thm1}:}
We will proceed by an induction on $n$, by assuming that the following estimate holds for all $m=0, \dots, n$
\begin{equation}
\label{ineq_hyp}
\frac{1}{\eps}\|X_k^m-\bx_{k,0}\|_{H^1_\tau} + \|Y_k^m\|_{H^1_\tau} \leq \frac{1}{\eps}\|X_k(s_m)-\bx_{k,0}\|_{H^1_\tau}
+ \|Y_k(s_m)\|_{H^1_\tau} + 1,
\end{equation}
and using the relations \eqref{formula} on $\be_X^n$ and $\be_Y^n$.

Using Lemma \ref{lemma21}, this implies in particular $\|Y_k^m\|_{H^1_\tau}\leq C$. Then, we can deduce an estimate
for the nonlinear part $F^m$ given by \eqref{F_err}
\begin{equation}
\label{thm1F}
\|F^m\|_{H^1_\tau}\leq C\|\be_Y^m\|_{H^1_\tau}+\frac{C}{\eps}\|\be_X^m\|_{H^1_\tau},\quad 0\leq m\leq n,
\end{equation}
using again Lemma \ref{lemma21} and Sobolev embeddings.
%$$
%\|X_k^0-\bx_{k,0}\|_{H^1_\tau}\leq \eps C,\quad \|Y_k^0\|_{L^\infty_\tau}\leq \|\_k\|_{L^\infty_s\times H^1_\tau}.
%$$
%We prove the theorem by an induction on (\ref{bound}). We shall show that once (\ref{bound}) is valid, the error estimate also follows. By assuming that
%$$\|X_k^m-\bx_{k,0}\|_{H^1_\tau}\leq \eps C,\quad \|Y_k^m\|_{H^1_\tau}\leq\|Y_k\|_{L^\infty_s\times H^1_\tau}+1,\quad 0\leq m\leq n,$$
%holds, we are aiming to validate it for $m=n+1$.
%By triangle inequality, Lemma \ref{lemma21}, induction assumption on the boundedness of the numerical solution and standard product lemma, we have the following estimate for $F^m$ given by \eqref{F_err}
%\textcolor{red}{VERIFIER}
%\begin{equation}
%\label{thm1F}
%\|F^m\|_{H^1_\tau}\leq C\|\be_Y^m\|_{H^1_\tau}+\frac{C}{\eps}\|\be_X^m\|_{H^1_\tau},\quad 0\leq m\leq n.
%\end{equation}
Finally, in view of deriving an estimate for $e^{n+1}_X$ and $e^{n+1}_Y$,
we will use the following elementary lemma on the $p_l$ coefficients
\begin{lemma}
Let $p_l=\frac{1}{1+il\Delta s/\eps}, \forall l\in \bZ$, then the following estimates hold
\begin{eqnarray*}
|p_l|\leq 1,\quad \forall l\in\bZ,\quad \mbox{and}\quad \Delta s|p_l|\leq \eps,\quad \forall l\in\bZ\backslash\{0\},&&\nonumber\\
\Delta s\left|\frac{p_l^{j+1}p_{l\pm1}-p_l p_{l\pm1}^{j+1}}{p_l-p_{l\pm1}}\right|\leq C\eps,\quad \forall l\in\bZ,\forall j\in\bN^\star,&&
\end{eqnarray*}
where the constant $C>0$ is independent of $l,j,\Delta s$ and $\eps$.
\end{lemma}
\begin{proof}
The first inequality is immediate. For the second one, we have $(\Delta s/\eps) |p_l| = (\Delta s/\eps)/(1+\Delta s^2/\eps^2)^{1/2}$
and conclude owing that the function $f(a)=a/\sqrt{1+a^2}, a\geq 0$ is bounded (with $a=\Delta s/\eps$).

Let consider the last inequality. Denoting again $a=\Delta s/\eps$, we have
$$
a\left|\frac{p_l^{j+1}p_{l\pm1}-p_l p_{l\pm1}^{j+1}}{p_l-p_{l\pm1}}\right| = a \left|\frac{p_l p_{l\pm1} }{p_l-p_{l\pm1}}\right|
| p_l^{j} - p_{l\pm1}^{j}| = | p_l^{j} - p_{l\pm1}^{j}| \leq 2
$$
since $a \left|\frac{p_l p_{l\pm1} }{p_l-p_{l\pm1}}\right|=1$ and $|p_l|\leq 1$.
\end{proof}
Thanks to the previous tools, we get from the error formula \eqref{formulaa} at $(n+1)$,
%\begin{align*}
%  \widehat{(\be_{X}^{n+1})}_l  =& \frac{\Delta s^2}{2b}\sum_{j=0}^{n}\alpha_{j+1, l}(I+iJ)\left[\widehat{(F^{n-j})}_{l+1}+\widehat{(\xi_{X}^{n-j})}_{l+1}\right] + \Delta s\sum_{j=0}^{n} p_l^{j+1} \widehat{(\xi_{X}^{n-j})}_{l}\\
%  +&\frac{\Delta s^2}{2b}\sum_{j=0}^{n}\beta_{j+1, l}(I-iJ)\left[\widehat{(F^{n-j})}_{l-1}+\widehat{(\xi_{X}^{n-j})}_{l-1}\right].
%\end{align*}
%The previous lemma gives estimate on $\alpha_{j,l}$ and $\beta_{j,l}$ so that we get
\begin{align}
  |\widehat{(\be_{X}^{n+1})}_l|
  \leq& \Delta s\sum_{j=0}^{n}|\widehat{(\xi_{X}^{n-j})}_{l}|\nonumber\\
 &\hspace{-1cm}+C\eps\Delta s\sum_{j=0}^{n}\left[|\widehat{(F^{n-j})}_{l+1}|+|\widehat{(\xi_{X}^{n-j})}_{l+1}+|\widehat{(F^{n-j})}_{l-1}|+|\widehat{(\xi_{X}^{n-j})}_{l-1}|\right].
 \label{eq:enp1}
\end{align}
Taking the square of \eqref{eq:enp1} and using Cauchy-Schwarz inequality lead to
\begin{align*}
   |\widehat{(\be_{X}^{n+1})}_l|^2
  \leq & C\Delta s \sum_{j=0}^{n}| \widehat{(\xi_{X}^{n-j})}_{l}|^2\\
  +& C\eps^2\Delta s \sum_{j=0}^{n}\left[  |\widehat{(F^{n-j})}_{l+1}|^2+
  |\widehat{(\xi_{X}^{n-j})}_{l+1}|^2 + |\widehat{(F^{n-j})}_{l-1}|^2 + |\widehat{(\xi_{X}^{n-j})}_{l-1}|^2\right].
\end{align*}
Then, to get $H^1_\tau$ estimate, we multiply by $(1, l^2)$, sum on $l\in\bZ$ and add the two resulting equations so that, using Parseval identity, we obtain
\begin{align*}
  \|\be_{X}^{n+1}\|_{H^1_\tau}^2
  \leq & C\Delta s \sum_{j=0}^{n}\|\xi_{X}^{n-j}\|_{H^1_\tau}^2+ C\eps^2\Delta s \sum_{j=0}^{n}\left[\|F^{n-j}\|_{H^1_\tau}^2+\|\xi_{Y}^{n-j}\|^2_{H^1_\tau}\right].
\end{align*}
By the error formula (\ref{formulac}), we get by similar arguments
$$\|\be_{Y}^{n+1}\|_{H^1_\tau}^2
  \leq C\Delta s\sum_{j=0}^{n}\left[\|F^{n-j}\|_{H^1_\tau}^2
  +\|\xi_{Y}^{n-j}\|_{H^1_\tau}^2\right].$$
Combining the two, we then have
\begin{align*}
\frac{1}{\eps^2}\|\be_{X}^{n+1}\|_{H^1_\tau}^2+\|\be_{Y}^{n+1}\|_{H^1_\tau}^2\leq&
C\Delta s\sum_{j=0}^{n}\left[\frac{1}{\eps^2}\|\xi_{X}^{j}\|_{H^1_\tau}^2+\|\xi_{Y}^{j}\|_{H^1_\tau}^2\right] +C\Delta s\sum_{j=0}^{n}\|F^{j}\|_{H^1_\tau}^2.
\end{align*}
 Now inserting estimates  \eqref{error_trunc} and (\ref{thm1F}), we get
\begin{align*}
\frac{1}{\eps^2}\|\be_{X}^{n+1}\|_{H^1_\tau}^2+\|\be_{Y}^{n+1}\|_{H^1_\tau}^2\leq&
C\Delta s^2 +C\Delta s\sum_{j=0}^{n}\left[\frac{1}{\eps^2}\|\be_{X}^{j}\|_{H^1_\tau}^2+\|\be_{Y}^{j}\|_{H^1_\tau}^2
  \right],
\end{align*}
and we conclude by using discrete Gronwall's lemma to get
$$
\frac{1}{\eps}\|\be_{X}^{n+1}\|_{H^1_\tau}+\|\be_{Y}^{n+1}\|_{H^1_\tau}\leq C\Delta s.
$$
As long as \eqref{ineq_hyp} is true, we then deduce
$$
\frac{1}{\eps}\|{X}^{n+1}_k-\bx_{k,0}\|_{H^1_\tau}+\|{Y}_k^{n+1}\|_{H^1_\tau}\leq \frac{1}{\eps}\|{X}_k(s_{n+1})-\bx_{k,0}\|_{H^1_\tau}+\|{Y}_k(s_{n+1})\|_{H^1_\tau} + C\Delta s,
$$
which completes the induction proof by considering $\Delta s\leq \Delta s_0 = 1/C$.
\qed

%\begin{remark}
%If we change the IMEX1 with
%$$\frac{X_k^{n+1}-X_k^n}{\Delta s}+\frac{1}{\eps}\partial_\tau X_k^{n+1}=\frac{\fe^{\tau J}}{b_k}Y_k^{n},$$
%the practical performance of the scheme based on our numerical experience is similar. That is to say we observed the same the error bound as (\ref{thm1error}) from the numerical results. However, in the error estimate, we would lose estimate $\|\xi_X^n\|_{H^1_\tau}=O(\eps\Delta t)$ and then the rigorous proof is not clear to us.
%\end{remark}

\subsection{Second order numerical scheme}

%\textcolor{red}{The second order scheme is somehow strange... ! Why $\bar{X}^{n+1/2}$ instead of
%$X^{n+1/2}$ ? The first stage of the scheme is not correct. }

A second order scheme (IMEX2) could be written down as follows.  For $n\geq0,$
\begin{subequations}\label{fd2}
\begin{align}
\frac{X_k^{n+1}-X_k^n}{\Delta s}+\frac{1}{2\eps}\partial_\tau \left(X_k^{n+1}
+X_k^n\right)=&\frac{\fe^{\tau J}}{b_k}Y_k^{n+1/2},\\
\frac{Y_k^{n+1}-Y_k^n}{\Delta s}+\frac{1}{2\eps}\partial_\tau \left(Y_k^{n+1}
+Y_k^n\right)=&
\frac{1}{\eps}\left(\frac{b(\overline{X}_k^{n+1/2})}{b_k}-1\right)JY^{n+1/2}_k\nonumber\\
&+\fe^{-J\tau}\frac{\bE^\eps(t_{n+1/2},\overline{X}_k^{n+1/2})}{b_k}.
\end{align}
\end{subequations}
with
\begin{align*}
&\overline{X}_k^{n+1/2}=\frac{X_k^n+X^{n+1}_k}{2},\\
&\frac{X_k^{n+1/2}-X_k^n}{\Delta s/2}+\frac{1}{\eps}\partial_\tau X_k^{n+1/2}=\frac{\fe^{\tau J}}{b_k}Y_k^n,\\
&\frac{Y_k^{n+1/2}-Y_k^n}{\Delta s/2}+\frac{1}{\eps}\partial_\tau Y_k^{n+1/2}=
\frac{1}{\eps}\left(\frac{b(X_k^n)}{b_k}-1\right)JY^n_k+\fe^{-J\tau}\frac{\bE^\eps(t_{n+1/2},X_k^{n+1/2})}{b_k}.
\end{align*}
The above IMEX2 scheme is also explicit. From practical results (as can be seen in the next section), we observe that the IMEX2 with the 3rd order prepared initial data (derived in Appendix \ref{appendix}) gives second order uniform accuracy. However, the rigorous error estimates would be more involved than the first order IMEX1 scheme
and it is still under-going. We will address it in a future work.

\subsection{A strategy for the Vlasov-Poisson case}
\label{vpcase}
When the Vlasov equation (\ref{eq:vp1}) is coupled to Poisson equation, the electric field $\bE^\eps$ is a self-consistent field
and the problem becomes nonlinear.  Under PIC discretisation, $\bE^\eps$ in (\ref{charact}) is given by
\begin{equation}\label{diracE}
\nabla_\bx\cdot\bE^\eps(t,\bx)=\sum_{k=1}^{N_p}\omega_k\delta(\bx-\bx_k(t)),\quad t\geq0,\ \bx\in\bR^2.
\end{equation}
Hence by under the scale of time $s_j=b_jt$, the electric field evaluated at
one particular particle $\bx_j(t)$ ($1\leq j\leq N_p$) solves (as we see in (\ref{charact2}))
$$
\nabla_\bx\cdot\bE^\eps(s_j/b_j,\bx_j(t))=\sum_{k=1}^{N_p}\omega_k\delta(\bx_j(t)-\bx_k(s_j/b_j)).
$$
Each particle carries its own frequency and now all the particle are coupled to each other through Poisson equation, which as a result mixes all the frequencies. Thus, (\ref{charact2}) is a multiple-frequency system with a large number $N_p$ of degrees  \cite{UAmutiple} and the proposed two-scale formulation is not rigorously working. Here, we give a practical
strategy that works well  based on our numerical experiments.

Note the above strategy relies on the key confinement property (see Lemma \ref{lemma1})
$$
b(\tilde{\bx}_k(s))-b_k=C(s)\eps,\quad s>0, \;\; 1\leq k\leq N_p,
$$
as well as the two-scale version in Lemma \ref{lemma2}. In order to have a better control of $C(s)$ in the oscillatory case, we consider the scale of time (\ref{scale time}) dynamically. Discretise time with $\Delta t>0$ and denote $t_n=n\Delta t$. For $n\geq0$ and $t_n\leq t\leq t_{n+1}$, define
\begin{equation}\label{scale time2}
b_k^n=b(\bx_k(t_n)),\quad t=t_n+\frac{s}{b_k^n},\quad \tilde{\bx}_k(s)=\bx_k(t),
\end{equation}
and we solve
 \begin{subequations}\label{charact4}
\begin{align}
   & \dot{\tilde{\bx}}_k(s)=\frac{\tilde{\bv}_k(s)}{b_k^n}, \\
  &\dot{\tilde{\bv}}_k(s)=\frac{b(\tilde{\bx}_k(s))}{\eps b_k^n}J\tilde{\bv}_k(t)+\frac{\bE^\eps(t_n+s/b_k^n,\tilde{\bx}_k(s))}{b_k^n},\quad 0<s\leq b_k^n\Delta t, \\
   & \tilde{\bx}_k(0)=\bx_{k}(t_n),\quad \tilde{\bv}_k(0)=\bv_{k}(t_n),
\end{align}
\end{subequations}
for one step with $\Delta s=b_k^n\Delta t$.
Again we isolate the leading order oscillation term, filter out this main oscillation (\ref{filter}) and then consider the two-scale formulation but leave the high-frequency character of the electric field part alone \cite{UAVP2d}.
We then obtain %in the time variable part of
\begin{subequations}\label{vp tau}
\begin{align}
\label{vp_Xtau}
&\partial_s X_k+\frac{1}{\eps}\partial_\tau X_k=\frac{\fe^{\tau J}}{b_k^n}Y_k, \quad 0<s\leq\Delta s,\\
\label{vp_Ytau}
&\partial_s Y_k+\frac{1}{\eps}\partial_\tau Y_k=\frac{1}{\eps}\left(\frac{b(X_k)}{b_k^n}-1\right)JY_k+
\fe^{-\tau J}\frac{\bE^\eps(t_n+s/b_k^n,X_k)}{b_k^n}.
\end{align}
\end{subequations}
For the initial data, we formally choose $X_k(0,\tau)=X_k^{1st}(\tau)$, $Y_k(0,\tau)=Y_k^{1st}(\tau)$ given by \eqref{Xfirst} and \eqref{Yfirst} by replacing $\bx_{k,0}$ and $\bv_{k,0}$ with respectively $\bx_k(t_n)$ and $\bv_k(t_n)$.
This initial data enables to offer second order uniform accuracy when an exponential integrator scheme (as in \cite{UAVP4d})
is used for \eqref{vp tau}.

%One can consider the second order initial data for $X_k(0, \tau)$ (\ref{Xsecond}) but only
%the first order for $Y_k(0, \tau)$, since in the higher order data for $Y$, e.g. (\ref{Ysecond}),  the time derivatives of $\bE^\eps$ will be involved but the fast time-scale has not been separated in $\bE^\eps$. Concerning this issue, we turn to use an exponential integrator scheme \cite{UAVP4d}.

We shall briefly derive the scheme. For the simplicity of notations, we put \eqref{vp_Xtau}-\eqref{vp_Ytau}
into the following compact form:
\begin{align}\label{eq: U}
\partial_sU_k(s,\tau)+\frac{1}{\eps}\partial_\tau U_k(s,\tau)=F_k(s,\tau),\quad
0<s\leq\Delta s,\ \tau\in\bT,
\end{align}
where we denote
$$U_k(s,\tau)=\binom{X_k(s,\tau)}{Y_k(s,\tau)},\quad F_k(s,\tau)=
\binom{F_{1,k}(s,\tau)}{F_{2,k}(s,\tau)},$$
with
\begin{subequations}
\begin{align}
&F_{1,k}(s,\tau)=\frac{\fe^{\tau J}}{b_k^n}Y_k(s,\tau),\\
&F_{2,k}(s,\tau)=\frac{1}{\eps}\left(\frac{b(X_k(s,\tau))}{b_k^n}-1\right)JY_k(s,\tau)+
\fe^{-\tau J}\frac{\bE^\eps(t_n+s/b_k^n,X_k(s,\tau))}{b_k^n}.
\end{align}
\end{subequations}
Applying Fourier transform in $\tau$ on (\ref{eq: U})
$$U_k(s,\tau)=\sum_{l=-N_\tau/2}^{N_\tau/2-1}\widehat{(U_k)}_l(s)\fe^{il\tau},\qquad
F_k(s,\tau)=\sum_{l=-N_\tau/2}^{N_\tau/2-1}\widehat{(F_k)}_l(s)\fe^{il\tau},$$
and then by Duhamel's principle  from $0$ to $\Delta s$ ($n\geq0$),
\begin{align*}
\widehat{(U_k)}_l(\Delta s)&=\fe^{-\frac{il\Delta s}{\eps}}\widehat{(U_k)}_l(0)
+\int_0^{\Delta s}\fe^{-\frac{il}{\eps}(\Delta s-\theta)}\widehat{(F_k)}_l(\theta)d\theta.
\end{align*}

A first order uniformly accurate scheme, shorted as EI1 in the following, is obtained as,
\begin{align*}
\widehat{(U_k)}_l(\Delta s)
&\approx\fe^{-\frac{il\Delta s}{\eps}}\widehat{(U_k)}_l(0)
+\int_0^{\Delta s}\fe^{-\frac{il}{\eps}(\Delta s-\theta)}\widehat{(F_k)}_l(0)
d\theta\\
&\approx\fe^{-\frac{il\Delta s}{\eps}}\widehat{(U_k)}_l(0)
+p^{E}_l\widehat{(F_k)}_l(0),
\end{align*}
where
\begin{equation*}
p^E_l:=\int_0^{\Delta s}\fe^{-\frac{il}{\eps}(\Delta s-\theta)}d\theta=\left\{\begin{split}
&\frac{i\eps}{l}\left(\fe^{-\frac{il\Delta s}{\eps}}-1\right),\quad l\neq0,\\
&\Delta s,\qquad\qquad\qquad\quad\, l=0.\end{split}\right.
\end{equation*}

A second order scheme, shorted as EI2, is given as,
\begin{align*}
\widehat{(U_k)}_l(\Delta s)
&\approx\fe^{-\frac{il\Delta s}{\eps}}\widehat{(U_k)}_l(0)
+\int_0^{\Delta s}\fe^{-\frac{il}{\eps}(\Delta s-\theta)}\left(\widehat{(F_k)}_l(0)
+\theta\frac{d}{ds}\widehat{(F_k)}_l(0)\right)d\theta\\
&\approx\fe^{-\frac{il\Delta s}{\eps}}\widehat{(U_k)}_l(0)
+p^E_l\widehat{(F_k)}_l(s_n)
+q^E_l\frac{1}{\Delta s}\left(\widehat{(F_k^*)}_l-\widehat{(F_k)}_l(0)\right),
\end{align*}
where
\begin{equation*}
q^E_l:=\int_0^{\Delta s}\fe^{-\frac{il}{\eps}(\Delta s-\theta)}\theta d\theta=
\left\{\begin{split}
&\frac{\eps}{l^2}\left(\eps-\eps\fe^{-\frac{il\Delta s}{\eps}}-il\Delta s\right),\quad l\neq0,\\
&\frac{\Delta s^2}{2},\qquad\qquad\qquad\qquad\qquad\ \, l=0.\end{split}\right.
\end{equation*}
and
\begin{align*}
\displaystyle &\widehat{(U_k)}_l^*=\fe^{-\frac{il\Delta s}{\eps}}\widehat{(U_k)}_l(0)
+p^E_l\widehat{(F_k)}_l(0),\\
\displaystyle &F_k^*(\tau)=
\displaystyle\binom{\fe^{\tau J}Y_k^*(\tau)/b_k^n}{\frac{1}{\eps}\left(b(X_k^*(\tau))/b_k^n-1\right)JY_k^*(\tau)+
\fe^{-\tau J}\bE^\eps(t_n+\Delta s/b_k^n,X_k^*(\tau))/b_k^n}.
\end{align*}

Suppose now we have computed numerically $X_k(\Delta s,\tau)$ as the two-scale solution for system (\ref{charact4}), we update the electric field for the next time level $t_{n+1}$ as
$$
\nabla_\bx \cdot\bE^\eps(t_{n+1},\bx)=\sum_{k=1}^{N_p}\omega_k\delta(\bx-X_k(\Delta s,\Delta s/\eps)),\quad \ \bx\in\bR^2.
$$

 \section{Numerical results} \label{sect:ne}
% \textcolor{red}{Precise what "second order data" means: does it refer to $(X^{2nd}, Y^{1st})$ ? }

This section is devoted to numerical illustrations of the numerical schemes introduced above.
We consider (\ref{eq:1}) with the following initial data
\begin{equation}\label{KH}
f_0(\bx,\bv)=\frac{1}{4\pi}\left(1+\sin(x_2)+\eta\cos(kx_1)\right)\left(\fe^{-\frac{(v_1+2)^2+v_2^2}{2}}+
 \fe^{-\frac{(v_1-2)^2+v_2^2}{2}}\right),
  \end{equation}
with $\bx=(x_1, x_2)$ and $\bv=(v_1, v_2)$ and the non-homogeneous  magnetic field
$$
b(\bx)=1+\sin(x_1)\sin(x_2)/2,
$$
to test convergence order. The spatial domain is $\bx=(x_1, x_2) \in\Omega=[0,2\pi/k]\times[0,2\pi]$ for some $k,\eta>0$. We choose
$\eta=0.05,k=0.5$ and discretise $\Omega$ with 64 points in $x_1$-direction and $32$ points in $x_2$-direction.
As a diagnostic, we consider the following two quantities:
\begin{align*}
&\rho^\eps(t,\bx)=\int_{\bR^2}f^\eps(t,\bx,\bv)d\bv,\qquad \bx\in\Omega,\\
&\rho_\bv^\eps(t,\bx)=\int_{\bR^2}|\bv|^2f^\eps(t,\bx,\bv)d\bv.
\end{align*}
We then compute the relative errors of the different numerical schemes with respect to
$\rho^\eps$ and $\rho^\eps_\bv$ at the final time $t_f=1$ in maximum space norm.
We numerically solve (\ref{eq:1}) with two configurations. For the first one,
we consider an external electric field given by
\begin{align*}
&\bE^\eps(t,\bx)=\binom{E_1(\bx)}{E_2(\bx)}(1+\sin(t)/2),\\
&E_1(\bx)=\cos(x_1/2)\sin(x_2)/2,\quad E_2(\bx)=\sin(x_1/2)\cos(x_2),
\end{align*}
which will be addressed as `given E' in the numerical results. For the second case,
we consider the nonlinear Vlasov-Poisson equation (\ref{eq:1})-(\ref{poisson}).
The reference solution is obtained by using a fourth order Runge-Kutta method on
the original problem (\ref{charact}) with step size $\Delta t=10^{-6}$.

For the PIC method, we choose $N_p=204800$ particles and
the projection of the particles on the uniform spatial grid is done by quintic splines.
The time step $\Delta s$ is determined by fixing $\Delta t$ (recall the relation $s=b_k t$ and $\Delta s=b_k \Delta t$)
so that after $N$ time steps (such that $t_f=N\Delta t$, every particles stops at the same time $t_f$.
Finally, we denote by $N_\tau$ the number of points in the $\tau$-direction: $\Delta \tau = 2\pi / N_\tau$.

In the sequel, %first order initial data will refer to $(X_k^{1st}, \bv_{k,0})$ with $X_k^{1st}$ given by  \eqref{Xfirst} whereas
second order initial data will refer to $(X_k^{2nd}, Y_k^{1st})$ with $X_k^{2nd}$ given by  \eqref{Xsecond}
and $Y_k^{1st}$ given by  \eqref{Yfirst}, whereas third order initial data will refer to $(X_k^{3rd}, Y_k^{2nd})$ with $X_k^{3rd}$ given by  \eqref{Xthird}
and $Y_k^{2nd}$ given by  \eqref{Ysecond}.

{\bf External electric field.\\}
We first study the `given E' case. In Figure  \ref{fig:VE1}, the errors (in $L^\infty_t$ norm) in time
of the IMEX1 scheme in $\rho^\eps$ with second order initial data %$(X_k^{2nd},Y_k^{1st})$
are displayed for different values of $\eps$.
As shown in the numerical analysis (see Proposition \ref{prop:1st} and Theorem \ref{thm1error}),
the scheme IMEX1 has uniform first order accuracy in time. Moreover, we can observe that
the error decreases as $\eps$ goes to zero, in agreement with theoretical results.

In Figure \ref{fig:VE1super}, the errors  (in $L^\infty_t$ norm) in time of the IMEX1 scheme with second order initial data
%$(X_k^{2nd},Y_k^{1st})$
is plotted regarding the quantities $\rho^\eps/\eps$ and $\rho^\eps_\bv$. We can observe the curves are almost superimposed
confirming the theoretical error estimates derived previously.

The influence of the discretization in the $\tau$ direction of the IMEX1
(using $\Delta t=10^{-5}$) is presented in Figure \ref{fig:VE1tau}. We computed the difference between
the numerical solution obtained with several $N_\tau$ and the one using $N_\tau=64$, for the two quantities
$\rho^\eps$ and $\rho^\eps_\bv$. The discretization error in $\tau$ has a  spectral behavior with respect to $N_\tau$,
and when $\eps$ becomes small, we observe that the method reaches machine accuracy with very few number of grids
(typically $N_\tau=8$ is sufficient when $\eps=10^{-3}$).

Next, we study the convergence of the second order IMEX2 scheme with third order initial data.  %$(X_k^{3rd},Y_k^{2nd})$.
In Figures \ref{fig:VE3}-\ref{fig:VE2}, we can see second order uniform
accuracy of the scheme with respect to $\eps$, regarding
both $\rho^\eps/\eps$ and $\rho^\eps_\bv$.

We also study the convergence rate of the Vlasov equation (\ref{eq:1}) to the asymptotic model (derived in
Appendix \ref{appendix2})
on the characteristics level when $\eps\to0$.
To do so, we measure the difference between the solution of (\ref{eq:1}) for several $\eps$ and the one
obtained with $\eps=10^{-4}$. In Figure \ref{fig:limit}, we show the convergence of the model (\ref{eq:1}) in the limit regime
(with $\Delta t=5\time 10^{-3}$),
for which the rate is equal to one.
Finally, we study the dynamics of the solution for a fixed $\eps=0.1$,
aiming to see the effect from the non-constant magnetic field.
The quantity $\rho^\eps(t,\bx)$ is plotted as a function of $\bx$ at different times
in Figure \ref{fig:2d1} (with $\Delta t=5\time 10^{-3}$).

{\bf Vlasov-Poisson case.\\}
For the Vlasov-Poisson case, we apply the strategy presented in subsection \ref{vpcase}, namely the dynamical scaling EI2 scheme with initial data $(X_k^{1st},Y_k^{1st})$.
In Figure \ref{fig:VP2update}, we show the convergence results in time regarding both  $\rho^\eps$ and $\rho^\eps_\bv$.
For a long-time diagnostic test of the scheme, we consider the energy of the Vlasov-Poisson equation which is conserved as
$$H(t):=\frac{1}{2}\int_{\bR^2}\int_{\Omega}|\bv|^2f^\eps(t,\bx,\bv)d\bx d\bv+\frac{1}{2}\int_\Omega|\bE^\eps
(t,\bx)|^2d\bx\equiv H(0).$$
We compute the numerical energy $H(t)$ by the EI2 scheme with $\Delta t=0.05,\ N_\tau=16$.
In Figure \ref{fig:energy}, we show the relative energy error $|H(t)-H(0)|/H(0)$ till $t_f=100$
for different values of $\varepsilon$.
For the Vlasov-Poisson case, the proposed strategy shows a promising performance.

Finally, we study the convergence rate of the Vlasov-Poisson equation (\ref{eq:1})-\eqref{poisson} to the asymptotic model when $\eps\to0$.
We proceed as in the linear case to plot in Figure \ref{fig:limit}
the convergence of the model (\ref{eq:1}) towards the limit regime. Here again, the rate is close to one.

\begin{figure}[t!]
$$\begin{array}{cc}
\psfig{figure=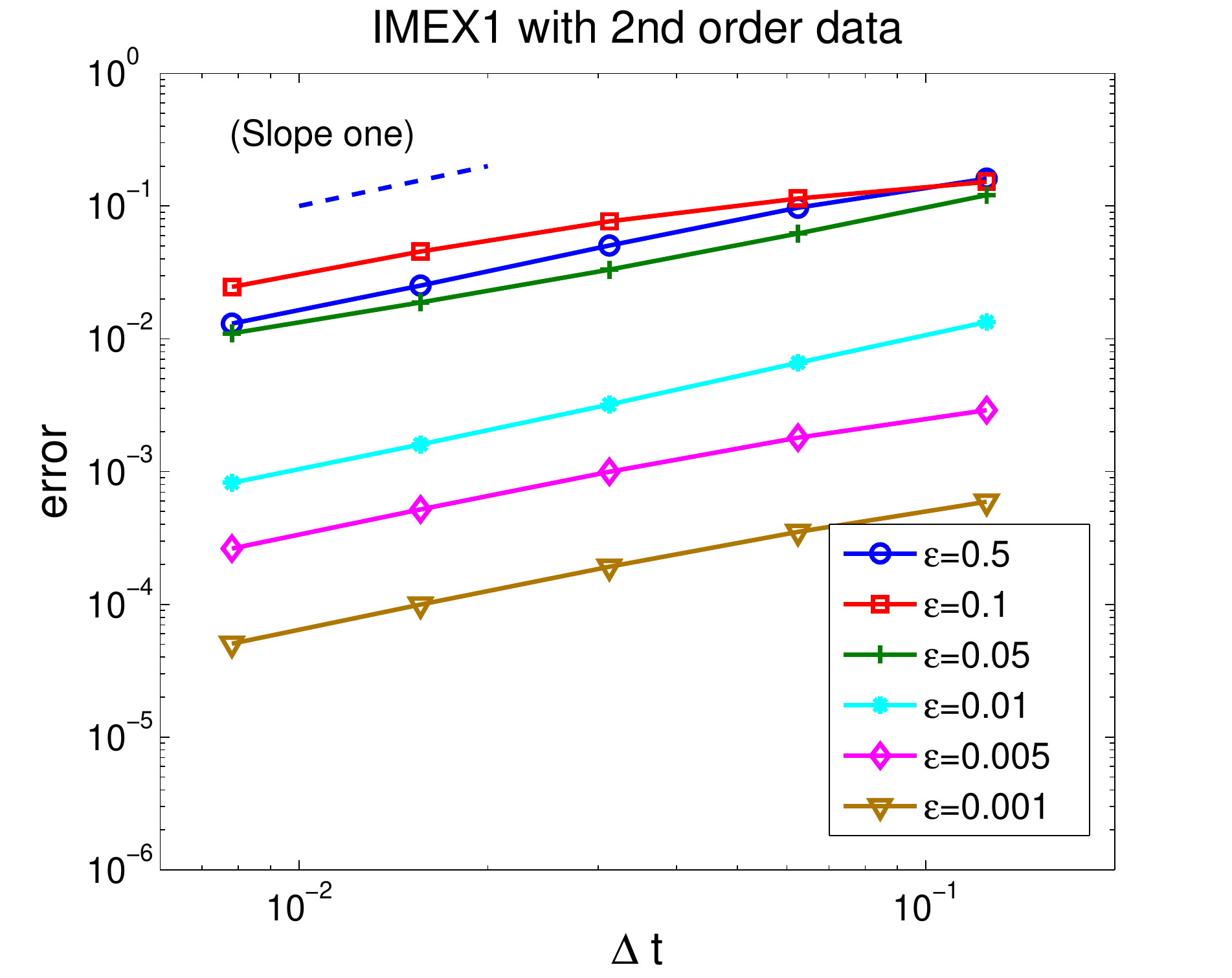,height=5cm,width=6.5cm}&\psfig{figure=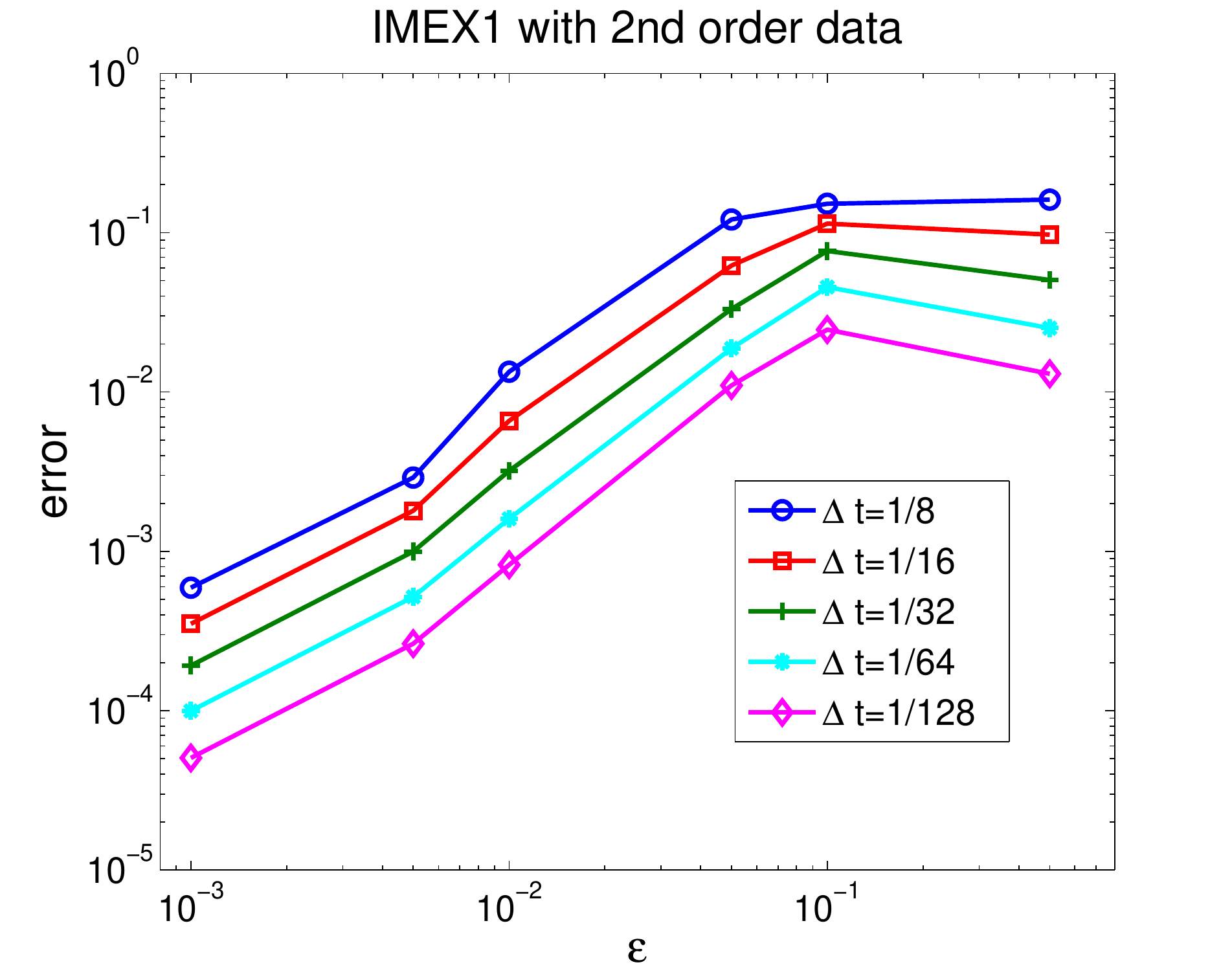,height=5cm,width=6.5cm}
\end{array}$$
\caption{Temporal error of IMEX1 with second order initial data $(X_k^{2nd},Y_k^{1st})$ for given $E$ case: relative maximum error in $\rho^\eps$. }\label{fig:VE1}
\end{figure}

\begin{figure}[t!]
$$\begin{array}{cc}
\psfig{figure=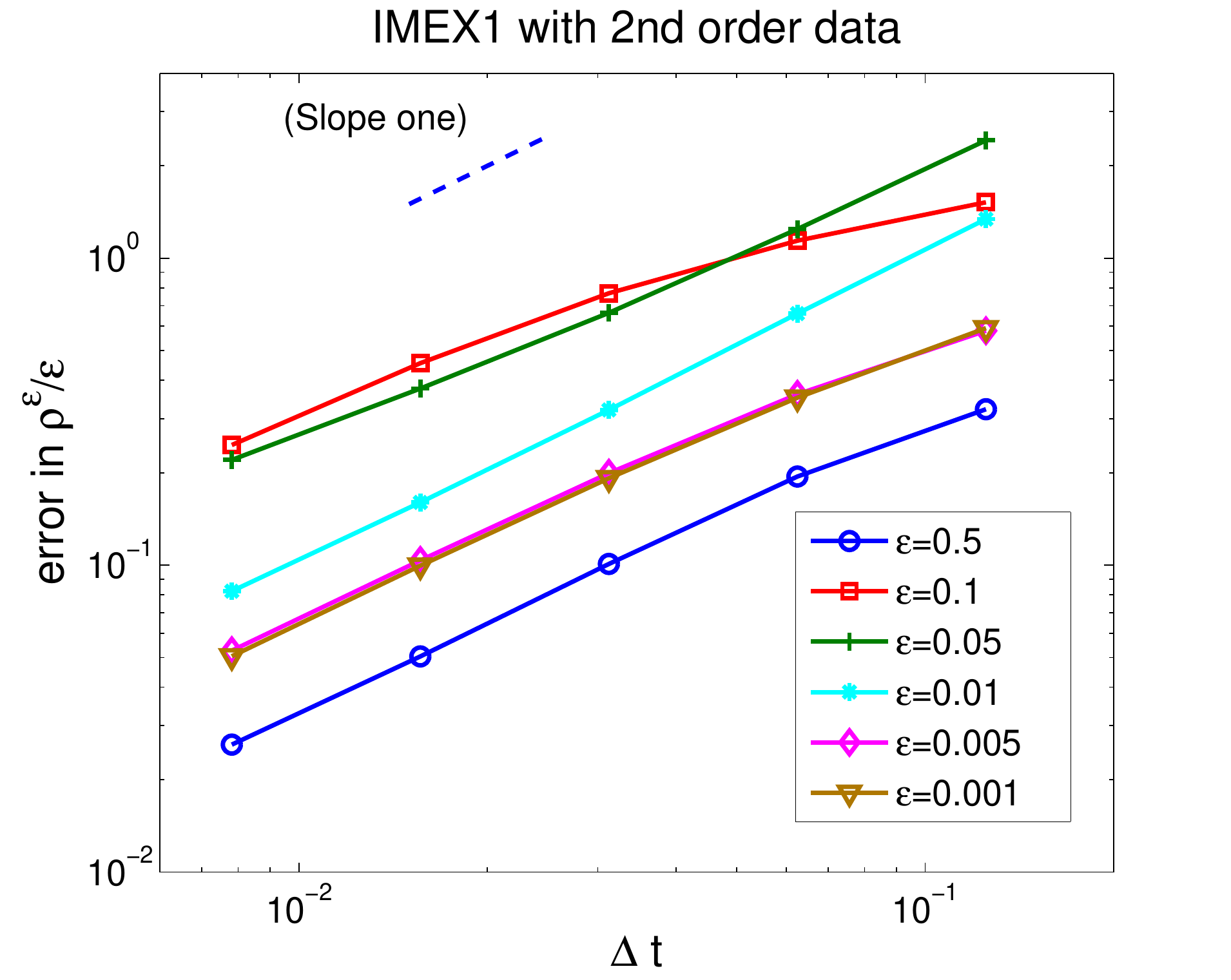,height=5cm,width=6.5cm}&\psfig{figure=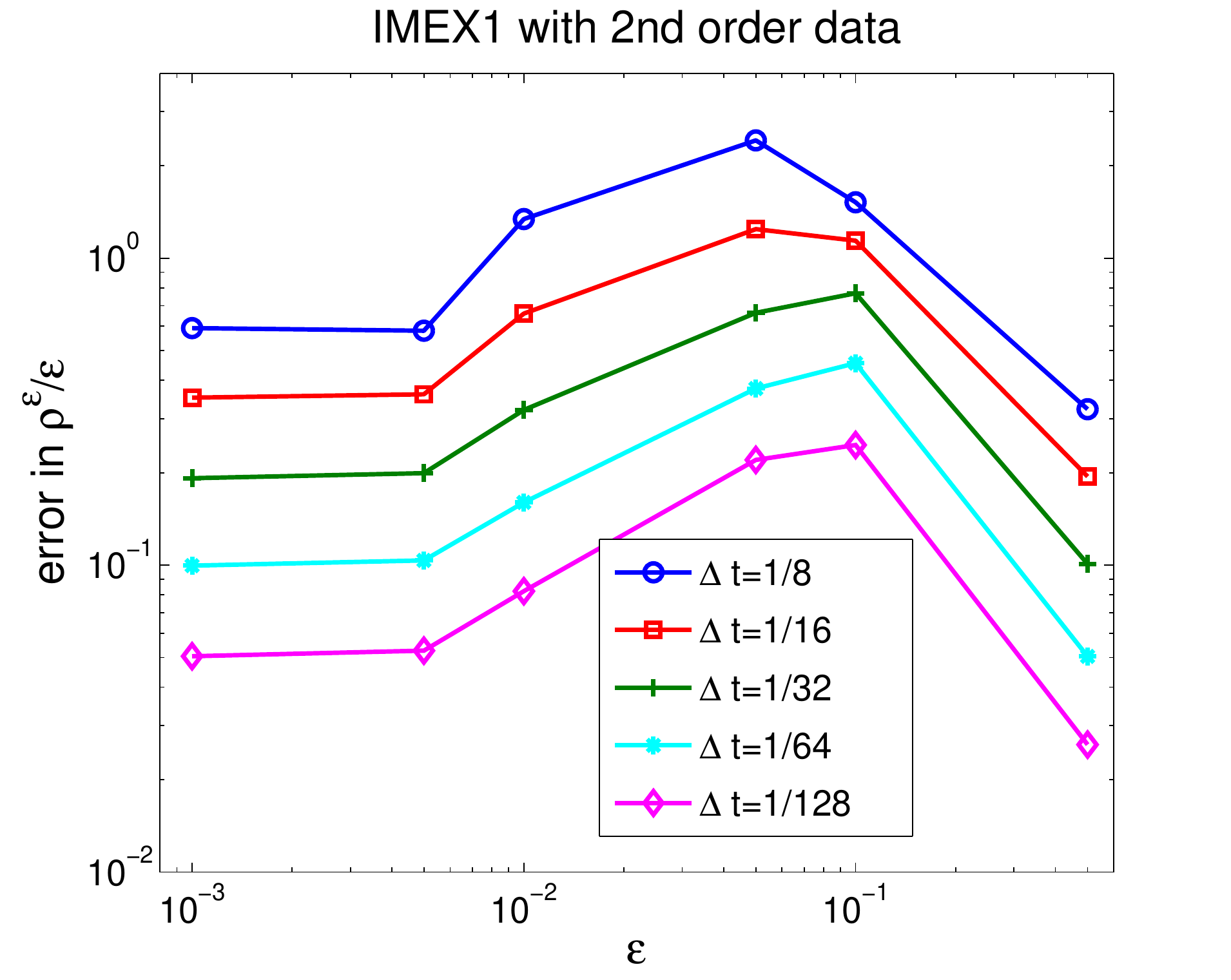,height=5cm,width=6.5cm}\\
\psfig{figure=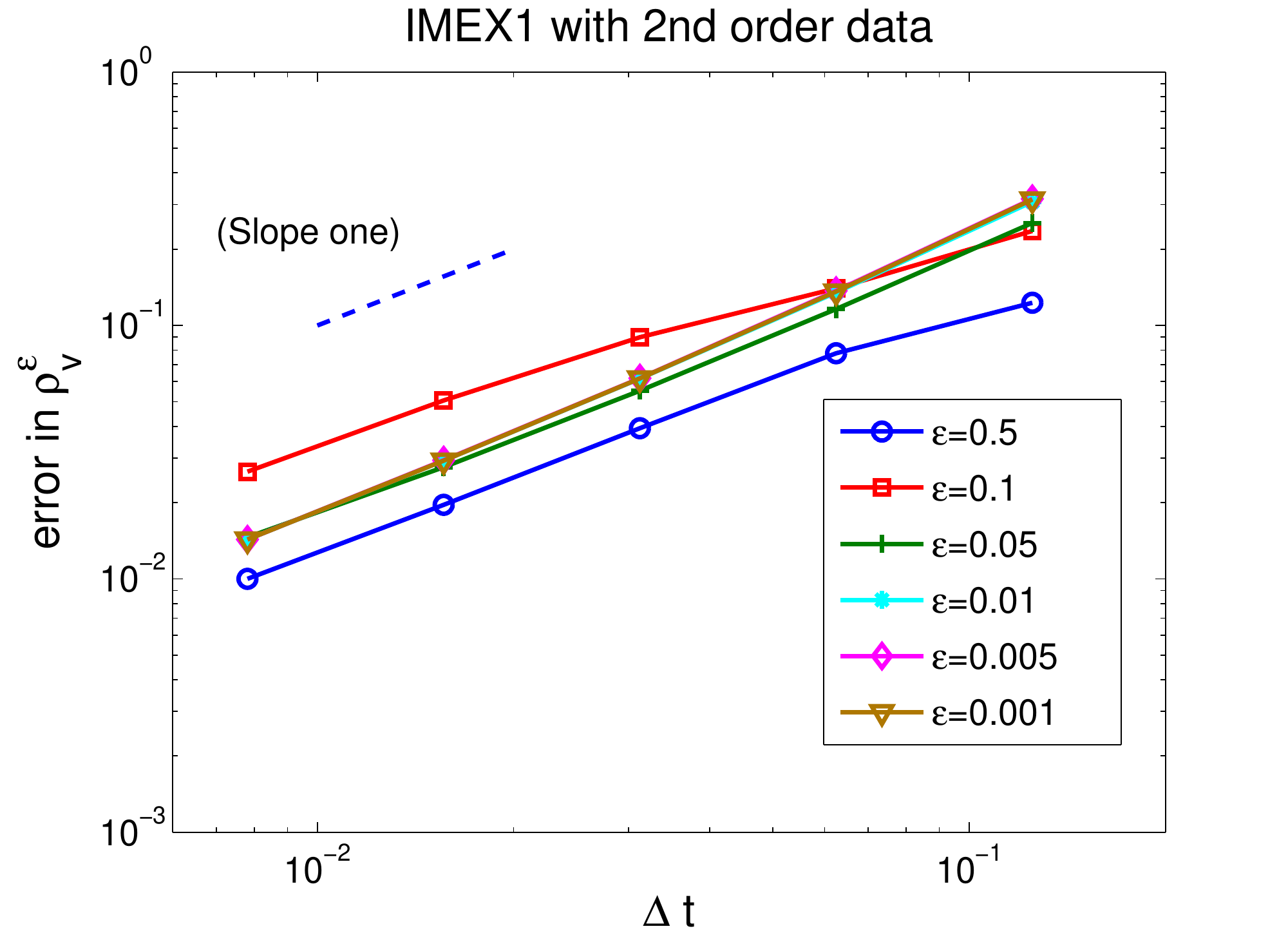,height=5cm,width=6.5cm}&\psfig{figure=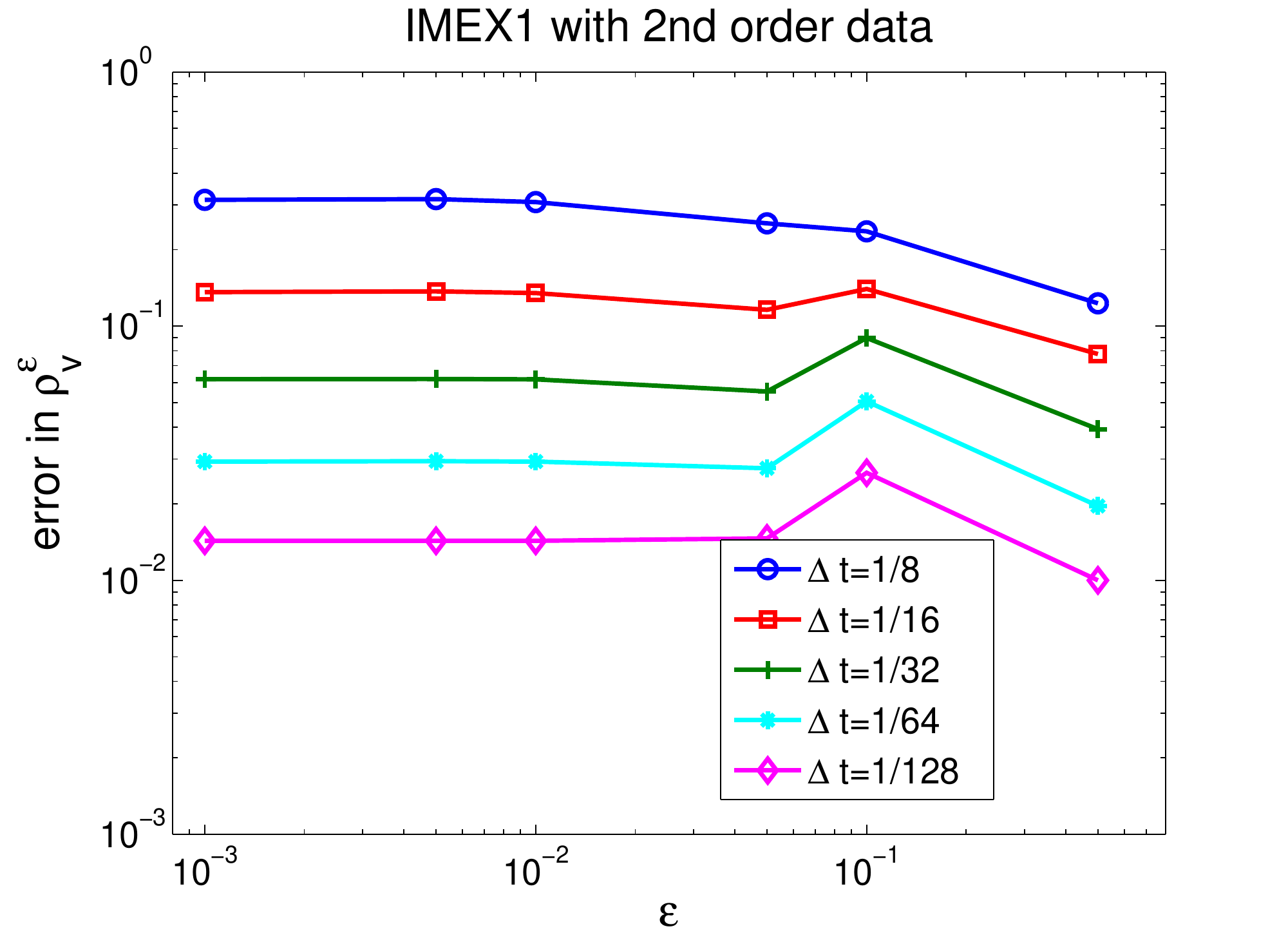,height=5cm,width=6.5cm}
\end{array}$$
\caption{Temporal error of IMEX1 with second order initial data $(X_k^{2nd},Y_k^{1st})$ for given $E$ case: relative maximum error in $\rho^\eps/\eps$ (first row) and in $\rho_\bv^\eps$ (second row).}\label{fig:VE1super}
\end{figure}

\begin{figure}[t!]
$$\begin{array}{cc}
\psfig{figure=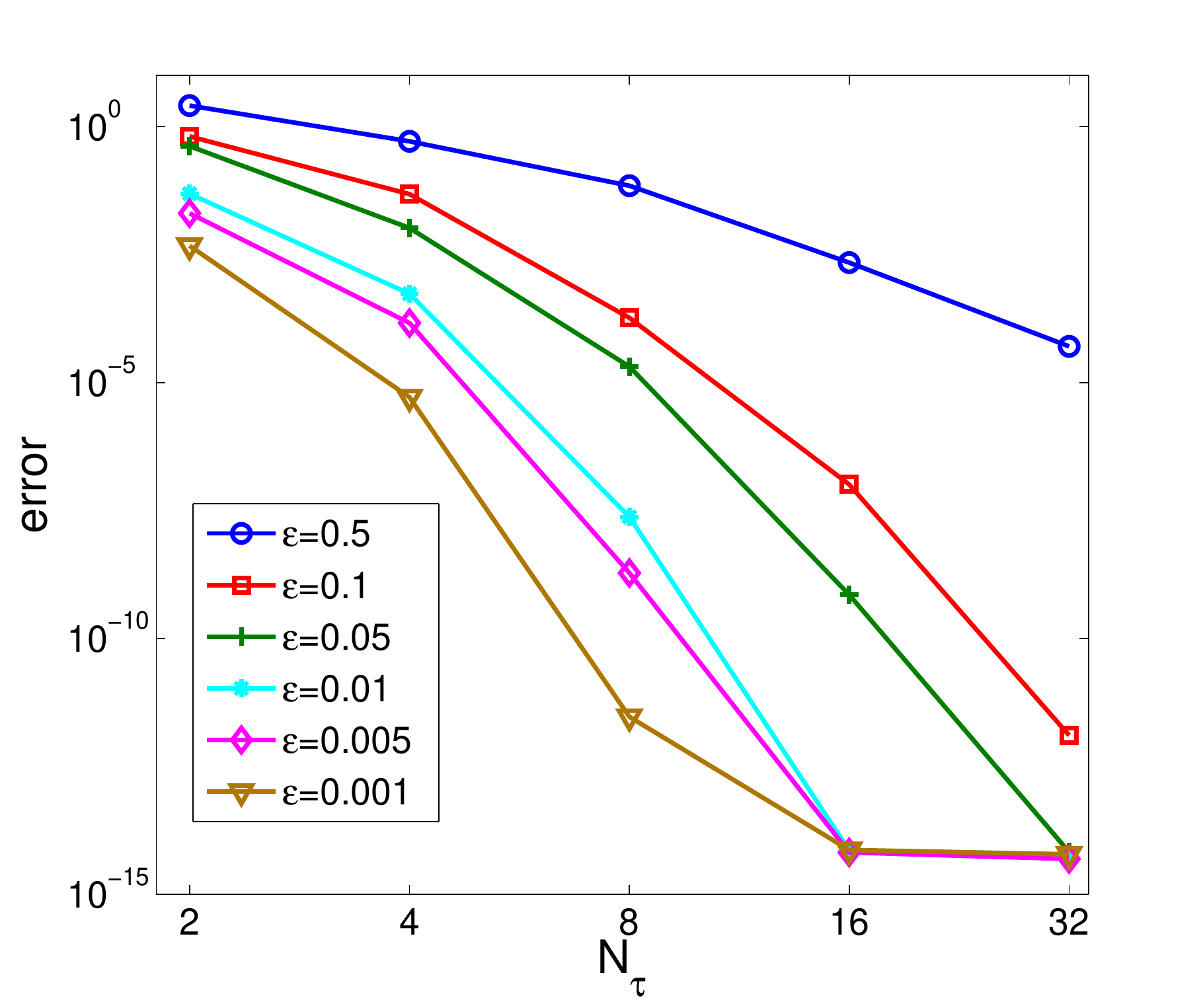,height=5cm,width=6.5cm}&\psfig{figure=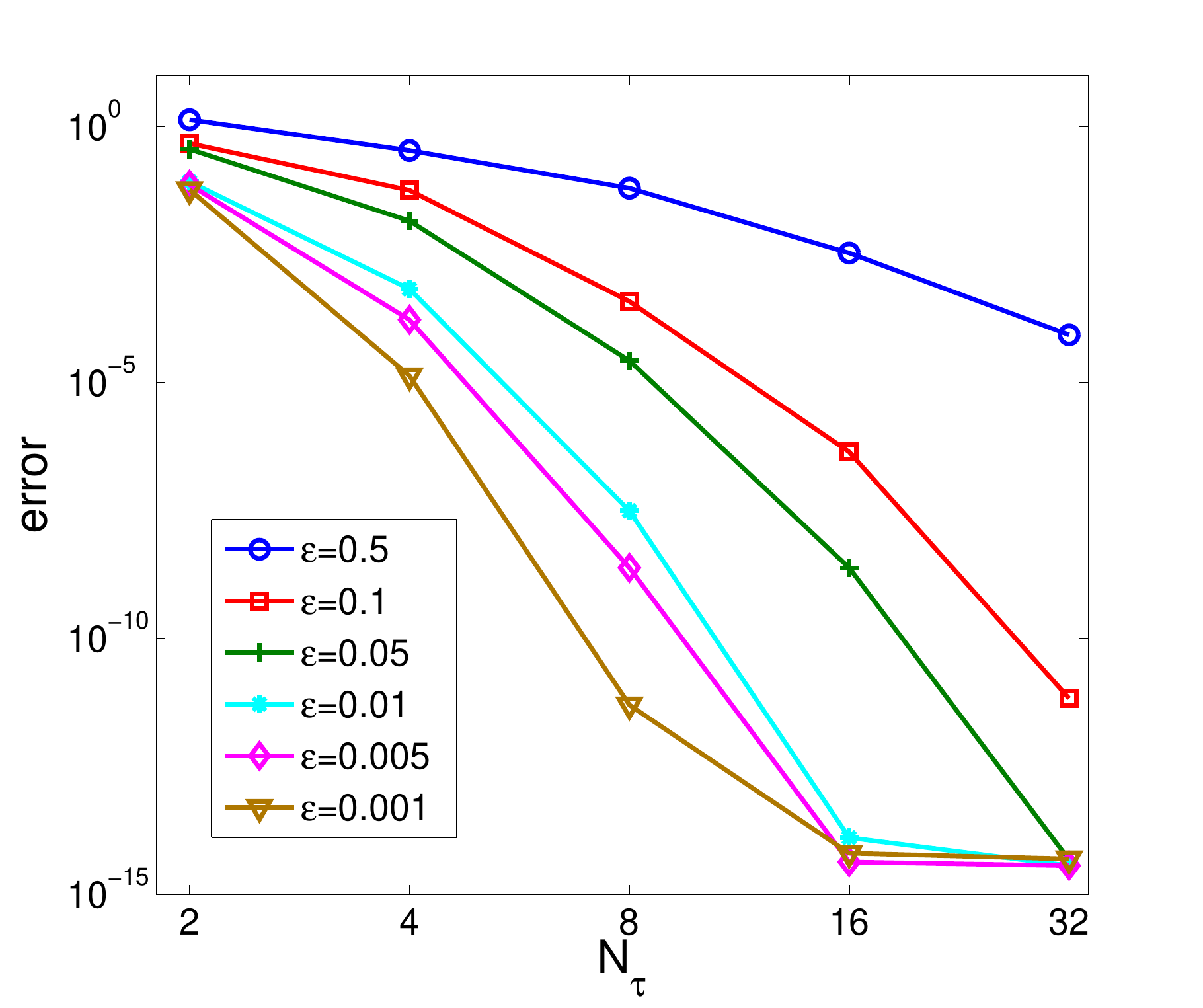,height=5cm,width=6.5cm}\end{array}$$
\caption{Error in $\tau$ of IMEX1 with  second order initial data $(X_k^{2nd},Y_k^{1st})$ for given $E$ case: relative maximum error in $\rho^\eps$ (left) and  in $\rho_\bv^\eps$ (right) with respect to the number of grids $N_\tau$.}\label{fig:VE1tau}
\end{figure}

\begin{figure}[t!]
$$\begin{array}{ccc}
\psfig{figure=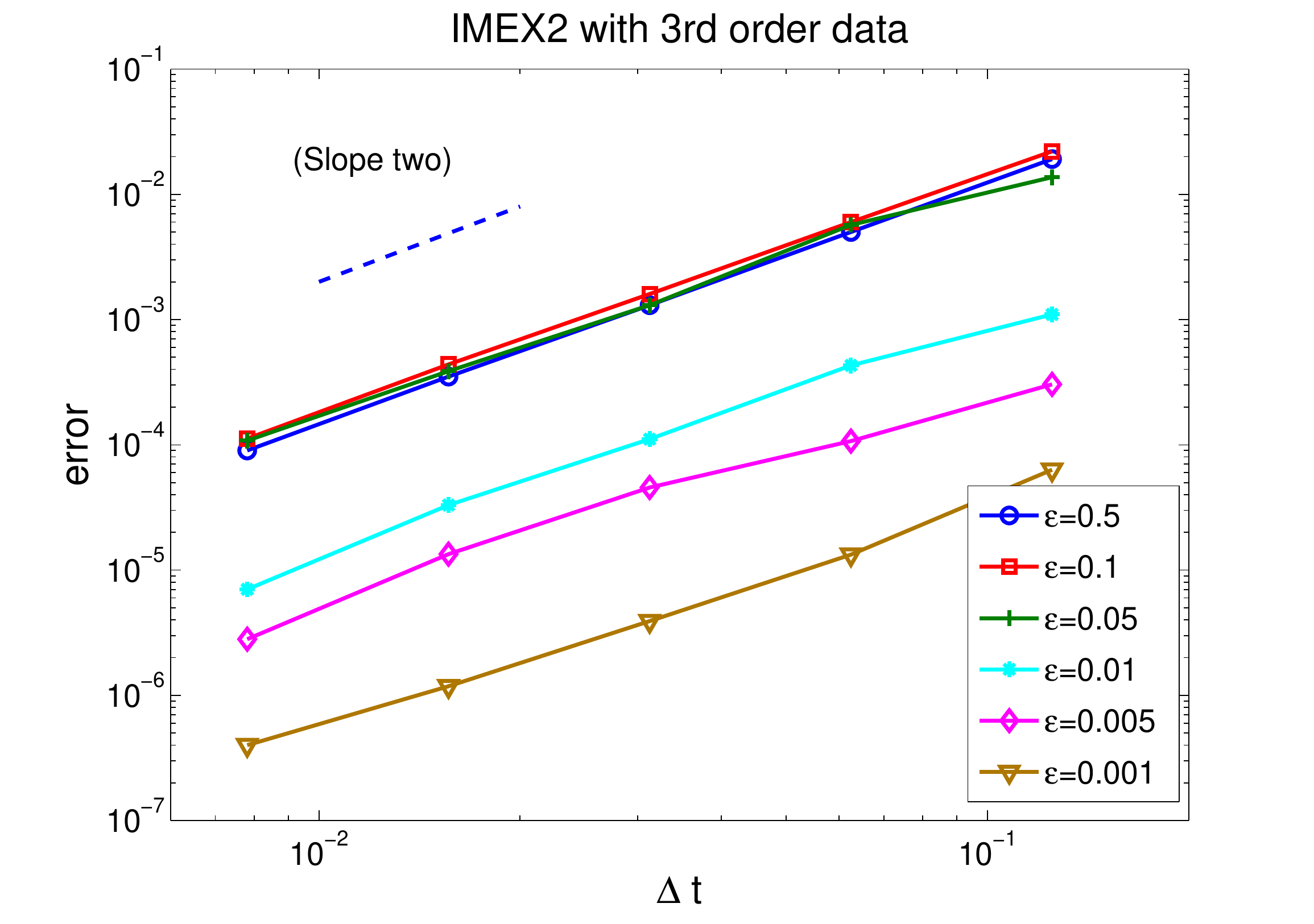,height=5.0cm,width=6.6cm}&\psfig{figure=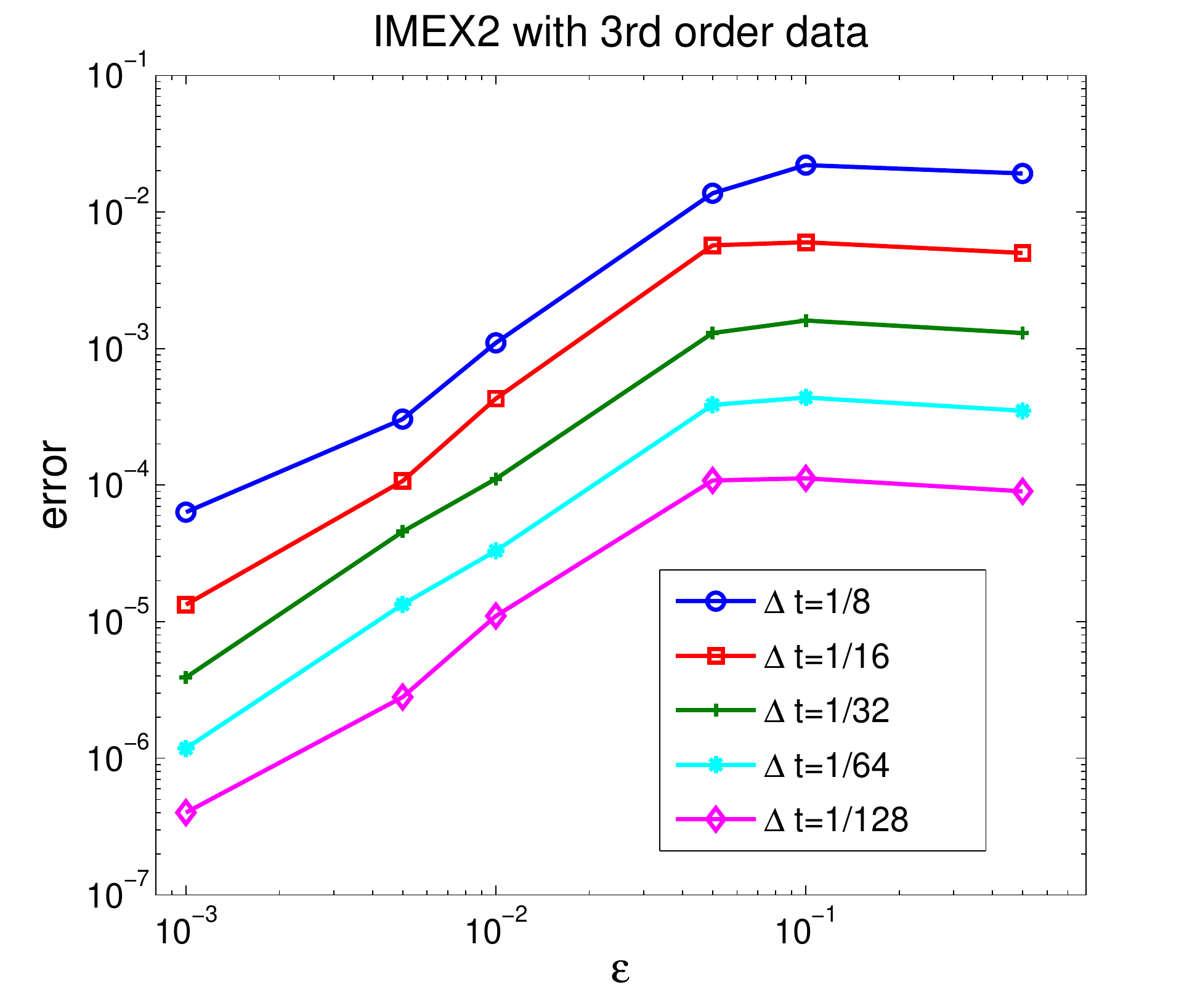,height=5.0cm,width=6.6cm}
\end{array}$$
\caption{Temporal error of IMEX2 with third order initial data $(X_k^{3rd},Y_k^{2nd})$ for given $E$ case: relative maximum error in $\rho^\eps$.}\label{fig:VE3}
\end{figure}

\begin{figure}[t!]
$$\begin{array}{cc}
\psfig{figure=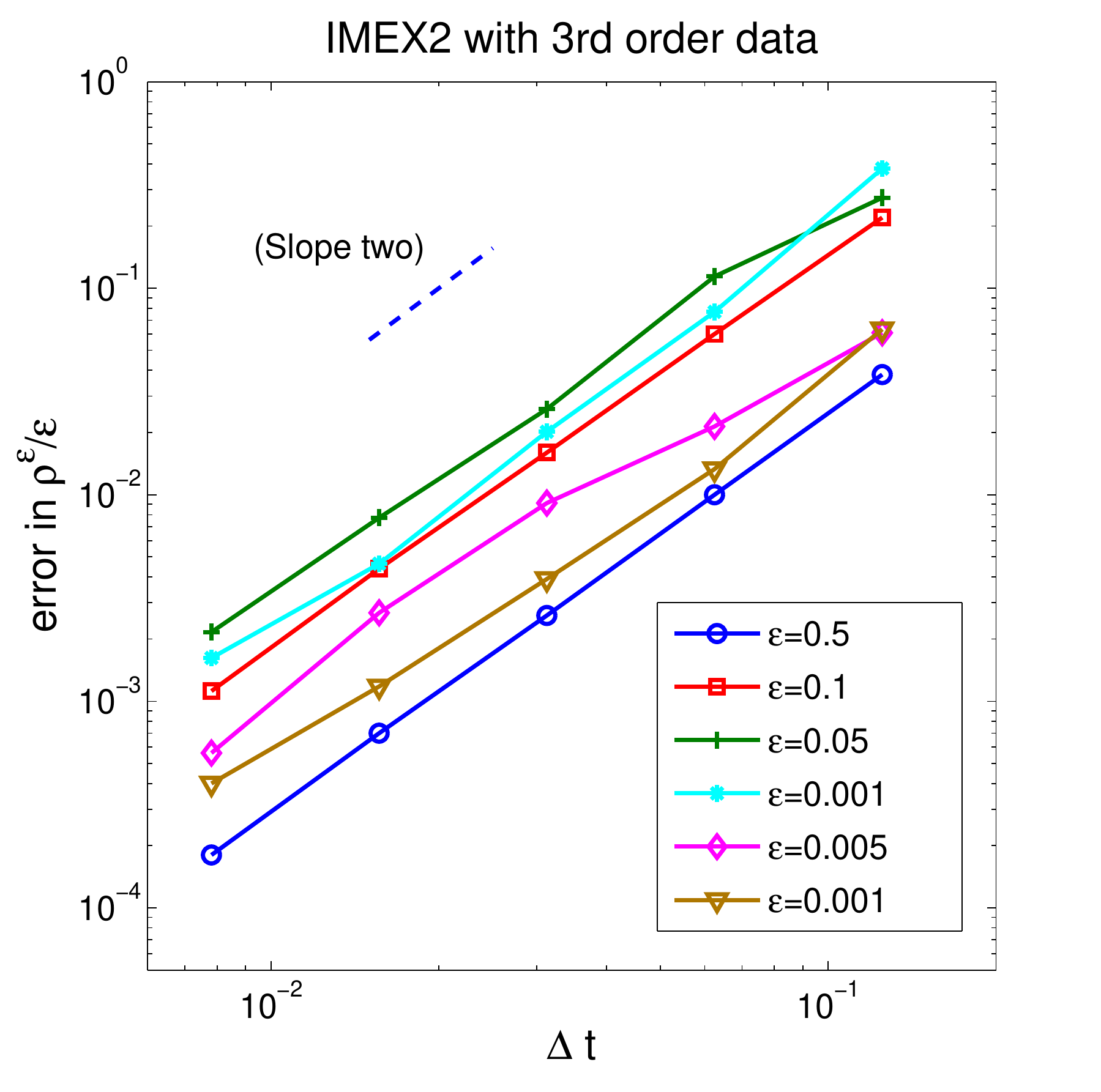,height=5cm,width=6.6cm}&\psfig{figure=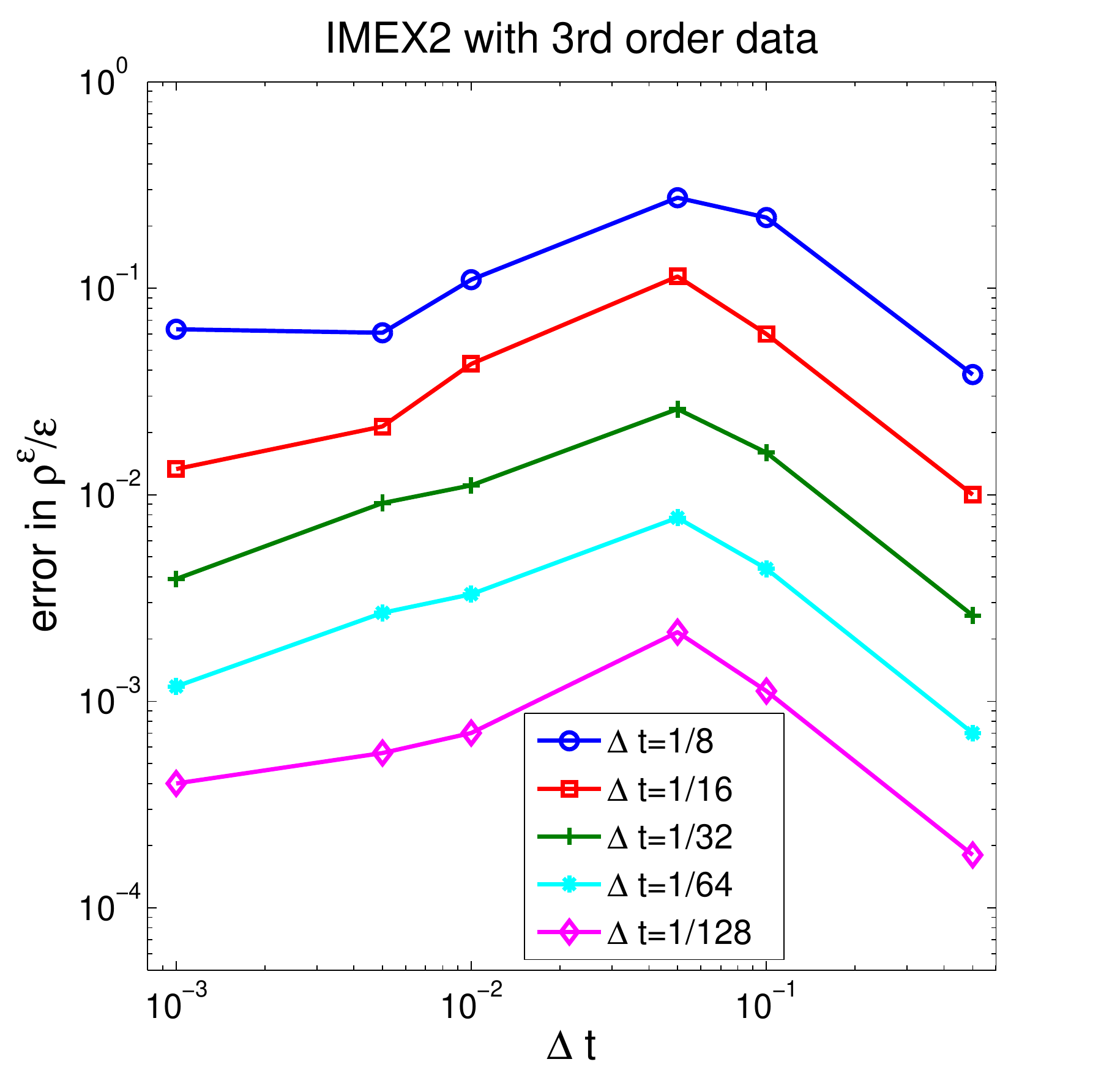,height=5cm,width=6.6cm}
\end{array}$$
\caption{Temporal error of IMEX2 with  third order initial data $(X_k^{3rd},Y_k^{2nd})$ for given $E$ case: relative maximum error in $\rho^\eps/\eps$.}\label{fig:super}
\end{figure}

\begin{figure}[t!]
$$\begin{array}{cc}
\psfig{figure=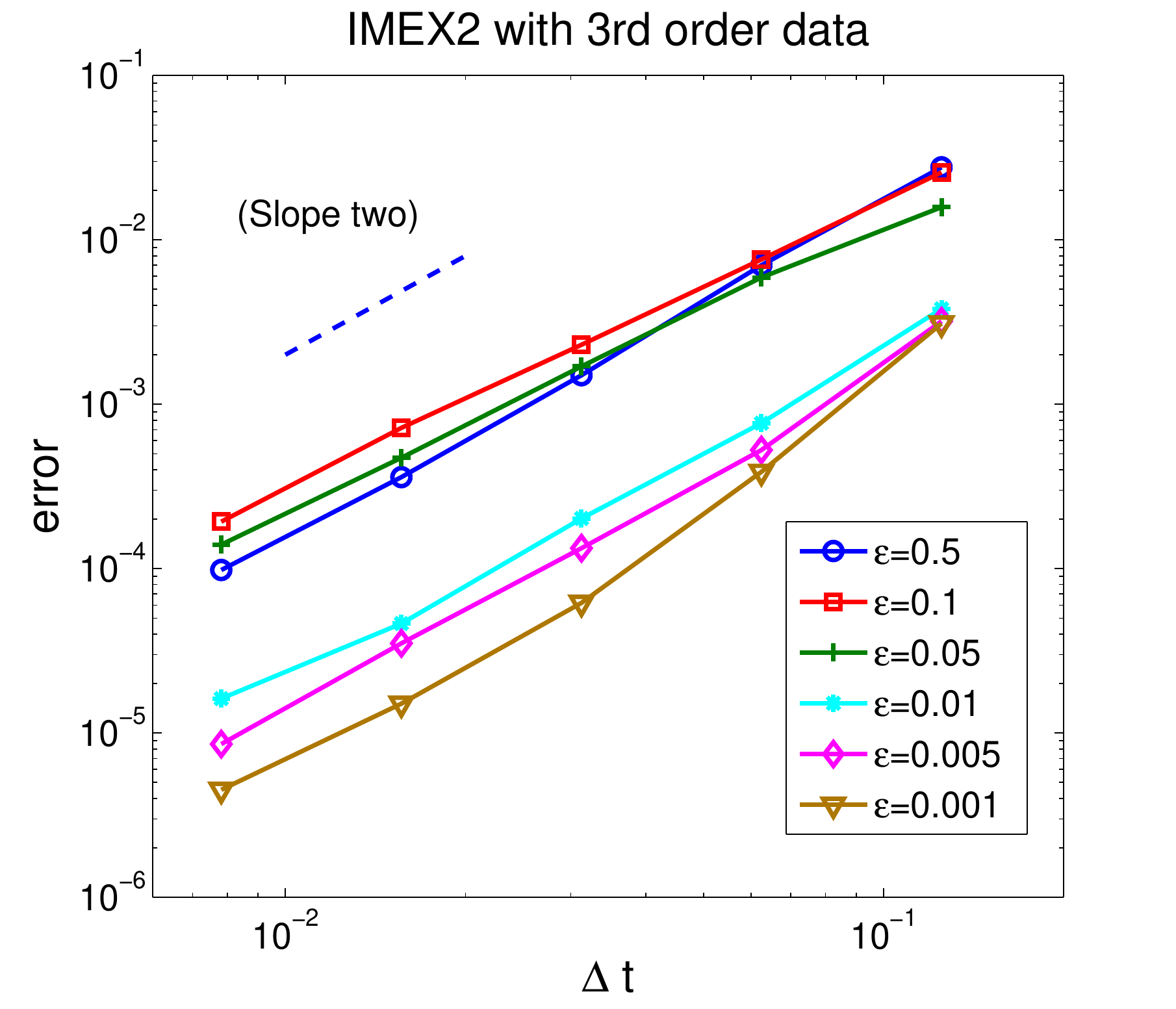,height=5cm,width=6.6cm}&\psfig{figure=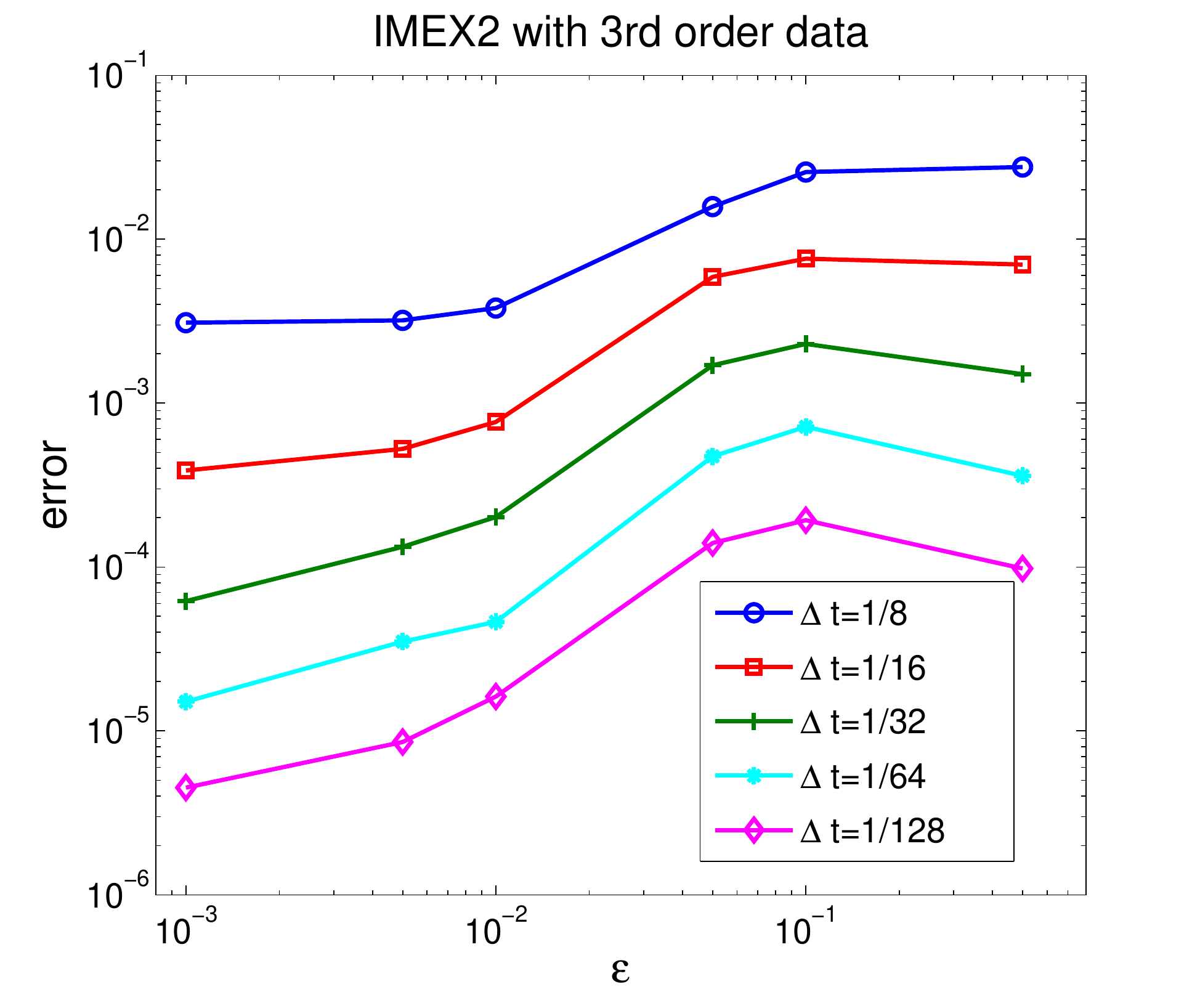,height=5cm,width=6.6cm}
\end{array}$$
\caption{Temporal error of IMEX2 with  third order initial data $(X_k^{3rd},Y_k^{2nd})$ for given $E$ case: relative maximum error in $\rho^\eps_\bv$.}\label{fig:VE2}
\end{figure}

\begin{figure}[h!]
$$\begin{array}{cc}
\psfig{figure=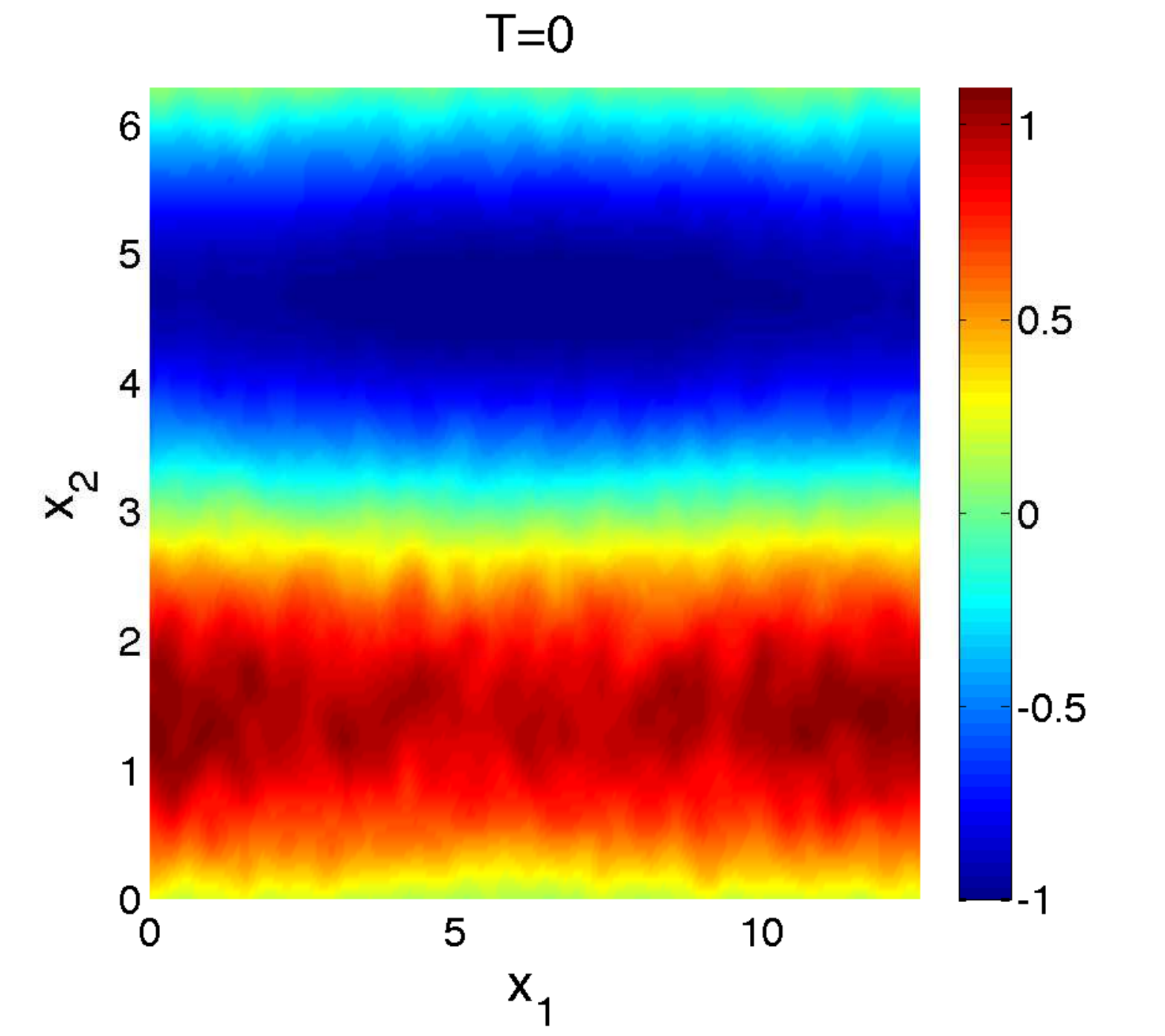,height=5cm,width=6.5cm}&\psfig{figure=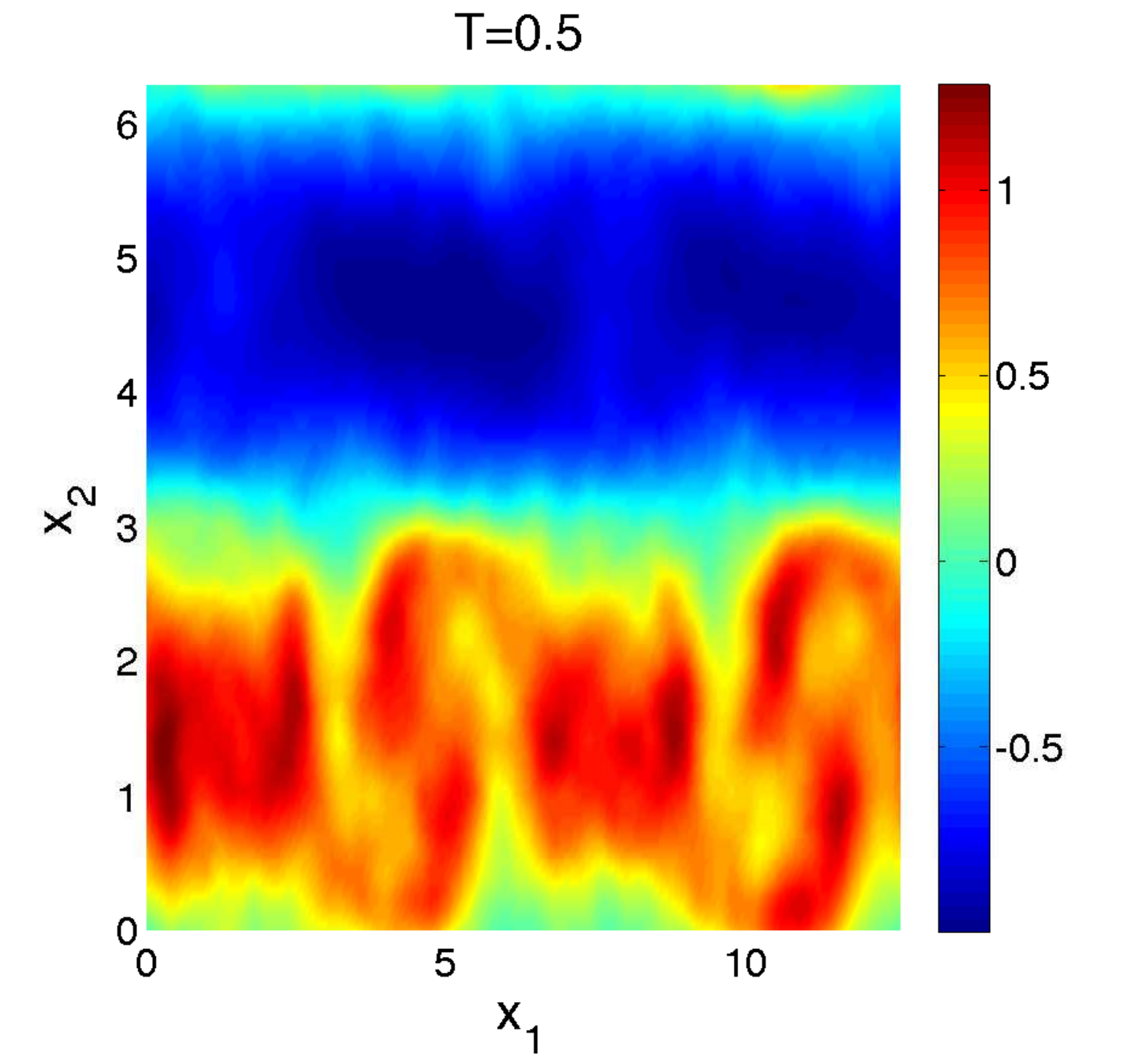,height=5cm,width=6.5cm}\\
\psfig{figure=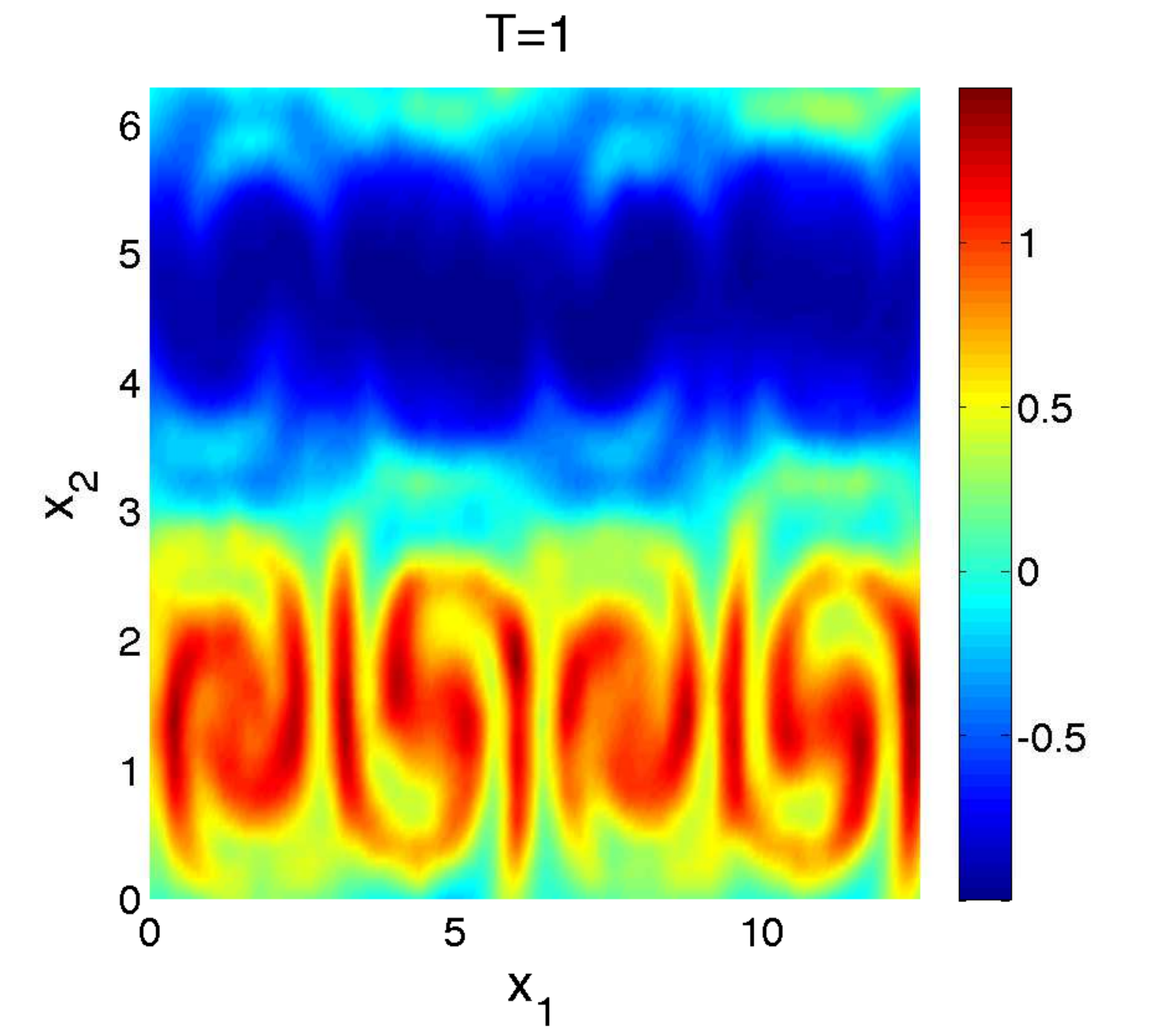,height=5cm,width=6.5cm}&\psfig{figure=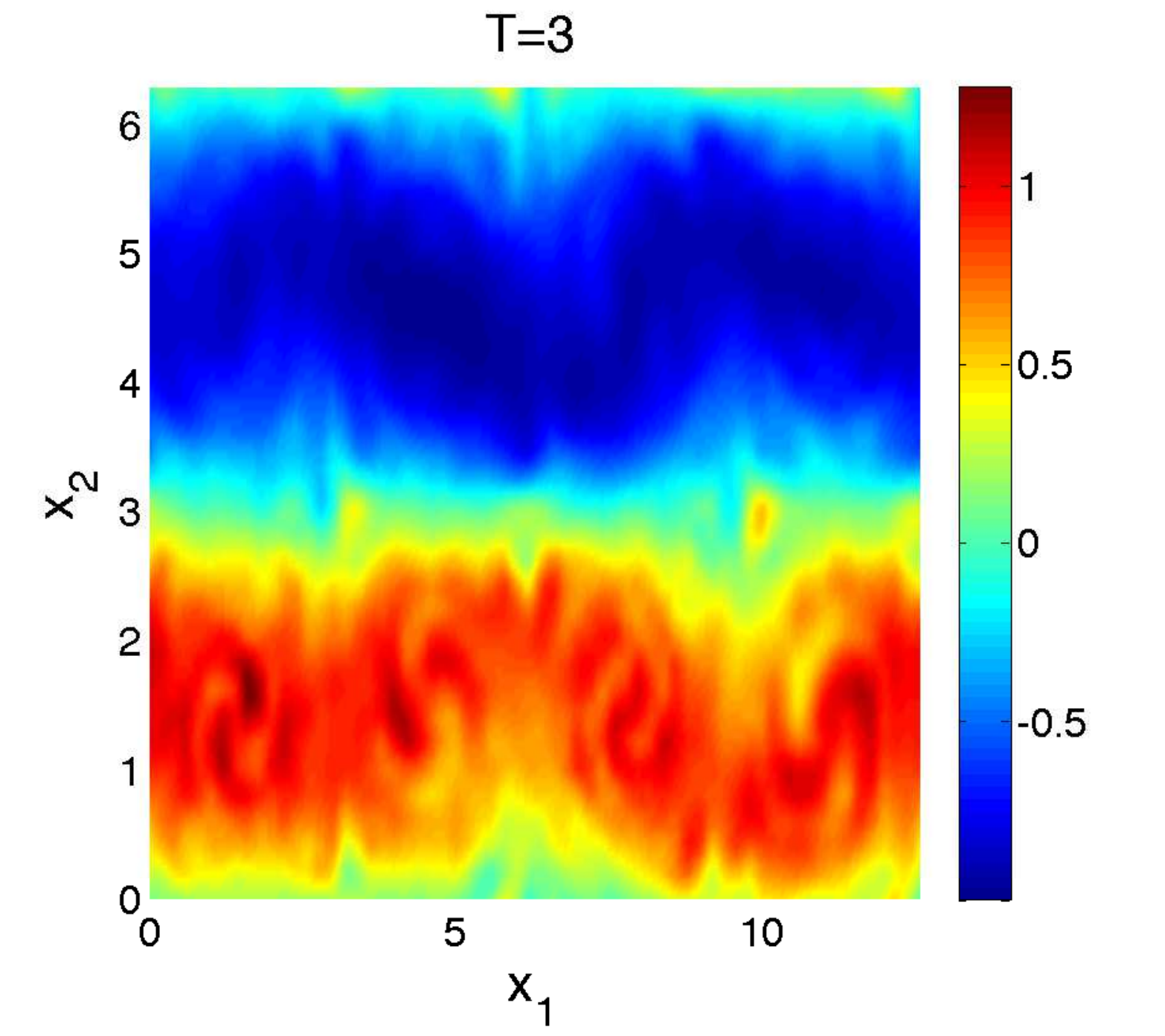,height=5cm,width=6.5cm}
\end{array}$$
\caption{2D plot of $\rho^\eps(t,\bx)$ at different $t$ for given $E$ case with $\eps=0.1$.}\label{fig:2d1}
\end{figure}

\begin{figure}[t!]
$$\begin{array}{cc}
\psfig{figure=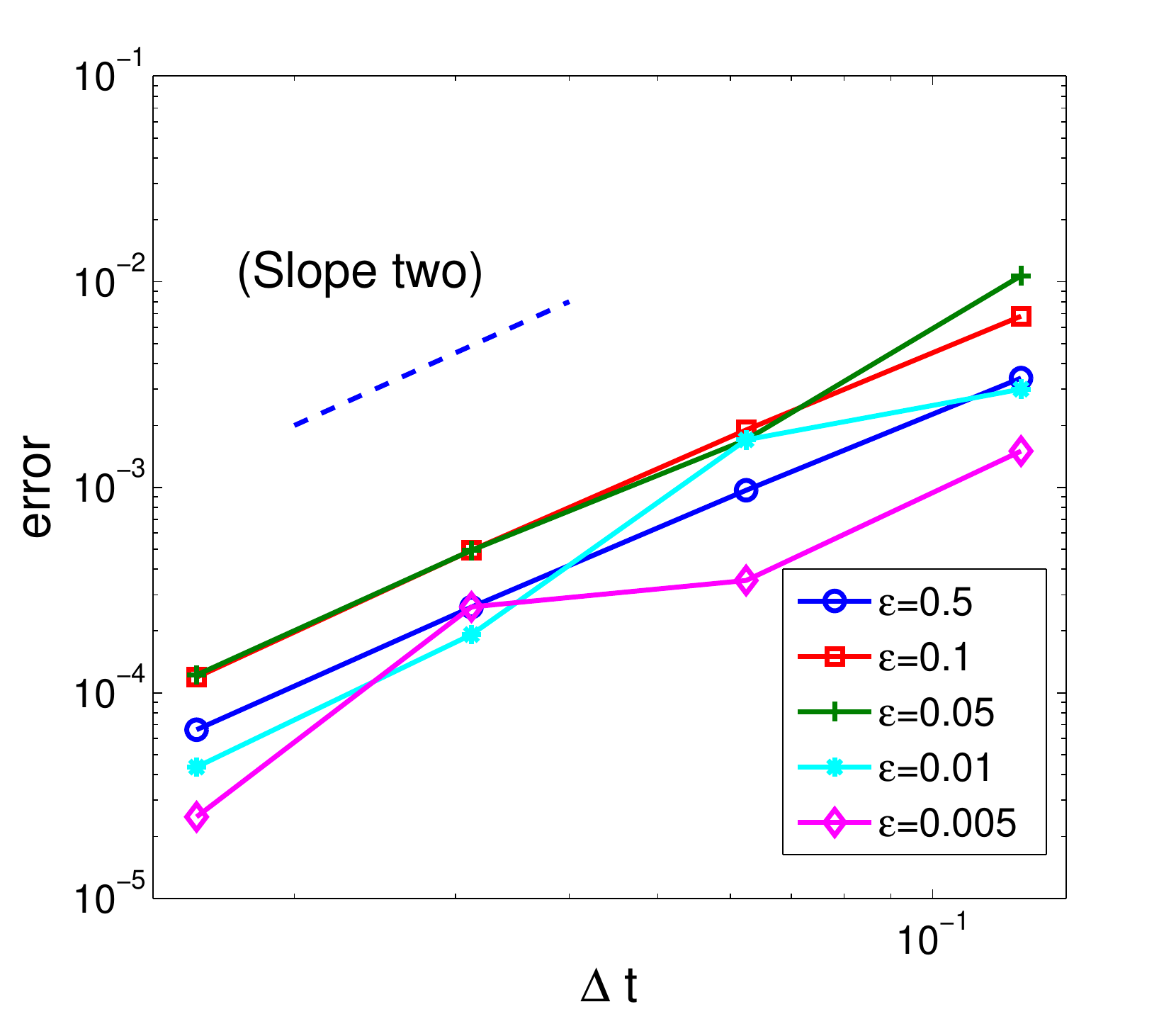,height=5cm,width=6.5cm}&\psfig{figure=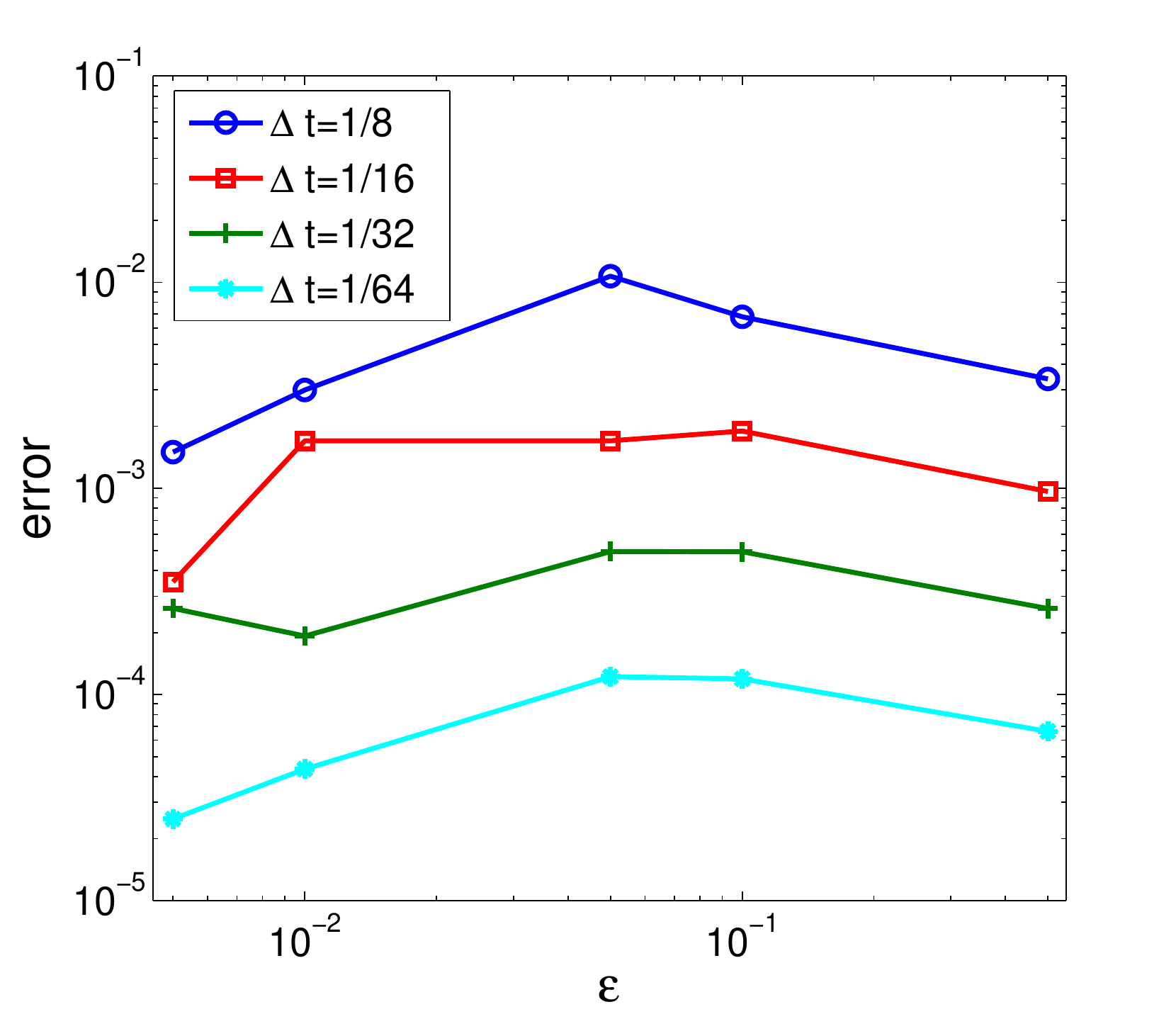,height=5cm,width=6.5cm}\\
\psfig{figure=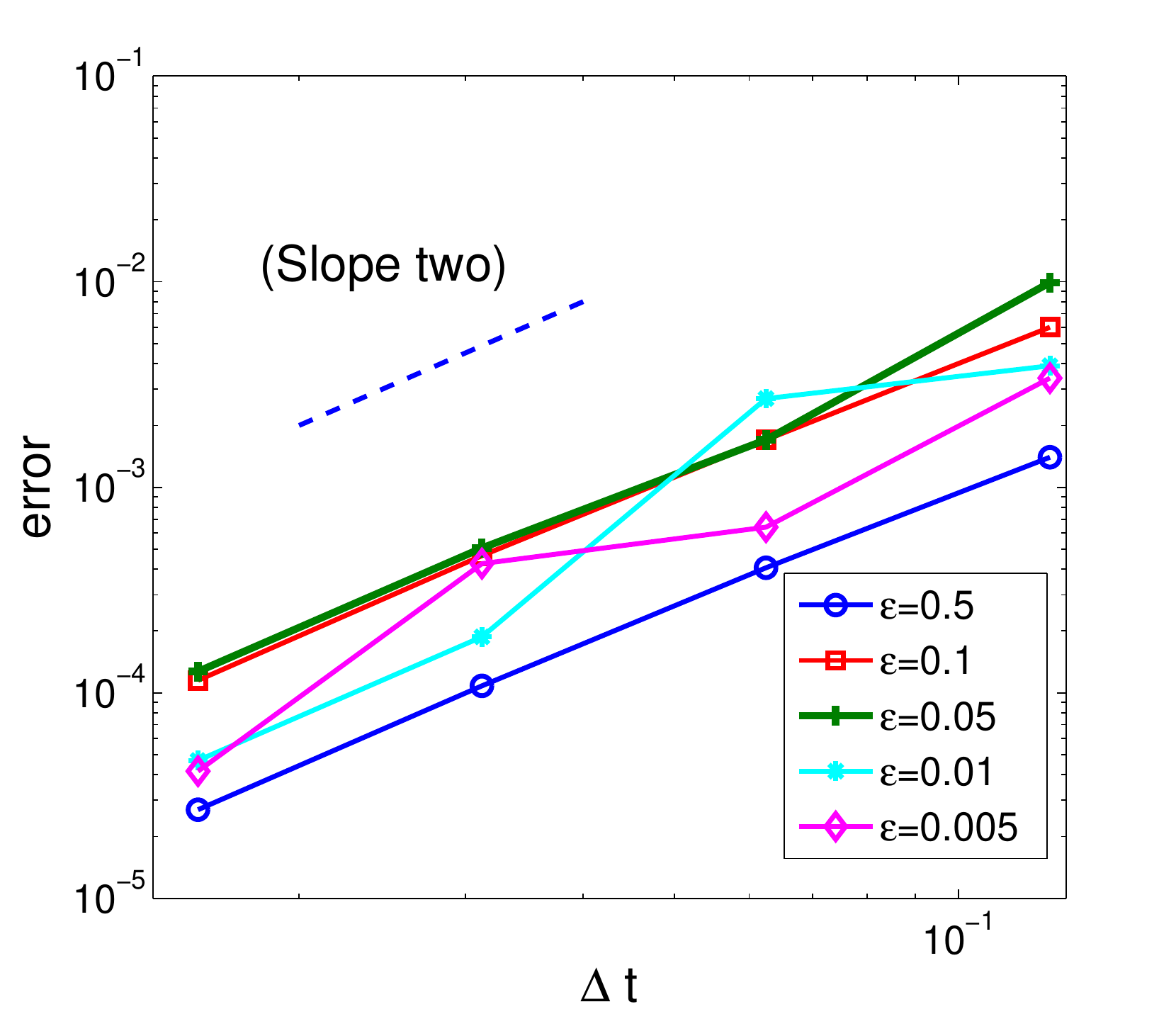,height=5cm,width=6.5cm}&\psfig{figure=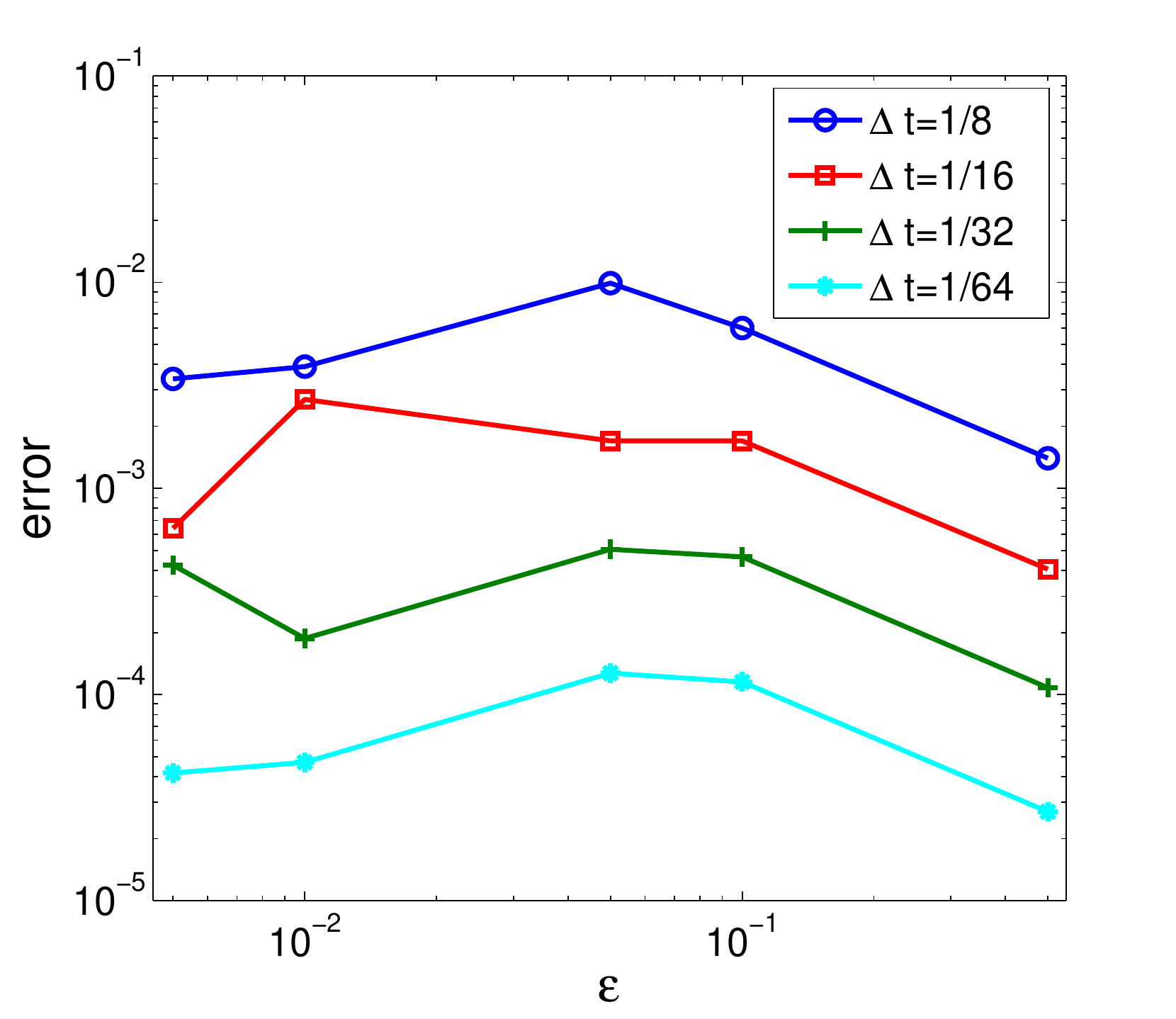,height=5cm,width=6.5cm}
\end{array}$$
\caption{Temporal error of EI2 with $(X_k^{1st},Y_k^{1st})$ for Vlasov-Poisson case:
relative maximum error in $\rho^\eps$ (first row) and $\rho_\bv^\eps$ (second row). }\label{fig:VP2update}
\end{figure}

\begin{figure}[t!]
$$\begin{array}{ccc}
\psfig{figure=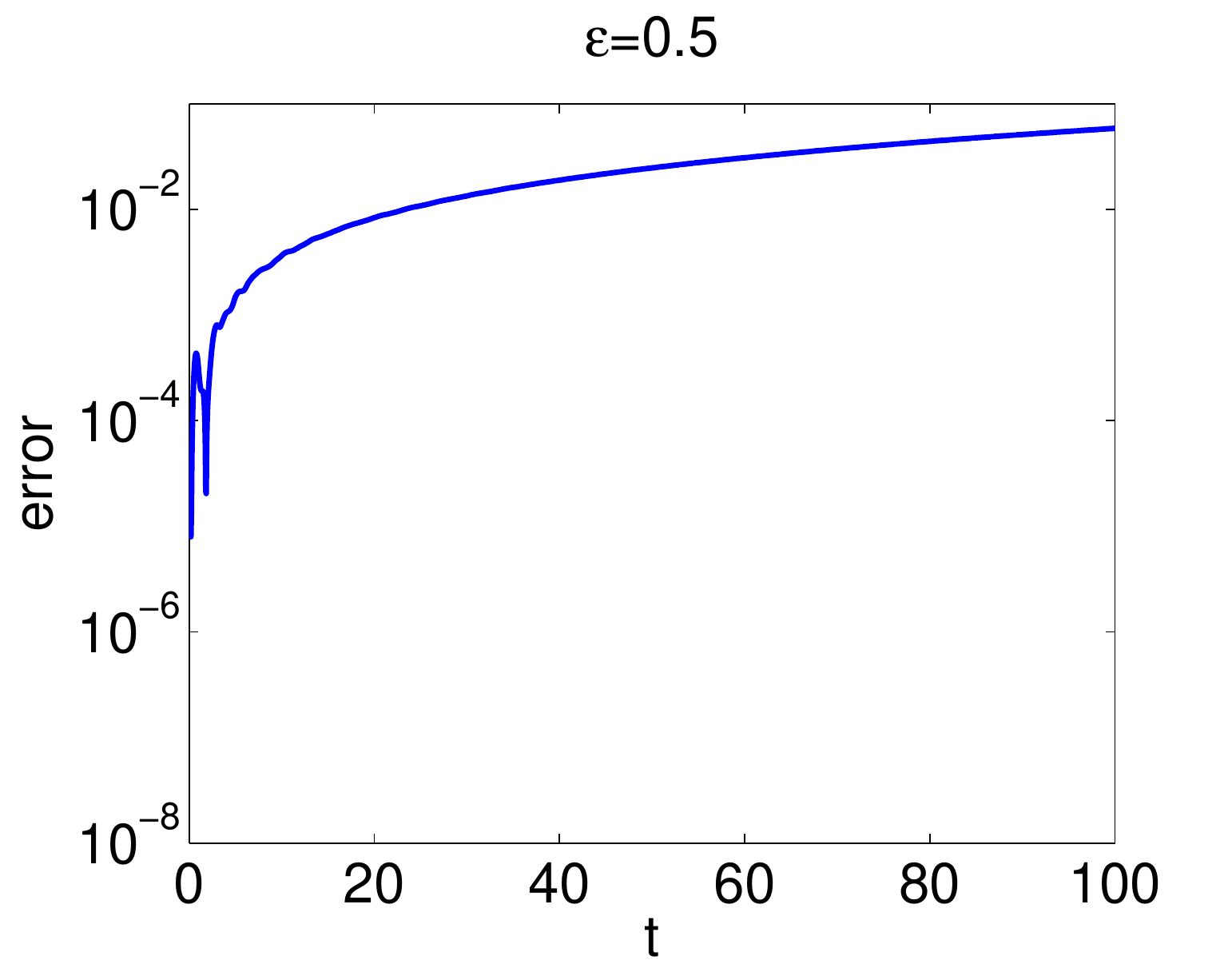,height=5cm,width=4cm}&\psfig{figure=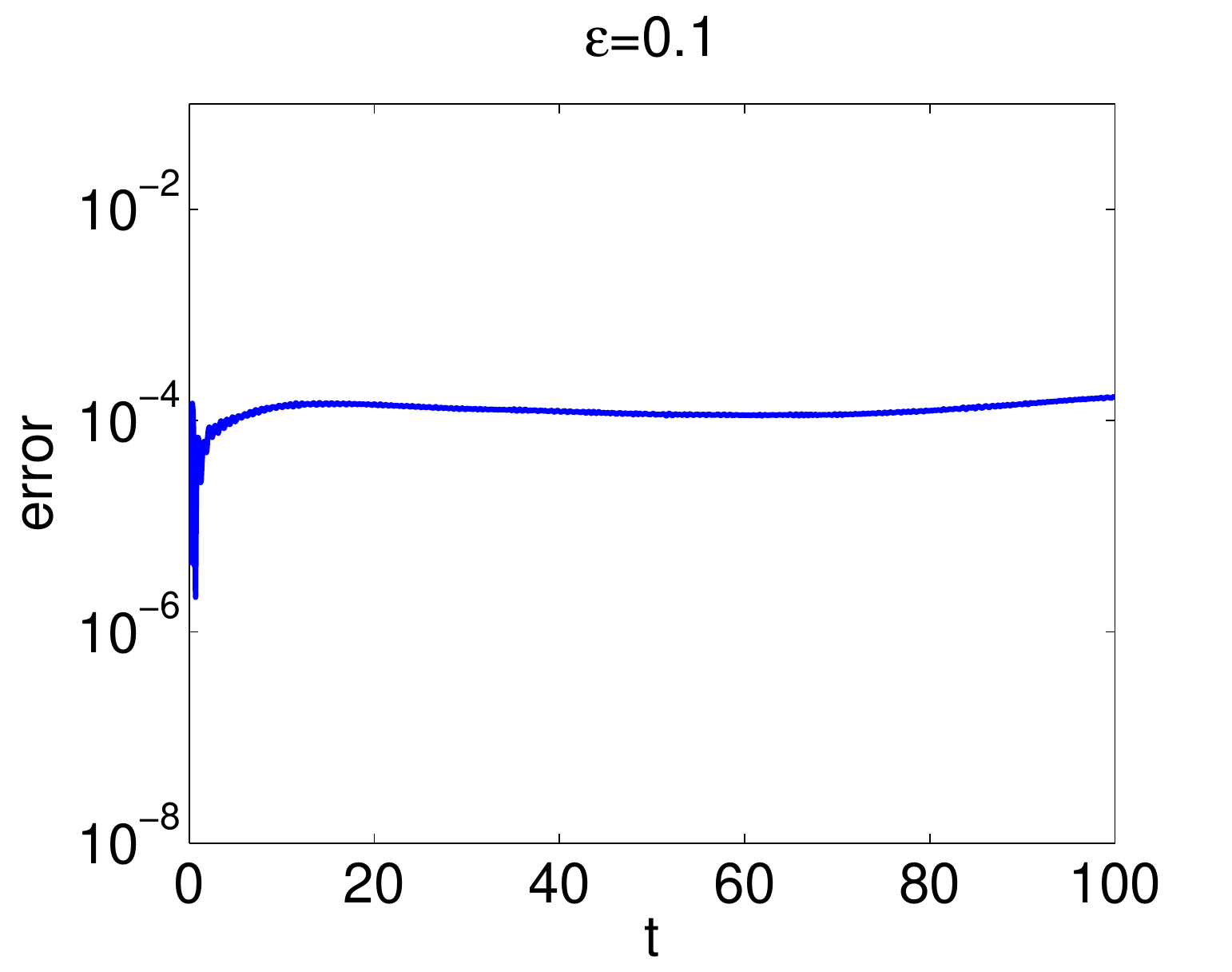,height=5cm,width=4cm}
&\psfig{figure=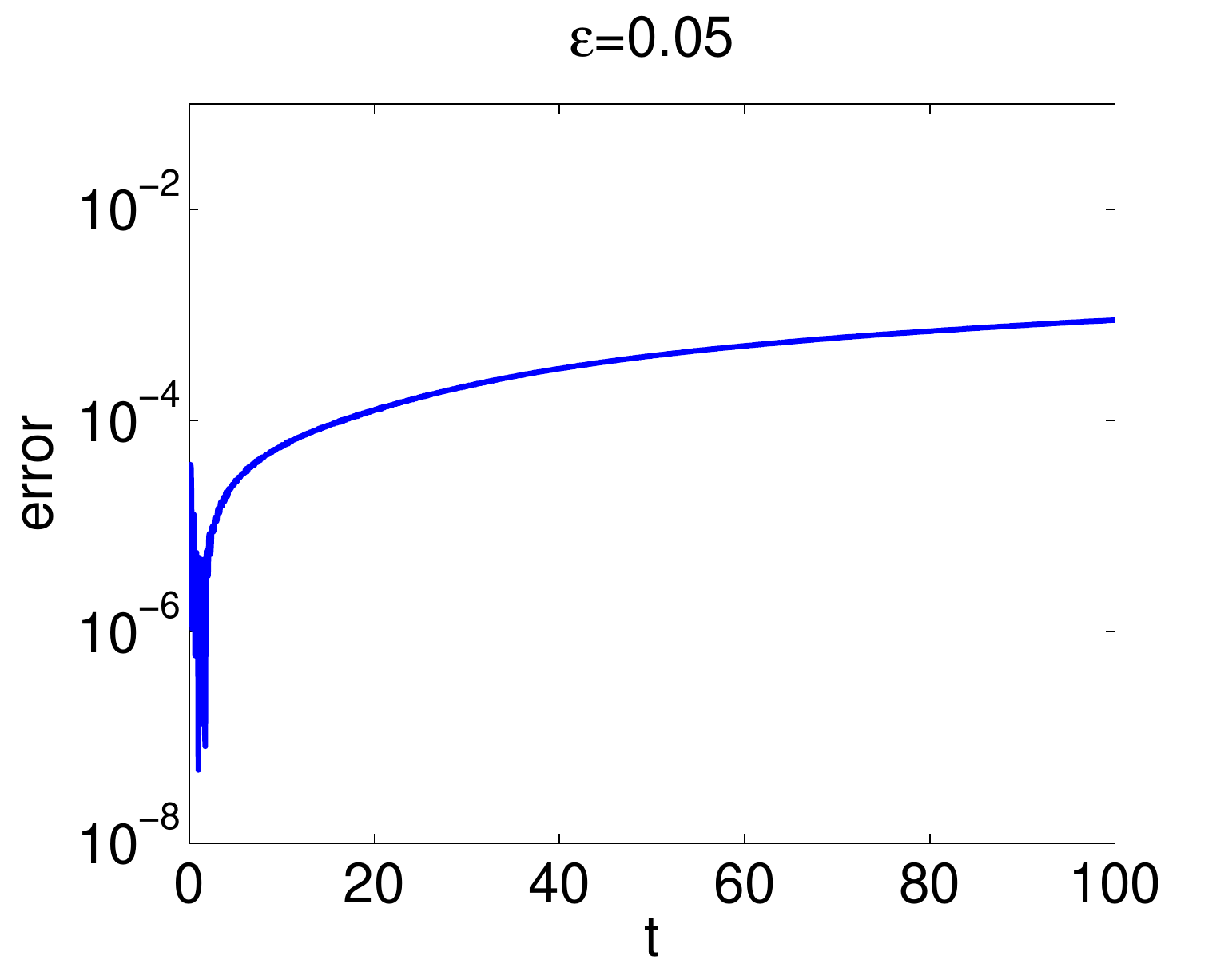,height=5cm,width=4cm}\\
\psfig{figure=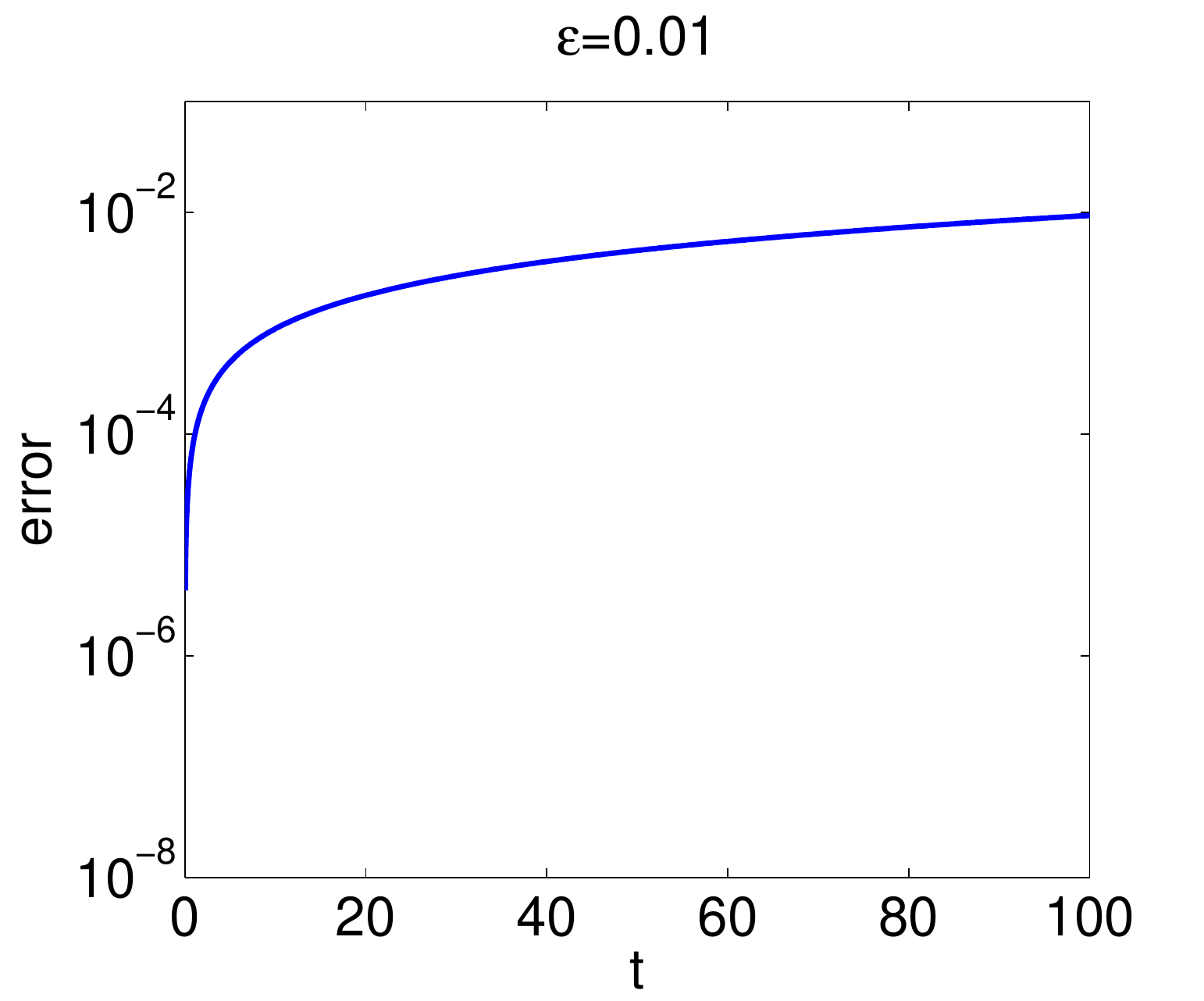,height=5cm,width=4cm}&\psfig{figure=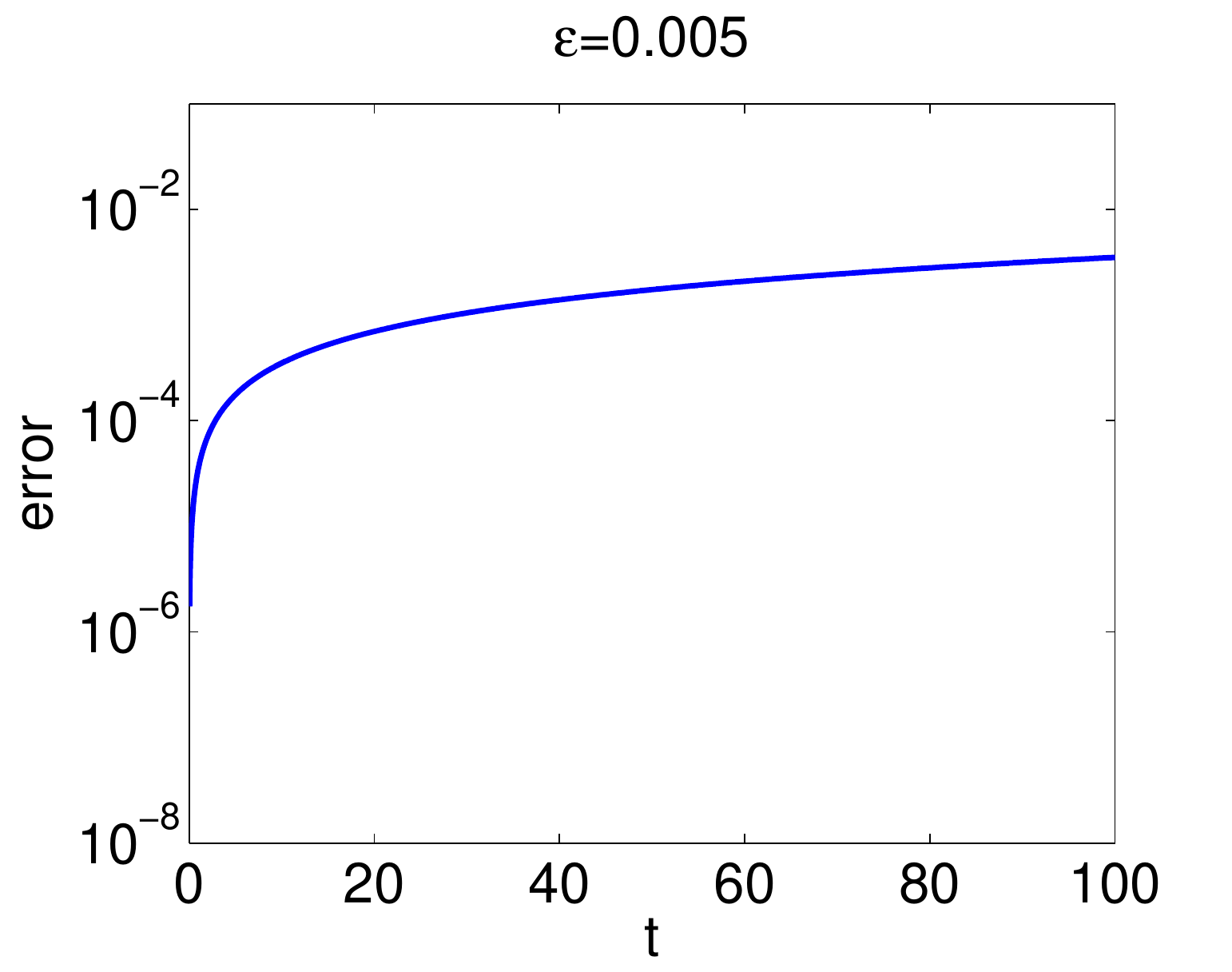,height=5cm,width=4cm}
&\psfig{figure=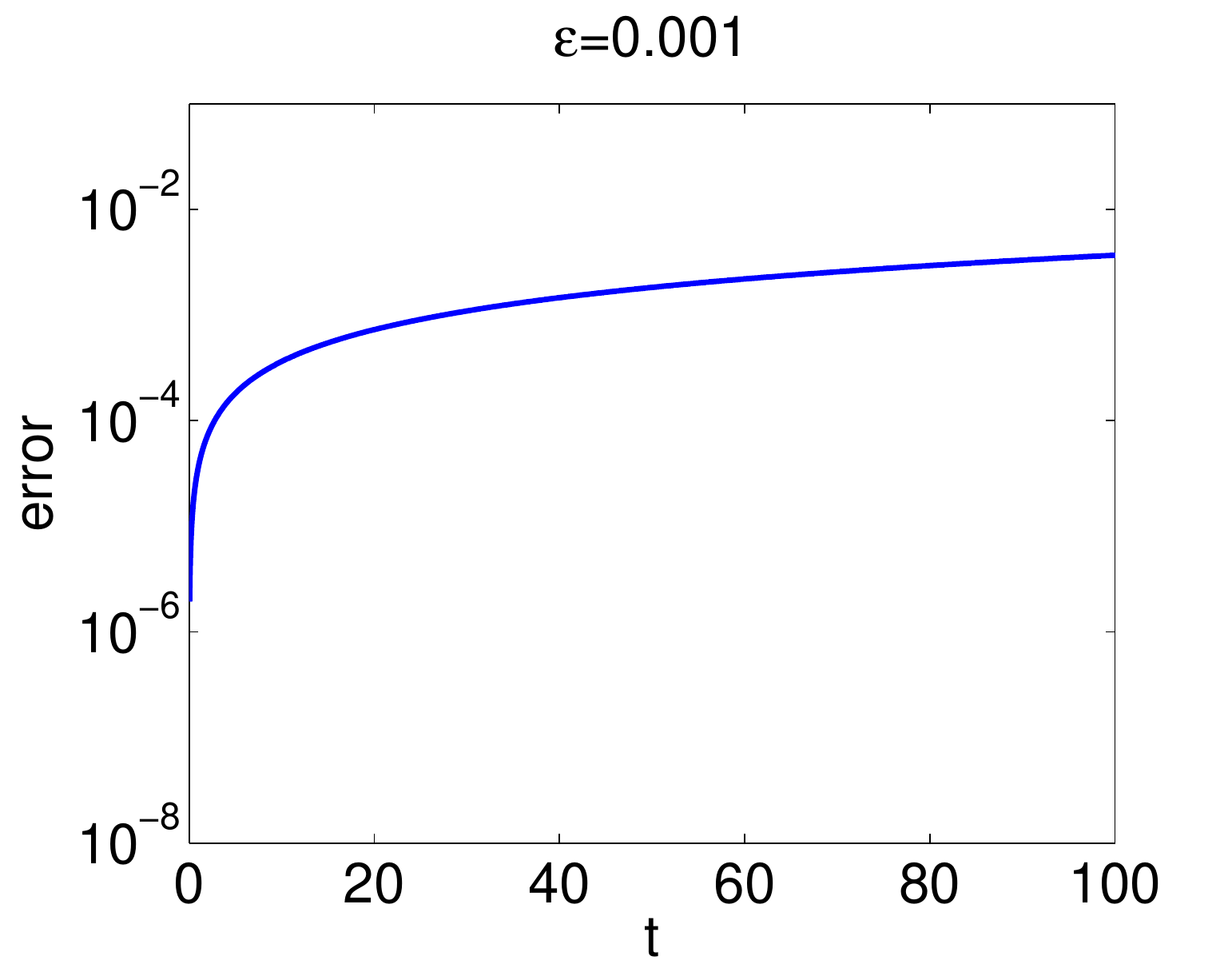,height=5cm,width=4cm}
\end{array}$$
\caption{Relative energy error $|H(t)-H(0)|/H(0)$ of the EI2 scheme with $\Delta t=0.05$ for the Vlasov-Poisson case under serval $\eps$.}\label{fig:energy}
\end{figure}

\begin{figure}[t!]
$$\begin{array}{cc}
\psfig{figure=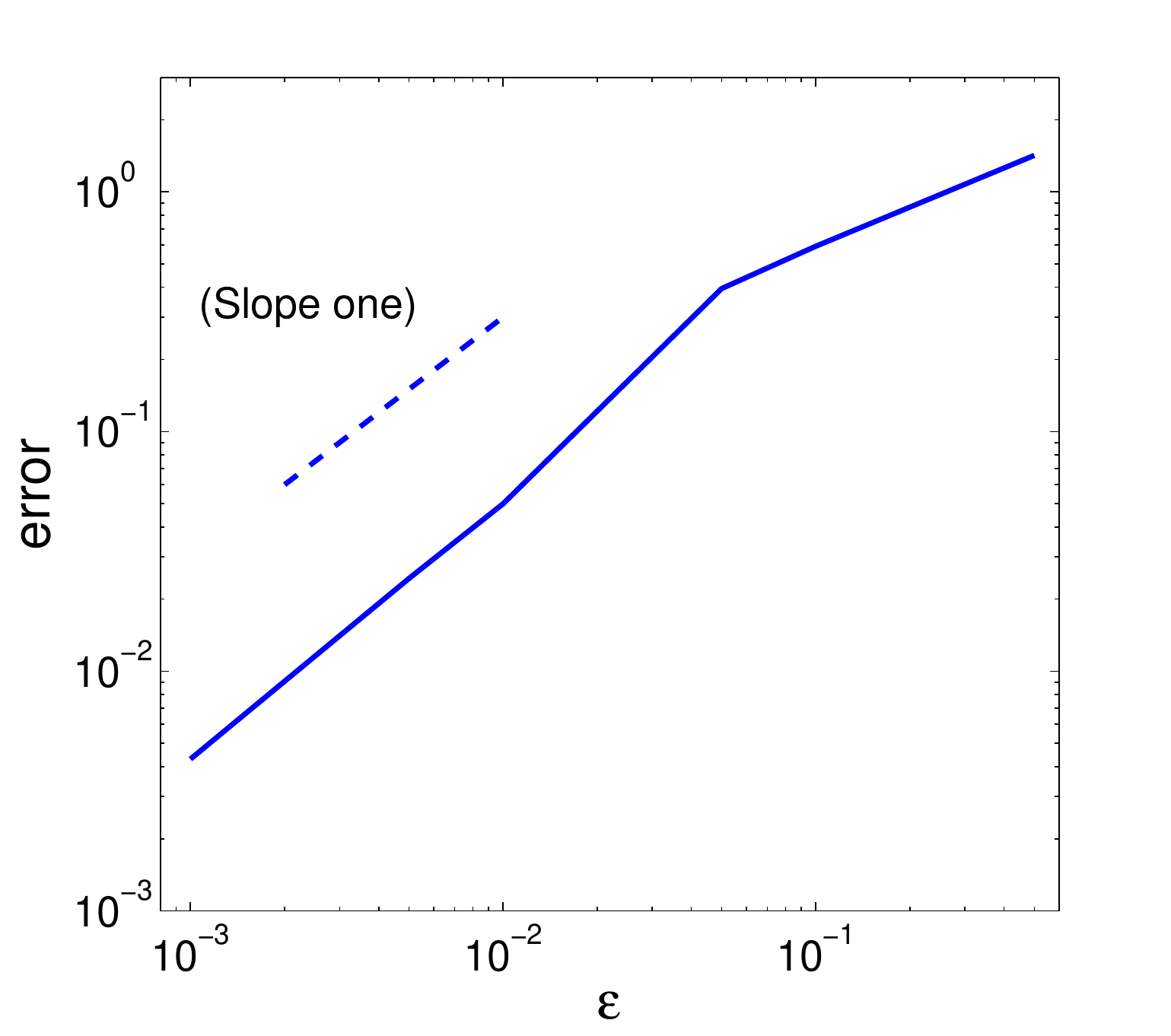,height=5cm,width=6.5cm}&\psfig{figure=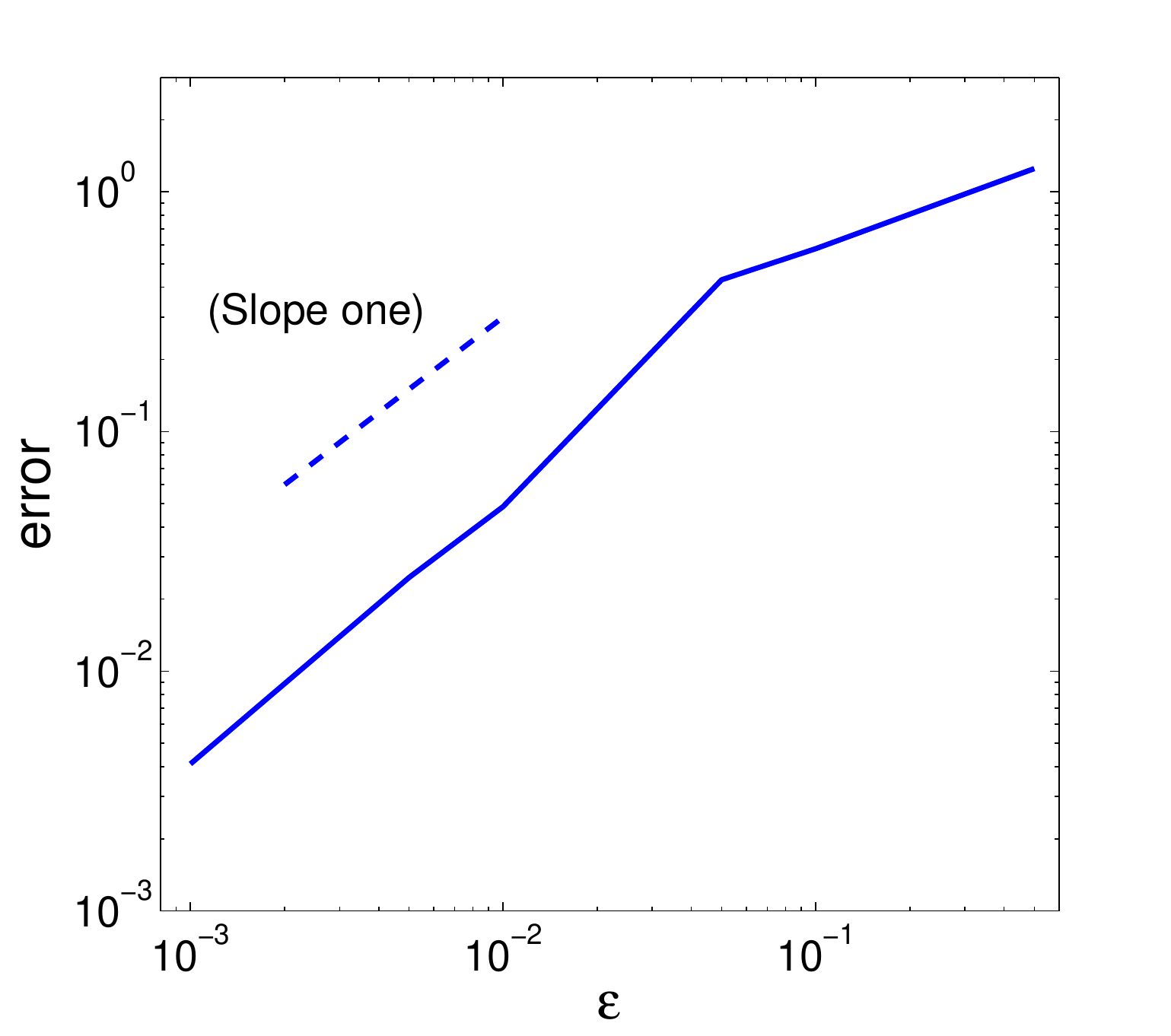,height=5cm,width=6.5cm}\\
\psfig{figure=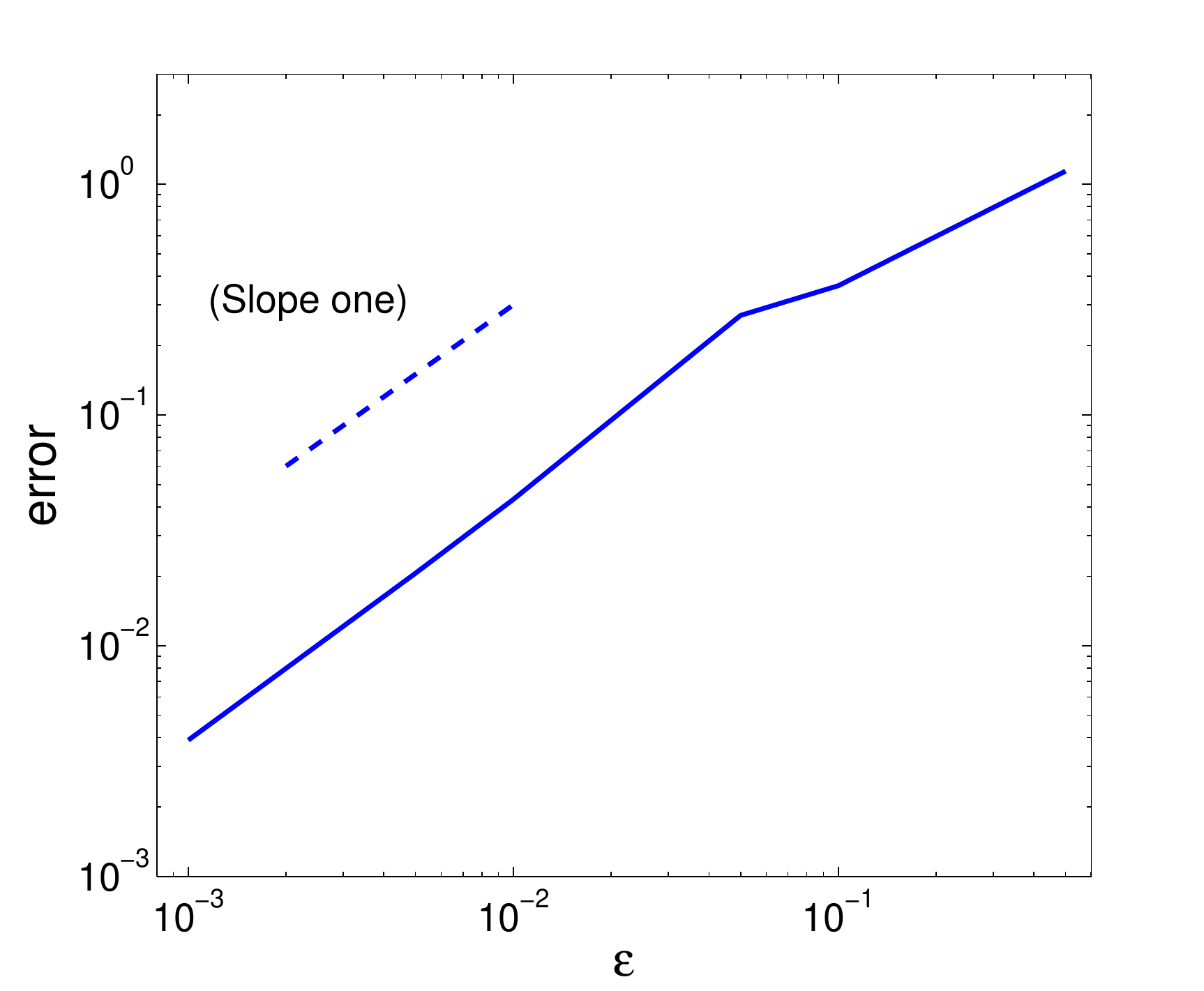,height=5cm,width=6.5cm}&\psfig{figure=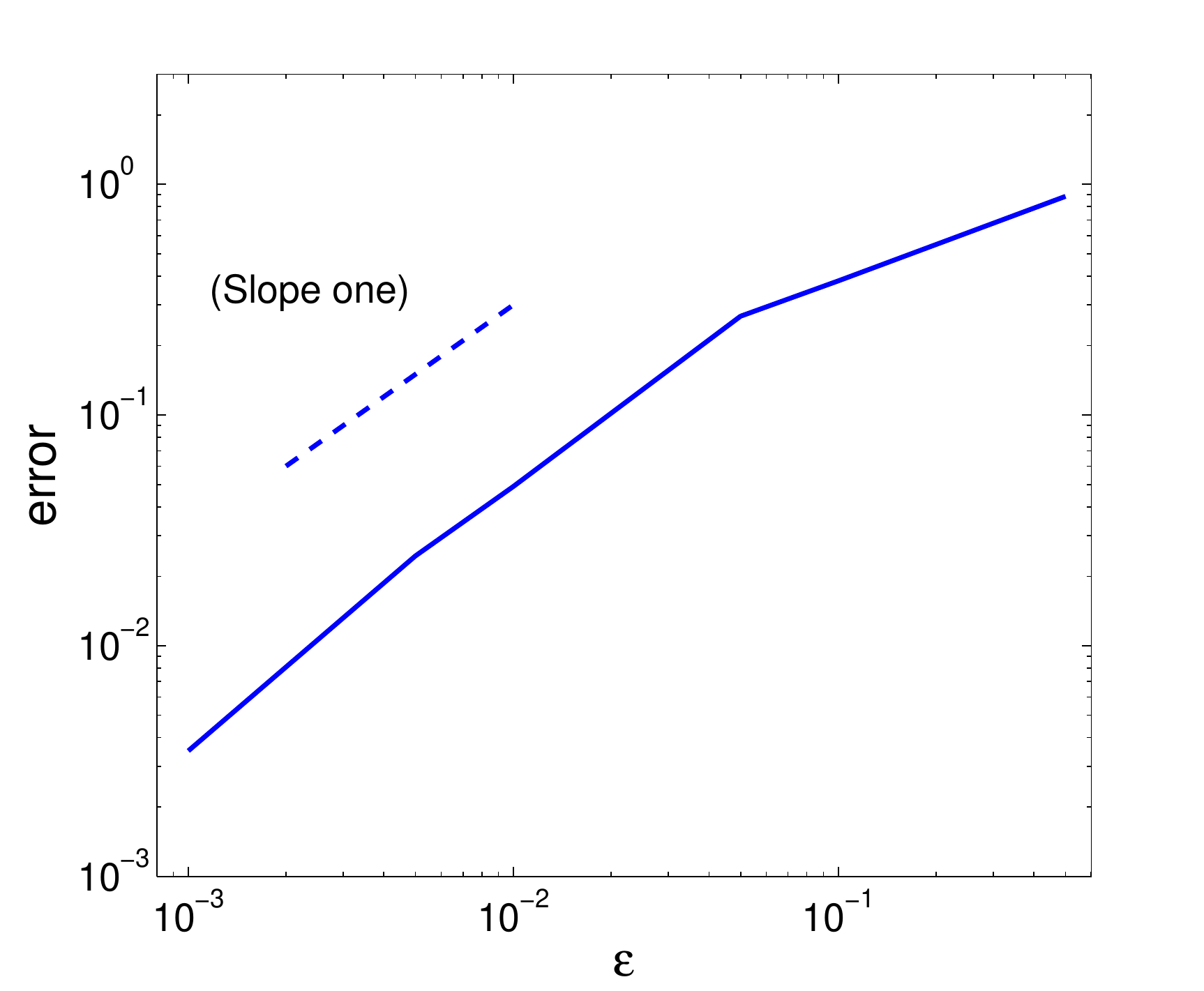,height=5cm,width=6.5cm}\end{array}$$
\caption{Convergence the Vlasov equation (\ref{eq:1}) in the given $E$ case (left) and in the Vlasov-Poisson case (right) to in the limit regime: relative maximum error in $\rho^\eps$ (first row) and $\rho^\eps_\bv$ (second row) at $t_f=1$.}\label{fig:limit}
\end{figure}

\section{Conclusion}
We proposed a multi-scale numerical scheme for the Vlasov equation with a strong non-homogeneous magnetic field by using a Particle-in-Cell strategy. The solution of the problem is highly oscillatory in time, space and velocity with non-periodic oscillation. Making use of the fact that the positions of the particles in this regime are confined around the initial position, we transformed the characteristics into a suitable form which enabled us to perform the separation of scales techniques.
A uniformly accurate first order scheme was then proposed and rigorously analyzed for the Vlasov equation
with external electric field; for this scheme, we also proved that it enjoys the confinement property at the discrete level.
Practical extensions are performed to achieve the second order accuracy, and also to deal with the case
of the Vlasov-Poisson equation. In the later case, it turns out that the characteristics equations
are a huge highly oscillatory system with multiple frequencies. Numerical results are then presented
to confirm the theoretical results and illustrate the efficiency of the proposed schemes.

\appendix
\section{Third order preparation} %{Appendix}
\label{appendix}
In this appendix, we derive the third order initial data, which will ensure (following the same strategy
used in Proposition \ref{prop:1st} that the quantities $\frac{1}{\eps} \partial_s^3 X_k$ and $\partial_s^3 Y_k$ are uniformly bounded.

We first derive (\ref{dthr def a}) with respect to $s$ and
assuming $\partial_s^3\bh_k=O(\eps)$, we find %,\,, we find
\begin{align*}
&\partial_s^2\bh_k(0,\tau)=\tilde{\tilde{\bh}}_k^{1st}(\tau)+O(\eps^2) \;\; \mbox{ with } \tilde{\tilde{\bh}}_k^{1st}(\tau) :=  \frac{\eps}{b_k} A \fe^{\tau J}\tilde{\tilde{{\underline{Y}}}}_k^{0th},
\end{align*}
where we used the following notations
\begin{align*}
&\tilde{\tilde{{\underline{Y}}}}_k^{0th}:=\frac{\nabla_\bx b(\bx_{k,0})\cdot\tilde{\underline{X}}_k^{1st}}{\eps b_k}J\bv_{k,0}+
\frac{1}{\eps}\Pi\left(\frac{b(X_k^{1st})}{b_k}-1\right)J\tilde{\underline{Y}}_k^{0th},\\
&\tilde{\underline{X}}_k^{1st}:=\frac{1}{b_k}\Pi (\fe^{\tau J}\br_k^{1st}), \;\; \tilde{\underline{Y}}_k^{0th}:=\Pi \frac{1}{\eps}\left( \frac{b({X}_k^{1st})}{b_k} - 1\right) J\bv_{k,0},
\end{align*}
with ${X}_k^{1st}$ and $\br_k^{1st}$ given by \eqref{Xfirst} and \eqref{rk1st}. Note that
$\partial_s^2 {\underline{Y}}_k(0)=\tilde{\tilde{\underline{Y}}}_k^{0th}+O(\eps)$,
$\partial_s {\underline{X}}_k(0)=\tilde{\underline{X}}_k^{1st}+O(\eps^2)$ and
$\partial_s \underline{Y}_k(0) = \tilde{\underline{Y}}_k^{0th}+O(\eps)$.
Similarly, deriving (\ref{dthr def b}) with respect to $s$ and assuming $\partial_s^3\br_k=O(\eps)$, we have
\begin{align*}
&\partial_s^2\br_k(0,\tau)=\tilde{\tilde{\br}}_k^{1st}(\tau)+O(\eps^2), \quad
\mbox{with}\quad \\
&\tilde{\tilde{\br}}_k^{1st}(\tau):=\frac{2}{b_k} A\nabla_\bx b(\bx_{k,0})\cdot\left(\tilde{\underline{X}}_k^{1st}+\tilde\bh_k^{1st}\right) J\tilde{\underline{Y}}_k^{0th}+A\left(\frac{b(X_k^{1st})}{b_k}-1\right)J\tilde{\tilde{\underline{Y}}}_k^{0th}\\
&\quad\quad\quad+\frac{1}{b_k} A\nabla_\bx b(\bx_{k,0})\cdot\left(\tilde{\tilde{\underline{X}}}_k^{1st}+\tilde{\tilde{\bh}}_k^{1st}\right) J\bv_{k,0},
\end{align*}
where we defined
\begin{align*}
&\tilde{{\underline{\bh}}}_k^{1st}(\tau):=\frac{\eps}{b_k}A \fe^{\tau J} \tilde{\underline{Y}}_k^{0th}, \;\;\;\;\;\; \tilde{\tilde{\underline{X}}}_k^{1st}:=\frac{1}{b_k}\Pi\fe^{\tau J}\tilde\br_k^{1st}, \\ %\quad \mbox{for}\quad
&\tilde{\br}_k^{1st} := A\left( \frac{b(X_k^{1st}}{b_k}-1\right) J\tilde{\underline{Y}}_k^{0th} + \frac{1}{b_k}A\nabla_\bx b(\bx_{k,0})\cdot \tilde{\bh}_k^{1st} J\bv_{k,0} + \frac{\eps}{b_k^2}A\fe^{-\tau J}\partial_t\bE(0, \bx_{k,0}).
\end{align*}
Note that $\ddot{\underline{X}}_k(0)=\tilde{\tilde{\underline{X}}}_k^{1st}+O(\eps^2)$.
Then we update to get
\begin{align*}
&\tilde{\underline{Y}}_k^{1st}=\Pi\left[ \frac{1}{\eps}\left(\frac{b(X_k^{2nd})}{b_k}-1\right)JY_k^{1st}+
\fe^{-J\tau}\frac{\bE^\eps(0, X_k^{1st})}{b_k} \right],\\
&\tilde{\underline{X}}_k^{2nd}=\frac{1}{b_k}\Pi (\fe^{\tau J}\br_k^{2nd}),
\end{align*}
where $X_k^{2nd}$ is given by \eqref{Xsecond} and $\br_k^{2nd}$ by \eqref{r2nd}. Moreover, we define
\begin{align*}
\tilde\bh_k^{2nd}(\tau):=& \frac{\eps}{b_k}A \fe^{\tau J}\left(\tilde{\underline{Y}}_k^{1st}+\tilde\br_k^{1st}\right)-\eps L^{-1}\tilde{\tilde{\bh}}_k^{1st},\\
\tilde\br_k^{2nd}(\tau):=&A\left(\frac{b(X_k^{2nd})}{b_k}-1\right)J\left(\tilde{\underline{Y}}_k^{1st}+\tilde\br_k^{1st}\right)\\
&+\frac{1}{b_k}A\nabla_\bx b(X_{k}^{1st})\cdot\left(\tilde{\underline{X}}_k^{2nd}+\tilde\bh_k^{2nd}\right)JY_{k}^{1st}\\
&+\frac{\eps}{b_k}A\fe^{-J\tau}\left(\frac{\partial_t\bE^\eps(0,X_k^{1st})}{b_k}+\nabla_\bx\bE^\eps(0,X_k^{1st})
(\tilde{\underline{X}}_k^{1st}+\tilde\bh_k^{1st})\right)\\
&-\eps L^{-1}\tilde{\tilde{\br}}_k^{1st}.
\end{align*}
Note that $\partial_s{\underline{X}}_k(0)=\tilde{\underline{X}}_k^{2nd}+O(\eps^3),\quad
\partial_s\bh_k(0,\tau)=\tilde\bh_k^{2nd}(\tau)+O(\eps^3)$ and
$\partial_s{\underline{Y}}_k(0)=\tilde{\underline{Y}}_k^{1st}+O(\eps^2),\quad
\partial_s\br_k(0,\tau)=\tilde\br_k^{2nd}(\tau)+O(\eps^3)$.

Eventually from (\ref{hr def a}), we define iteratively
$$
\bh_k^{3rd}(\tau):=\frac{\eps}{b_k}A\fe^{\tau J}Y_k^{2nd}-\eps L^{-1}\tilde\bh_k^{2nd},
$$
where $Y_k^{2nd}$ is given by \eqref{Ysecond}, so that the third order initial data for the first equation is
\begin{equation}\label{Xthird}
X_k^{3rd}(\tau):=\bx_{k,0}+\bh_k^{3rd}(\tau)-\bh_k^{3rd}(0).
\end{equation}
We can than define
$$
\br_k^{3rd}(\tau):=A\left(\frac{b(X_k^{3rd})}{b_k}-1\right)\fe^{\tau J}Y_k^{2nd}+\frac{\eps}{b_k}
A\fe^{-J\tau}\bE^\eps(0,X_k^{2nd})-\eps L^{-1}\tilde\br_k^{2nd},
$$
so that the third order initial data for the second equation is
\begin{equation}\label{Ythird}
Y_k^{3rd}(\tau):=\bv_{k,0}+\br_k^{3rd}(\tau)-\br_k^{3rd}(0).
\end{equation}
Note that $\bh_k(0,\tau)=\bh_k^{3rd}(\tau)+O(\eps^4)$ and $\br_k(0,\tau)=\br_k^{3rd}(\tau)+O(\eps^4)$.

\begin{proposition}\label{prop:3rd}
Assume (\ref{Lips}) and  $E^\eps(t,\bx)\in {\mathcal C}^4([0,T]\times\bR^2)$ and $b(\bx)\in {\mathcal C}^4(\bR^2)$.
With the third order initial data $X_k(0,\tau)=X_k^{3rd}(\tau)$ given by  \eqref{Xthird}
and $Y_k(0,\tau)=Y_k^{2nd}(\tau)$ given by  \eqref{Ysecond},
the solution of the two-scale system (\ref{charact 2scale}) satisfies
$$
\frac{1}{\eps}\|\partial_s^\ell X_k\|_{L_s^\infty( W^{1,\infty}_\tau)}+\|\partial_s^\ell Y_k\|_{L_s^\infty( W^{1,\infty}_\tau)}\leq C_0, \quad \ell=1,2,3,
$$
for some constant $C_0>0$  independent of $\eps$.
\end{proposition}
\begin{proof}
The proof is a recursive process of the known results and it is very similar to that of Proposition \ref{prop:1st}. We omit the details here for brevity.
\end{proof}

\section{Limit model} %{Appendix}
\label{appendix2}
In this appendix, we derive limit mode at the characteristics level. %\textcolor{red}{The asymptotic model is not really used in this section. It would be better to put it before in the paper. }
From \eqref{eq:X}-\eqref{eq:Y}, we derive the averaged model by considering $\eps<\!\!<1$.  Indeed, from the equations on
$\bh_k$ and $\br_k$, we get
\begin{eqnarray*}
\bh_k &=& -\frac{\eps}{b_k}J\fe^{\tau J}\underline{Y}_k + O(\eps^2), \nonumber\\
\br_k &=& -\frac{\eps}{b_k^2} (\fe^{\tau J}\underline{Y}_k\cdot \nabla_\bx b(\underline{X}_k) J\underline{Y}_k + \frac{\eps}{b_k}J\fe^{-\tau J} \bE^\eps(s/b_k, \underline{X}_k) + O(\eps^2).
\end{eqnarray*}
Then, injecting in the macro equations on $\underline{X}_k$ and $\underline{Y}_k$, we obtain
\begin{eqnarray*}
\partial_s \underline{X}_k &=& -\frac{\eps}{b_k^3} \Pi\left[\fe^{\tau J} (\fe^{\tau J}\underline{Y}_k\cdot \nabla_\bx b(\underline{X}_k))  J\underline{Y}_k\right] + \frac{\eps}{b_k^2}  \Pi\left[  \fe^{\tau J} J \fe^{-\tau J} \bE^\eps(s/b_k, \underline{X}_k)  \right]+ O(\eps^2), \nonumber\\
\partial_s \underline{Y}_k &=& \frac{1}{\eps}\Pi\left[  B_k(\underline{X}_k+\bh_k) J\underline{Y}_k\right] + \Pi \left[ \fe^{-\tau J}\frac{1}{b_k} \bE^\eps(s/b_k, \underline{X}_k)  \right] + O(\eps),
\end{eqnarray*}
where we defined $B_k(X)=b(X)/b(\bx_{k,0}) -1$. After some computations, it comes
\begin{eqnarray*}
\partial_s \underline{X}_k &=& -\frac{\eps}{2b_k^3} J \nabla_\bx b(\underline{X}_k) \;  |\underline{Y}_k|^2 +  \frac{\eps}{b_k^2}  J\bE^\eps(s/b_k, \underline{X}_k) + O(\eps^2), \nonumber\\
\partial_s \underline{Y}_k &=&  \frac{1}{\eps} B_k(\underline{X}_k) J\underline{Y}_k + O(\eps), \nonumber\\
&=& \frac{1}{\eps b_k}(b(\underline{X}_k) - b_k) J\underline{Y}_k + O(\eps).
\end{eqnarray*}
Under the fact that $b(\bx_k(t))=b_k+O(\eps)$, the limit model above is consistent with the one derived in
\cite{degond}.

\section*{Acknowledgements}
This work is supported by the French ANR project MOONRISE ANR-14-CE23-0007-01.
N. Crouseilles and M. Lemou are supported by the Enabling Research EUROFusion project
CfP-WP14-ER-01/IPP-03.
X. Zhao is supported by the IPL FRATRES.

\bibliographystyle{model1-num-names}

\end{document}